\documentclass[final,onefignum,onetabnum,letterpaper]{siamonline190516}

\usepackage{lineno}

\usepackage{lipsum}
\usepackage{amsfonts}
\usepackage{epstopdf}
\ifpdf
  \DeclareGraphicsExtensions{.eps,.pdf,.png,.jpg}
\else
  \DeclareGraphicsExtensions{.eps}
\fi

\usepackage{datetime}
\newdateformat{monthyeardate}{%
  \monthname[\THEMONTH] \THEDAY, \THEYEAR}

\usepackage{academicons}
\usepackage{xcolor}
\renewcommand{\orcid}[1]{\href{https://orcid.org/#1}{\textcolor[HTML]{A6CE39}{orcid.org/#1}}}

\usepackage{amsmath}
\allowdisplaybreaks
\usepackage{amssymb}
\usepackage{commath}
\usepackage{mathtools}
\usepackage{bbm}

\usepackage{color}
\usepackage{graphicx}
\usepackage[small]{caption}
\usepackage{subcaption}

\usepackage{relsize}
\usepackage{adjustbox}
\usepackage{algorithm}
\usepackage[noend]{algpseudocode}
\usepackage{booktabs}
\usepackage{tikz}

\usepackage{verbatim}

\usepackage{enumitem}
\setlist[enumerate]{leftmargin=.5in}
\setlist[itemize]{leftmargin=.5in}

\newsiamremark{remark}{Remark}
\newsiamremark{hypothesis}{Hypothesis}
\crefname{hypothesis}{Hypothesis}{Hypotheses}
\newsiamthm{claim}{Claim}

\headers{Stable high-order CFs for experimental data}{Jan Glaubitz}

\title{Stable high-order cubature formulas for experimental data
\thanks{
\monthyeardate\today 
\corresponding{Jan Glaubitz (\email{Jan.Glaubitz@Dartmouth.edu}, \orcid{0000-0002-3434-5563})} 
\funding{This work was partially supported by AFOSR \#F9550-18-1-0316 and ONR \#N00014-20-1-2595.}
}}

\author{Jan Glaubitz\thanks{Department of Mathematics, Dartmouth College, Hanover, NH 03755, USA}
}

\usepackage{amsopn}

\DeclareMathOperator{\diag}{diag}
\DeclareMathOperator*{\argmin}{arg\,min} 
\newcommand{\scp}[2]{\left\langle{#1,\, #2}\right\rangle} 
\newcommand{\intd}{\, \mathrm{d}}
\newcommand{\N}{\mathbb{N}}

\newcommand{\R}{\mathbb{R}}

\ifpdf
\hypersetup{
  pdftitle={Glaubitz2021stableCFs_experimental},
  pdfauthor={Jan Glaubitz}
}
\fi

\begin{document}

\maketitle

\begin{abstract}
	In many applications, it is impractical---if not even impossible---to obtain data to fit a known cubature formula (CF). 
Instead, experimental data is often acquired at equidistant or even scattered locations. 
In this work, stable (in the sense of nonnegative only cubature weights) high-order CFs are developed for this purpose. 
These are based on the approach to allow the number of data points $N$ to be larger than the number of basis functions $K$ which are integrated exactly by the CF. 
This yields an $(N-K)$-dimensional affine linear subspace from which cubature weights are selected that minimize certain norms corresponding to stability of the CF. 
In the process, two novel classes of stable high-order CFs are proposed and carefully investigated. 
\end{abstract}



\section{Introduction} 
\label{sec:introduction} 

Numerical integration  is an omnipresent problem in mathematics and myriad other scientific areas. 
In fact, measuring areas and volumes dates back at least to the ancient Babylonians and Egyptians \cite{boyer2011history}. 
The present work is concerned with the determination---approximately or exactly---of integrals over regions in two or more dimensions. 
This problem was first studied systematically by Maxwell in 1877 \cite{maxwell1877approximate} and is today known as \emph{cubature}.

Let $\Omega \subset \R^q$ be a bounded domain with positive volume, $|\Omega| > 0$. 
Given are $N$ distinct data pairs $\{ (\mathbf{x}_n,f_n) \}_{n=1}^N \subset \Omega \times \R$ with $f: \Omega \to \R$ and $f_n := f(\mathbf{x}_n)$.
The aim is to approximate the weighted integral 
\begin{equation}\label{eq:I}
  I[f] := \int_\Omega f(\boldsymbol{x}) \omega(\boldsymbol{x}) \intd \boldsymbol{x}
\end{equation}
with nonnegative weight function $\omega$ (assumed to be integrable) by an \emph{$N$-point CF} 
\begin{equation}\label{eq:CR}
  C_N[f] = \sum_{n=1}^N w_n f(\mathbf{x}_n).
\end{equation} 
Here, the distinct points $\{ \mathbf{x}_n \}_{n=1}^N$ are called \emph{data points} and the $\{ w_n \}_{n=1}^N$ are called \emph{cubature weights}.
A sequence of CFs $(C_N)_{N \in \N}$ is called a \emph{cubature rule (CR)}.

In one dimension ($q=1$) the construction of CFs---usually referred to as \emph{quadrature formulas (QFs)}---is dominated by the idea of interpolating the data (by a polynomial) and exactly integrating the polynomial then. 
Equidistant data points lead to Newton--Cotes, Chebyshev points to Clenshaw--Curtis and roots of Jacobi polynomials to Gauss--Jacobi rules \cite{clenshaw1960method,gautschi1997numerical,krylov2006approximate,davis2007methods,trefethen2008gauss,brass2011quadrature,glaubitz2020shock}.
These QFs can also be used to construct CFs for certain higher-dimensional domains ($q > 1$) and weight functions, resulting in (generalized) Cartesian product rules \cite{davis2007methods}. 

Other approaches to construct CFs include minimal CFs, that are designed to be exact for (algebraic or trigonometric) polynomials of high degree using as few data points as possible; minimum-norm CFs, which are based on the idea to minimize the norm of the cubature error considered as a linear functional; and number-theoretical CFs, essentially derived from the ideas of Diophantine approximation and equidistribution modulo $1$. 
We refer to a rich body of literature 
\cite{haber1970numerical,stroud1971approximate,cools1997constructing,krommer1998computational,cools2003encyclopaedia,davis2007methods}
and references therein. 
Of course, this list is by no means exhaustive.
Another important class of CFs is Monte Carlo (MC) and quasi-Monte Carlo (QMC) methods 
\cite{metropolis1949monte,caflisch1998monte,dick2013high}. 
In these, the data points are random samples (uniformly distributed over $\Omega$) or correspond to partially or fully deterministic low-discrepancy sequences.\footnote{These are supposed to enhance uniformity of the data points.}
All of these methods have their own advantages and disadvantages. 
Yet, it should be stressed that most of the above formulas require a specific distribution of the data points.

There are also several optimization strategies for computing high-order CFs \cite{taylor2000algorithm,taylor2007cardinal,ryu2015extensions,jakeman2018generation,keshavarzzadeh2018numerical}. 
These strategies typically rely on the ability to have full flexibility in placing data points inside the domain $\Omega$.

In many applications, however, it is impractical---if not even impossible---to obtain data to fit a known CF \cite{wilson1970discrete,hoge1980oil,reeger2020approximate}. 
For instance, experimental measurements are often performed at equidistant or even scattered locations. 
Furthermore, in some applications, numerical integration is a follow-up to some other task (e.\,g.\ numerically solving PDEs). 
In such a situation, it is not reasonable to require data points that are specific to a certain CF.
The present work is therefore concerned with the construction of stable and high-order CF for general sets of data points.

At least in one dimension ($q=1$), some first steps towards such a goal have already been discussed in 1970 by Wilson. 
In \cite{wilson1970discrete}, he proposed to construct stable high-order QFs by allowing the number of data points $N$ to be larger than the desired degree of exactness (DoE) $d$. 
This yields an underdetermined least squares (LS) problem. 
While Wilson referred to the resulting QFs as \emph{nearest point QFs}, in later works \cite{huybrechs2009stable,glaubitz2020shock,glaubitz2020stable}, the name \emph{LS-QFs} was coined. 
Focusing on the constant weight function $\omega \equiv 1$ and equidistant data points, it was shown in \cite{wilson1970necessary} that stability of LS-QFs (in the sense of nonnegative only  cubature weights) can be ensured essentially by choosing $N = (d+1)^2$. 
In the process, Wilson utilized a beautiful connection between stable QFs and discrete orthogonal polynomials (DOPs). 
The connection between (Gaussian) QFs and continuous orthogonal polynomials, e.\,g., the Legendre polynomials, is well known. 
In contrast, the interplay between QFs and DOPs was---to the best of the author's knowledge---only developed further nearly 40 years later in  \cite{huybrechs2009stable,glaubitz2020shock,glaubitz2020stableQRs}. 
In \cite{huybrechs2009stable} and \cite[Chapter 4]{glaubitz2020shock} the original works of Wilson \cite{wilson1970discrete,wilson1970necessary} were revisited and it was shown that stability of these rules also holds for more general positive weight functions and scattered (not necessarily equidistant) data points. 
An application of these QFs to numerical PDEs was explored in \cite{glaubitz2020stable}.
Furthermore, in \cite{glaubitz2020stableQRs}, the idea of LS-QFs was utilized to construct stable QFs even for general weight functions (potentially having mixed signs).\footnote{It should be pointed out that in this case stability holds in a weaker sense than compared to positive weight functions. In particular, it was argued in \cite{glaubitz2020stableQRs} that one should distinguish between stability and sign-consistency of QFs for general weight functions.}
Yet, to this date, no extension of these QFs to higher dimensions ($q>1$) has been discussed. 

The present work aims to fill this gap in the literature and to construct stable high-order CFs for experimental data in two and more dimensions. 
These CFs assume a fixed set of data points as input and then strive to provide a stable numerical integration procedure. 
Moreover, the DoE of this numerical integration procedure is, in a certain sense, as high as possible. 
In particular, the proposed CFs satisfy a list of properties that are universally considered to be highly desirable \cite{davis1967construction,haber1970numerical,cools2001cubature,krylov2006approximate,brass2011quadrature,van2020adaptive}: 

\begin{enumerate} 
	\item[(P1)] 
	The data points lie inside the integration domain $\Omega$. 
	
	\item[(P2)]  
	The cubature weights are all nonnegative.
	
	\item[(P3)] 
	The CF has a high (or even optimal) DoE for fixed data points.
	
\end{enumerate}

In what follows, two different methods to achieve this goal for fairly general bounded domains $\Omega \subset \R^q$ are proposed. 
The first method transfers the idea of LS-QFs to higher dimensions. 
This approach is discussed in \S \ref{sub:LS-CFs} and the resulting CFs will be called \emph{LS-CFs}. 
They are essentially based on selecting a weighted LS solution from the solution space of the underdetermined linear system of exactness conditions \eqref{eq:ex-cond-system}. 
The second method, on the other hand, is based on finding a least-absolute-values ($\ell^1$) solution. 
This method is presented in \S \ref{sub:l1-CFs} and the resulting CFs will be referred to as \emph{$\ell^1$-CFs}. 
These can provide higher DoE than LS-CFs (for the same set of data points). 
At the same time, however, they can be expected to be computationally more expensive. 

It should be pointed out that there are strong connections between the present manuscript and some other recent works \cite{migliorati2018stable,van2020generating,van2020adaptive}. 
In \cite{migliorati2018stable} high-order CFs were constructed for independent random data points. 
These randomized CFs were moreover shown to have positive weights with a high probability if the number of (random) data points is sufficiently larger than the DoE. 
The results presented in \S \ref{sec:theoretical} of the present manuscript are of a similar flavor. 
We do not restrict ourselves to independent random data points, however, and the proposed formulas are always ensured to have nonnegative weights. 
In \cite{van2020generating,van2020adaptive}, the authors are concerned with the construction (and application to Bayesian prediction) of nested positive CFs. 
Yet, it should be noted that in these works exactness of the CFs is not understood w.\,r.\,t.\ to the continuous integral as in \eqref{eq:I} but a discrete approximation of the form ${I^{(K)}[f] = \frac{1}{K+1} \sum_{k=0}^K f(\mathbf{y}_k)}$ with $K \gg N$. 
If the samples $\mathbf{y}_k$ are drawn from an appropriate distribution, $I^{K}[f]$ approximates $I[f]$. 
Still, it is assumed that a large number of samples can be determined fast and efficiently. 
In fact, it might be argued that the method discussed in \cite{van2020generating,van2020adaptive} is closer to the technique of subsampling \cite{wilson1969general,seshadri2017effectively,van2017non,piazzon2017caratheodory,piazzon2017caratheodoryLS,bos2019catchdes} than the construction of CFs as discussed here.  
A more detailed comparison of the proposed method to construct stable high-order CFs to other existing methods can be found in \S \ref{sec:connection}. 

Finally, it should be stressed that if the reader is only interested in evaluating some integral \eqref{eq:I} and is \emph{not} constrained to specific data points---the function values can be obtained at any desired location---there are certainly other CFs available for this purpose in most cases. 
The stable high-order CFs proposed in the present work, on the other hand, will find their greatest utility when it is difficult---or even impossible---to obtain data at locations required for a particular CF. 
The Matlab code corresponding to the methods developed in this work can be found at \cite{glaubitz2020github}. 

The rest of this work is organized as follows. 
\S \ref{sec:what} provides some preliminaries of stability and exactness of CFs. 
The two new classes of stable and high-order CFs are presented in \S \ref{sec:methods}. 
In \S \ref{sec:theoretical}, it is shown that both CFs are nonnegative if a sufficiently large number of (appropriately distributed) data points is used. 
\S \ref{sec:connection} addresses the connection of the two presented CFs to several existing ones. 
Some computational details are offered in \S \ref{sec:constr}. 
In \S \ref{sec:tests}, we demonstrate the performance of the proposed CFs for a series of different numerical tests. 
Finally, concluding thoughts and an outlook to future research is given in \S \ref{sec:summary}. 
\section{What Do We Want? Stability and Exactness} 
\label{sec:what}

In many applications, it is not possible to get exact measurements $\{f_n\}_{n=1}^N$. 
Instead, we are left with \emph{experimental measurements} $\{f_n^\varepsilon\}_{n=1}^N$ with an inherent \emph{data error} (or \emph{measurement error}): 
\begin{equation}
  \norm{ \mathbf{f} - \mathbf{f}^\varepsilon }_\infty \leq \varepsilon
\end{equation}
Here, $\mathbf{f}$ and $\mathbf{f}^\varepsilon$ respectively denote the vectors $(f_1,\dots,f_N)^T$ and 
$(f_1^\varepsilon,\dots,f_N^\varepsilon)^T$. 
Such errors may be round-off or truncation errors (if $f$ is defined analytically), or errors of measurement or experiment when $f$ is determined by a physical process. 
In this case, we do not only have to ensure that the CF $C_N$ is a good approximation of the exact integral $I$. 
In addition, the growth of the data error should be bounded, and as small as possible. 
This can be observed from the following: 
If we estimate the error between the exact integral of $f$ and the result of a CF applied to $f^\varepsilon$, we observe that 
\begin{equation}
  \left| I[f] - C_N[f^\varepsilon] \right| 
    \leq \left| I[f] - C_N[f] \right| + \left| C_N[f] - C_N[f^\varepsilon] \right|.
\end{equation}
To the first term, we refer to as the \emph{approximation error}. 
For the second term, we note that 
\begin{equation}\label{eq:stab-error}
  \left| C_N[f] - C_N[f^\varepsilon] \right| 
  	\leq \varepsilon \kappa(\mathbf{w}), \quad 
	\kappa(\mathbf{w}) := \sum_{n=1}^N |w_n|.
\end{equation} 
Thus, the second term is bounded by the data error $\varepsilon$ times an amplification factor $\kappa(\mathbf{w})$ which depends on the cubature weights.\footnote{Note that \eqref{eq:stab-error} relates to a special set of parameters in the H\"older inequality. 
Other choices are possible and discussed in Remark \ref{rem:other-stab-measures}.}
This amplification factor is strongly connected to the stability of a CF and is minimal if the CF has nonnegative only cubature weights.

\subsection{Stability} 
\label{sub:stability}

Given are two functions $f,f^\varepsilon:\Omega \to \R$ with $|f(\boldsymbol{x}) - f^\varepsilon(\boldsymbol{x})| \leq \varepsilon$ for all $\boldsymbol{x} \in 
\Omega$. 
Then, we have 
\begin{equation}\label{eq:error-int}
  \left| I[f] - I[f^\varepsilon] \right|
    \leq I[1] \varepsilon.
\end{equation}
This means that the growth of errors in the input (data errors) are bounded by the factor 
$I[1]$.
For an $N$-point CF, on the other hand, we have \eqref{eq:stab-error}.
Here, the growth of data errors is bounded by the \emph{stability value} $\kappa(\mathbf{w})$, which is a usual stability measure for CFs. 
Note that 
\begin{equation}\label{eq:lower-bound-kappa}
	\kappa(\mathbf{w}) \geq \sum_{n=1}^N w_n = I[1] 
\end{equation} 
if the corresponding CF $C_N$ is exact for constants ($C_N[1] = I[1]$). 
Hence, to minimize $\kappa(\mathbf{w})$ and therefore the amplification of errors, a CF with nonnegative weights is desired. 
To such a CF we refer to as being \emph{nonnegative} (sometimes also \emph{perfectly stable} \cite{glaubitz2020shock,glaubitz2020stableQRs}). 
It should be stressed that the sharp lower bound \eqref{eq:lower-bound-kappa} is exclusive for nonnegative weight functions $\omega$ and CFs that are exact for constants. 
In fact, if $C_N[1] \neq I[1]$ it is possible that $\kappa(\mathbf{w}) < I[1]$. 
At the same time, for general (not necessarily nonnegative) weight functions, nonnegative weights might not ensure stability \cite{glaubitz2020stableQRs}.

\subsection{Exactness}
\label{sub:exactness}

Another important design criterion for CFs, which is strongly connected to their accuracy, is exactness.

\begin{definition}[The DoE]\label{def:DOE}
  A CF $C_N$ on $\Omega \subset \R^q$ is said to have \emph{(polynomial) DoE} $d$ if the 
\emph{exactness condition}
  \begin{equation}\label{eq:ex-cond}
    C_N[f] = I[f] \quad \forall f \in \mathbb{P}_d(\R^q)
  \end{equation}
  holds.\footnote{Many authors add that the CF should be inexact for a polynomial of degree $d+1$. This results in uniqueness for the DoE: It is the largest number $d \in \N_0$ such that \eqref{eq:ex-cond} holds. 
	Following Definition \ref{def:DOE}, on the other hand, a CF which has DoE $d$ also has DoE $\tilde{d}$ for $\tilde{d} \leq d$. 
	Yet, in the present discussion, this will not yield any problems and we therefore proceed to use the slightly simpler Definition \ref{def:DOE}.}
\end{definition}

Here, $\mathbb{P}_d(\R^q)$ denotes the vector space of all (algebraic) polynomials of  degree at most $d$.

\begin{remark}
	Note that an (algebraic) polynomial in $q$ variables ${\boldsymbol{x} = (x_1,\dots,x_q)}$ is a finite linear combination of monomials of the form ${\boldsymbol{x}^{\boldsymbol{\alpha}} = x_1^{\alpha_1} \dots x_q^{\alpha_q}}$ with the \emph{degree} (also known as the \emph{total degree}) defined as $|\boldsymbol{\alpha}| = \sum_{i=1}^q \alpha_i$.
	Then, the \emph{(total) degree of a polynomial} is the maximum of the degrees of its monomials. 
\end{remark}

\begin{remark} 
	While we only focus on the total degree in this work, other choices would be possible as well. 
	These include the \emph{absolute degree} ($|\boldsymbol{\alpha}|_{\infty} = \max_{i=1,\dots,q} \alpha_i$) and the \emph{Euclidean degree} ($|\boldsymbol{\alpha}|_{2}^2 = \sum_{i=1}^q \alpha_i^2$). 
	In fact, in some recent works \cite{trefethen2017cubature,trefethen2017multivariate,trefethen2021exactness}, it was pointed out that for angle-independent resolution in the hypercube it would be necessary to base CFs on the Euclidean degree (instead of the often used total degree). 
	Future work might provide a numerical comparison of different degrees in the context of the CFs discussed here. 
\end{remark}

Exactness ensured that polynomials up to a certain degree are treated exactly by the CF.
Let $\{ p_k \}_{k=1}^K$ be a basis of $\mathbb{P}_d(\R^q)$, where 
$K := \dim \mathbb{P}_d(\R^q) = \binom{d+q}{q}$.
Then, the exactness condition \eqref{eq:ex-cond} yields a linear system: 
\begin{equation}\label{eq:ex-cond-system}
  \underbrace{
  \begin{pmatrix}
    p_1(\mathbf{x}_1) & \dots & p_1(\mathbf{x}_N) \\ 
    \vdots & & \vdots \\ 
    p_K(\mathbf{x}_1) & \dots & p_K(\mathbf{x}_N)
  \end{pmatrix}}_{=: P}
  \underbrace{
  \begin{pmatrix} 
    w_1 \\ \vdots \\ w_N 
  \end{pmatrix}}_{=: \mathbf{w}} 
  = 
  \underbrace{
  \begin{pmatrix} 
    m_1 \\ \vdots \\ m_K 
  \end{pmatrix}}_{=: \mathbf{m}}
\end{equation}
Here, $m_k := I[p_k]$ denotes the \emph{$k$-th moment}. 
We can immediately note that a CF $C_N$ has DoE $d$ if and only if its weights 
solve \eqref{eq:ex-cond-system}. 
Let us assume $K < N$. 
Then, \eqref{eq:ex-cond-system} becomes underdetermined. 
These are well-known to either have no solution or infinitely 
many. 
The existence of infinitely many solutions is ensured if the set of data points is $\mathbb{P}_d(\R^q)$-unisolvent. 

\begin{definition}[Unisolvent Point Sets] 
  A set of points $X = \{\mathbf{x}_n\}_{n=1}^N \subset \R^q$ is called \emph{$\mathbb{P}_d(\R^q)$-unisolvent} if 
  \begin{equation}
    p(\mathbf{x}_n) = 0, \ n=1,\dots,N \implies p \equiv 0
  \end{equation}
  for all $p \in \mathbb{P}_d(\R^q)$. 
  That is, the only polynomial of degree $\leq d$ that interpolates zero data 
is the zero polynomial. 
\end{definition}

In this case, it is easy to note the following lemma. 

\begin{lemma}\label{lem:solution-space}
  Let $K = \dim \mathbb{P}_d(\R^q) < N$ and let $X = \{\mathbf{x}_n\}_{n=1}^N$ be 
$\mathbb{P}_d(\R^q)$-unisolvent. 
  Then, the SLE \eqref{eq:ex-cond-system} is underdetermined and induces an $(N-K)$-dimensional 
affine linear subspace of solutions 
  \begin{equation}\label{eq:sol-space}
    W := \left\{ \mathbf{w} \in \R^N \mid P\mathbf{w}=\mathbf{m} \right\}.
  \end{equation}
\end{lemma}

\begin{proof}
  Note that $W$ can be rewritten as 
  \begin{equation}
    W = \mathbf{w}_s + W_0, \quad 
    W_0 := \left\{ \mathbf{w}_0 \in \R^N \mid P\mathbf{w}_0 = 0 \right\}.
  \end{equation}
  Here, $\mathbf{w}_s$ is a specific solution of $P\mathbf{w}=\mathbf{m}$ and $W_0$ is the linear solution space of 
the homogeneous problem. 
	$X$ being $\mathbb{P}_d(\R^q)$-unisolvent results in the $K$ rows of $P$ being linearly independent. 
	Thus, $P$ has full rank and $\mathbf{w}_s$ is ensured to exist. 
 	Finally, the rank-nullity theorem \cite[Theorem 2.8]{roman2005advanced} yields $\dim W_0 = N-K$ and therefore the assertion.
\end{proof}

\section{Proposed Methods} 
\label{sec:methods}

Let $N$ data pairs $(\mathbf{x}_n,f_n^\varepsilon)_{n=1}^N \subset \Omega \times \R$ be given. We now aim to construct nonnegative CFs with a high DoE. 
This is done by first and foremost ensuring that all weights are nonnegative and maximizing the DoE only afterward. 
This basic idea could be summarized as ``\emph{stability before exactness}''. 
In what follows, we present two methods to realize this strategy. 
This results in---to the best of the author's knowledge---novel stable high-order CFs for scattered data points. 
Both procedures start with the following two steps: 

\begin{enumerate}[label={(S\arabic*)}]
  \item\label{item:1}
  Determine the largest $d \in \N$ such that $X = \{\mathbf{x}_n\}_{n=1}^N$ is $\mathbb{P}_d(\R^q)$-unisolvent. 
  
  \item\label{item:2} 
  Formulate the linear system \eqref{eq:ex-cond-system} for this $d$. 

\end{enumerate}

Recall that \eqref{eq:ex-cond-system} becomes underdetermined for $K = \dim \mathbb{P}_d(\R^q) < N$. 
In this case, Lemma \ref{lem:solution-space} ensures that \eqref{eq:ex-cond-system} induces an $(N-K)$-dimensional affine linear subspace of solutions $W$. 
Every element $\mathbf{w} \in W$ results in a CF with DoE $d$. 
The two methods below, aim to determine a vector of weights $\mathbf{w} \in W$ that also yields favorable stability properties. 
That is, $\kappa(\mathbf{w})$ should be as small as possible. 
Note that $\kappa(\mathbf{w}) \geq I[1]$ with equality if and only if all weights are nonnegative.

\subsection{$\ell^1$ Cubature Formulas} 
\label{sub:l1-CFs}

Following the main goal---to ensure stability---it seems natural to determine an element $\mathbf{w}^* \in W$ such that 
\begin{equation} 
	\kappa(\mathbf{w}^*) \leq \kappa(\mathbf{w}) \quad \forall \mathbf{w} \in W.
\end{equation}
Since $\kappa(\mathbf{w}) = \| \mathbf{w} \|_1$, the resulting optimization problem corresponds to $\ell^1$-minimization. 
Thus, the element $\mathbf{w}^*$ is called an \emph{$\ell^1$-solution} from $W \subset \R^N$, which  is denoted as 
\begin{equation}\label{eq:l1-sol}
  \mathbf{w}^{\ell^1} = \argmin_{\mathbf{w} \in W} \ \norm{ \mathbf{w} }_1.
\end{equation} 
Constraint optimization problems of this type are known as basis pursuit problems \cite{chen2001atomic,boyd2004convex}. 
They play a central role in modern statistical signal processing, particularly the theory of compressed sensing; see \cite{candes2006stable,candes2006robust,donoho2006compressed,bruckstein2009sparse,foucart2017mathematical} and references therein.\footnote{It is also worth noting that basis pursuit is closely connected with linear programming \cite{bloomfield1983least,dantzig1998linear,dantzig2006linear,gill1991numerical,vanderbei2020linear}.} 
Once the cubature weights $\mathbf{w}^{\ell^1}$ have been computed, we check whether or not the resulting CF is nonnegative.
That is, if $\mathbf{w}^{\ell^1} \geq 0$ holds. 
If this is the case, the desired integral $I[f]$ is approximated by the following \emph{$\ell^1$-CF}:
\begin{equation}\label{eq:l1-CF}
	C^{\ell^1}_N[f^\varepsilon] := \sum_{n=1}^N w_n^{\ell^1} f^{\varepsilon}_n
\end{equation} 
By construction, this CF is perfectly stable, while having DoE $d$.
On the other hand, if the resulting CF is \emph{not} nonnegative, one decreases the DoE by one ($d \mapsto d-1$) and returns to \ref{item:2}.
The whole procedure is summarized in Algorithm \ref{algo:l1-CF}.

\begin{algorithm}
\caption{Construction of $\ell^1$-CFs}
\label{algo:l1-CF}
\begin{algorithmic}[1]
    \State{Determine the maximal $d \in \N$ such that $X$ is $\mathbb{P}_d(\R^q)$-unisolvent} 
    \State{Formulate the linear system $P \mathbf{w} = \mathbf{m}$ for DoE $d$}
    \State{Compute an $\ell^1$-solution $\mathbf{w}^{*} = \argmin_{\mathbf{w} \in W} \ \norm{ \mathbf{w} }_1$}
    \While{$\mathbf{w}^{\ell^1} \not\geq 0$} 
      \State{Reduce the DoE: $d = d-1$} 
      \State{Formulate linear system $P \mathbf{w} = \mathbf{m}$ for the decreased DoE $d$} 
      \State{Compute an $\ell^1$-solution $\mathbf{w}^{\ell^1} = \argmin_{\mathbf{w} \in W} \ \norm{ \mathbf{w} }_1$}
    \EndWhile 
    \State{Approximate $I[f]$ by $C^{\ell^1}_N[f^\varepsilon]$ as in \eqref{eq:l1-CF}} 
\end{algorithmic}
\end{algorithm} 

It should be pointed out that Algorithm \ref{algo:l1-CF} is only intended to provide a simple explanation of how the $\ell^1$-CFs are constructed. 
A computationally more efficient reformulation of the above construction procedure is described in Algorithm \ref{algo:CF-efficient}. 

\begin{remark}[Computation and Uniqueness of the $\ell^1$-Solution]
	Note that $\norm{\cdot}_1$ is a convex, but not strictly convex, norm. 
	Hence, in general, the $\ell^1$-solution will not be unique. 
	It is also well-known, however, that in most cases this is not a problem and the $\ell^1$-solution in fact is unique. 
	In many cases, the $\ell^1$-solution furthermore has the property of being a sparse solution; see  \cite{donoho2003optimally,donoho2006most,tsaig2006breakdown} (also see \cite{donoho2006most2}). 
	In recent years, this motivated many researchers to use the $\ell^1$-norm as a surrogate for the $\ell^0$-"norm" (number of nonzero entries) 
	\cite{candes2006stable,candes2006robust,donoho2006compressed,glaubitz2019high}.\footnote{Of course, the $\ell^0$-"norm" is not really a norm---it is not absolutely homogenous---and the problem of computing $\ell^0$-solutions is NP-hard.} 
\end{remark}

\subsection{Least Squares Cubature Formulas} 
\label{sub:LS-CFs}

Another option is to minimize a weighted $\ell^2$-norm instead of the $\ell^1$-norm. 
This approach is motivated by the wish to have---at least formally---an explicit representation for the cubature weights.  
The resulting vector of weights is referred to as the \emph{(weighted) LS solution}:
\begin{equation}\label{eq:LS-sol} 
	\mathbf{w}^{\mathrm{LS}} = \argmin_{\mathbf{w} \in W} \ \norm{ R^{-1/2} \mathbf{w} }_{2}
\end{equation}
Here, the weight matrix $R^{-1/2}$ is given by 
\begin{equation} 
	R^{-1/2} = \diag{ \frac{1}{\sqrt{r_1}}, \dots, \frac{1}{\sqrt{r_N}} }, \quad 
	r_n = \frac{\omega(\mathbf{x}_n)|\Omega|}{N} > 0.
\end{equation}
Thereby, $|\Omega|$ denotes the volume of $\Omega \subset \R^q$. 
This choice ensures that the cubature weights are nonnegative if $N$ is sufficiently larger than $K$ and $d$, respectively. 
A theoretical result concerning this is derived in \S \ref{sec:theoretical}; in particular, see Theorem \ref{thm:main}. 
Furthermore, the numerical tests performed in \S \ref{sec:tests} indicate the ratio $N \sim K^s$ with $s \approx 2$ to be sufficient. 
This is in accordance with similar findings from other works \cite{wilson1970necessary,huybrechs2009stable,glaubitz2020stable,glaubitz2020stableQRs,migliorati2018stable}. 

Regarding stability, the LS solution is expected to be inferior to the $\ell^1$ solution. 
In fact, we have $\kappa(\mathbf{w}^{\ell^1}) \leq \kappa(\mathbf{w}^{\mathrm{LS}})$. 
That said, the LS solution can be computed more efficiently than an $\ell^1$ solution. 
Moreover, $\mathbf{w}^{\mathrm{LS}}$ is unique and has an explicit representation (\cite{cline1976l_2}):
\begin{equation}\label{eq:LS-sol2}
  \mathbf{w}^{\mathrm{LS}} = R P^T (P R P^T)^{-1} \mathbf{m}  
\end{equation} 
Note that $R P^T (P R P^T)^{-1}$ is the Moore--Penrose pseudoinverse of $R^{-1/2}P$; see \cite{ben2003generalized}.
By utilizing a beautiful connection to DOPs, this formula can be considerably simplified; see \S \ref{sub:char}. 
Formula \eqref{eq:LS-sol2} has theoretical advantages and will be used to prove nonnegativity of the weights. 
In our implementation, however, the LS solution $\mathbf{w}^{\mathrm{LS}}$ is computed stably and more efficiently using the Matlab function \emph{lsqminnorm}. 
This function uses a pivoted QR decomposition of $A = P R^{1/2}$; see \cite{trefethen1997numerical,golub2012matrix,horn2012matrix,strang2019linear}.

Once $\mathbf{w}^{\mathrm{LS}}$ has been computed, the procedure is the same as for the $\ell^1$-CFs: 
The DoE $d$ is decreased until the LS solution \eqref{eq:LS-sol} yields in a nonnegative CF. 
The resulting CF is referred to as the \emph{LS-CF} and denoted by 
\begin{equation}\label{eq:LS-CF}
	C^{\mathrm{LS}}_N[f^\varepsilon] := \sum_{n=1}^N w_n^{\mathrm{LS}} f^{\varepsilon}_n.
\end{equation} 
The whole procedure is summarized in Algorithm \ref{algo:LS-CF}.
Algorithm \ref{algo:LS-CF} is only intended to provide a simple explanation of how the LS-CFs are constructed, however. 
Again, a computationally more efficient reformulation is provided in Algorithm \ref{algo:CF-efficient}.

\begin{algorithm}
\caption{Construction of LS-CFs}
\label{algo:LS-CF}
\begin{algorithmic}[1]
    \State{Determine the greatest $d \in \N$ such that $X$ is $\mathbb{P}_d(\R^q)$-unisolvent} 
    \State{Formulate the linear system $P \mathbf{w} = \mathbf{m}$ for DoE $d$}
    \State{Compute the LS solution $\mathbf{w}^{\mathrm{LS}} = \argmin_{\mathbf{w} \in W} \ \norm{ R^{-1/2} \mathbf{w} }_2$}
    \While{$\mathbf{w}^{\mathrm{LS}} \not\geq 0$} 
      \State{Reduce the DoE: $d = d-1$} 
      \State{Formulate the linear system $P \mathbf{w} = \mathbf{m}$ for the decreased DoE $d$} 
      \State{Compute the LS solution $\mathbf{w}^{\mathrm{LS}} = \argmin_{\mathbf{w} \in W} \ \norm{ R^{-1/2} \mathbf{w} }_2$}
    \EndWhile 
    \State{Approximate $I[f]$ by $C^{\mathrm{LS}}_N[f^\varepsilon]$ as in \eqref{eq:LS-CF}} 
\end{algorithmic}
\end{algorithm}

\begin{remark} 
	Nonnegativity for the $\ell^1$ and LS weights holds at latest for $d=0$. 
	In this case, there exists an $\ell^1$ solution with a single nonzero weight $w_k^{\ell^1} = I[1]$ and $w_n^{\ell^1} = 0$ for $n \neq k$. 
	The LS weights for $d=0$, on the other hand, are uniquely given by 
	${w_n^{\mathrm{LS}} = \omega(\mathbf{x}_n) ( \sum_{m=1}^N \omega(\mathbf{x}_m))^{-1} I[1]}$. 
\end{remark}
\section{Theoretical Results} 
\label{sec:theoretical}

At least formally, the LS weights \eqref{eq:LS-sol} are explicitly given by \eqref{eq:LS-sol2}. 
It is not recommended to actually solve \eqref{eq:LS-sol2}, however, since the normal matrix $P R P^T$ is known to often be ill-conditioned. 
Yet, at least when we incorporate DOPs, \eqref{eq:LS-sol2} is convenient for theoretical investigations. 
In fact, it is shown in this section that \eqref{eq:LS-sol2} reduces to ${\mathbf{w}^{\text{LS}} = R P^T \mathbf{m}}$ if $P$ and $\mathbf{m}$ are formulated w.\,r.\,t.\ DOPs. 
Building upon this formula it is proved that arbitrarily high DoEs are possible for the nonnegative $\ell^1$- and LS-CFs.

\subsection{Main Result and Consequences}

The (theoretical) main result of this work is the following theorem. 
It states that for any fixed DoE $d$ the corresponding LS weights are all nonnegative if a sufficiently large number of $\mathbb{P}_d(\R^q)$-unisolvent data points is used. 

\begin{theorem}\label{thm:main}
  Let $N_0 \in \N_0$ such that $X_{N_0} = \{ \mathbf{x}_n \}_{n=1}^{N_0} \subset \Omega$ is $\mathbb{P}_d(\R^q)$-unisolvent and 
$\omega(\mathbf{x}_n) > 0$ for ${n=1,\dots,N_0}$. 
  Moreover, for $N>N_0$, let $X_N = X_{N_0} \cup \{\mathbf{x}_n\}_{n=N_0 +1}^N$ and $r_n = \omega(\mathbf{x}_n) |\Omega|/N$. 
  Assume that 
  \begin{equation}\label{eq:cond4}
    \lim_{N \to \infty} \frac{|\Omega|}{N} \sum_{n=1}^N u(\mathbf{x}_n) v(\mathbf{x}_n) \omega(\mathbf{x}_n) 
      = \int_\Omega u(\boldsymbol{x}) v(\boldsymbol{x}) \omega(\boldsymbol{x}) \intd \boldsymbol{x} 
      \quad \forall u,v \in \mathbb{P}_d(\R^q).
  \end{equation}
  Then, there exists an $N_1 \geq N_0$ such that for all $N \geq N_1$ the cubature weights \eqref{eq:LS-sol} of the LS-CF with DoE $d$ are all nonnegative.
\end{theorem}

The proof of the above theorem is provided in \S \ref{sub:proof}.
The following corollary is a direct consequence of Theorem \ref{thm:main}. 

\begin{corollary}\label{cor:cor_main}
  Let $d \in \N_0$. 
  Under the same assumptions as in Theorem \ref{thm:main}, there exists an $N_1 \in \N_0$ such that for all $N \geq 
N_1$ the following statements hold: 
  \begin{enumerate} 
    \item[(a)] The LS-CF with cubature weights $\mathbf{w}^{\mathrm{LS}} \in \R^N$ given by \eqref{eq:LS-sol} is nonnegative.
    \item[(b)] The $\ell^1$-CF with cubature weights $\mathbf{w}^{\ell^1} \in \R^N$ given by \eqref{eq:l1-sol} is nonnegative. 
  \end{enumerate}
\end{corollary}

The first statement just summarizes Theorem \ref{thm:main}.  
The second statement follows from the observation that ${\kappa(\mathbf{w}^{\ell^1}) \leq \kappa(\mathbf{w}^\mathrm{LS})}$.

\begin{remark}
  	It can be argued that \eqref{eq:cond4} is a reasonable assumption on the (sequence) of data points. 
  	E.\,g., when the data points are obtained by random samples, 
  $\frac{|\Omega|}{N} \sum_{n=1}^N u(\mathbf{x}_n) v(\mathbf{x}_n) \omega(\mathbf{x}_n)$ corresponds to MC integration. 
  	In this is \eqref{eq:cond4} is ensured in a probabilistic sense by the law of large numbers; see \cite[Ch. 5.9]{davis2007methods}. 
	In fact, this special case is in accordance with the results obtained in \cite{migliorati2018stable}. 
  	Other classes of (sequences of) data points satisfying \eqref{eq:cond4} include low-discrepency  \cite{metropolis1949monte,caflisch1998monte,dick2013high} and \emph{equidistributed} (also called \emph{uniformly distributed}) \cite{weyl1916gleichverteilung,kuipers2012uniform} sequences of (partially or fully deterministic) data points. 
\end{remark}

\subsection{Continuous and Discrete Orthogonal Polynomials} 
\label{sub:DOPs}

Let us consider the following continuous inner product induced by the nonnegative weight function $\omega$: 
\begin{equation}\label{eq:cont-inner-prod}
  \scp{u}{v} = \int_{\Omega} u(\boldsymbol{x}) v(\boldsymbol{x}) \omega(\boldsymbol{x}) \intd \boldsymbol{x}
\end{equation}
The corresponding norm is $\norm{\cdot} = \sqrt{\scp{\cdot}{\cdot}}$. 
If the inner product \eqref{eq:cont-inner-prod} is positive definite on $\mathbb{P}_d(\R^q)$ it induces a basis of orthogonal (OG) polynomials $\{ \pi_k \}_{k=1}^K$, where ${K=\text{dim } \mathbb{P}_d(\R^q)}$. 
That is, the $\pi_k$ satisfy 
\begin{equation}
  \scp{\pi_k}{\pi_l} = \delta_{k,l},
\end{equation} 
and span the space $\mathbb{P}_d(\R^q)$. 
These polynomials are referred to as \emph{continuous orthogonal polynomials (COPs)} and denoted by $\pi_k(\cdot,\omega)$.
%
Analogously, a discrete inner product can be induced by a vector of nonnegative weights $\mathbf{r} = (r_1,\dots,r_N)$:
\begin{equation}\label{eq:disc-inner-prod}
  [u,v]_N = \sum_{n=1}^N r_n u(\mathbf{x}_n) v(\mathbf{x}_n)
\end{equation} 
The corresponding norm is $\norm{\cdot}_N = \sqrt{[\cdot,\cdot]_N}$.
Let us denote the set of all data points $\mathbf{x}_n$ for which the corresponding weights $r_n$ are positive by $X^+$. 
That is, $X^+ = \{ \, \mathbf{x}_n \in X \mid r_n > 0 \, \}$. 
If this set is $\mathbb{P}_d(\R^q)$-unisolvent, then \eqref{eq:disc-inner-prod} is positive definite on $\mathbb{P}_d(\R^q)$. 
Hence, \eqref{eq:disc-inner-prod} induces a basis $\{ \pi_k \}_{k=1}^K$ of so-called \emph{DOPs} then. 
These satisfy 
\begin{equation}
  [\pi_k,\pi_l]_N = \delta_{k,l},
\end{equation} 
while spanning $\mathbb{P}_d(\R^q)$. 
We denote them by $\pi_k(\cdot,\mathbf{r})$. 
%
Both OG bases can be constructed, for instance, by \emph{Gram--Schmidt (GS) orthogonalization} \cite{trefethen1997numerical}: 
Let $\{e_k\}_{k=1}^K$ be the set of monomials $e_k(\boldsymbol{x}) := \boldsymbol{x}^{\boldsymbol{\alpha}}$ with $|\boldsymbol{\alpha}| \leq d$. 
These are assumed to be ordered w.\,r.\,t.\ their degree. 
In particular, $e_1 \equiv 1$.
Then, the OG polynomials are respectively constructed as 
\begin{equation}\label{eq:Gram-Schmidt}
\begin{aligned}
  \tilde{\pi}_k(\boldsymbol{x};\omega) & = e_k(\boldsymbol{x}) - \sum_{l=1}^{k-1} \scp{e_k}{\pi_l(\cdot;\omega)} \pi_l(\boldsymbol{x};\omega), \quad 
  && \pi_k(\boldsymbol{x};\omega) = \frac{\tilde{\pi}_k(\boldsymbol{x};\omega)}{\norm{\tilde{\pi}_k(\cdot;\omega)}}, \\ 
  \tilde{\pi}_k(\boldsymbol{x};\mathbf{r}) & = e_k(\boldsymbol{x}) - \sum_{l=1}^{k-1} [e_k,\pi_l(\cdot;\mathbf{r})]_N \pi_l(\boldsymbol{x};\mathbf{r}), \quad 
  && \pi_k(\boldsymbol{x};\mathbf{r}) = \frac{\tilde{\pi}_k(\boldsymbol{x};\mathbf{r})}{\norm{\tilde{\pi}_k(\cdot;\mathbf{r})}_N}.
\end{aligned}
\end{equation} 
Note that we only utilize GS orthogonalization for theoretical purposes. 
In our implementation, the LS weights are computed based on a pivoted QR decomposition of $A = P R^{1/2}$. 

\subsection{Characterization of the Least Squares Solution} 
\label{sub:char}

At least formally, the LS solution $\mathbf{w}^{\text{LS}}$ is given by \eqref{eq:LS-sol2}.
The real beauty of the LS approach is revealed, however, once we incorporate the concept of DOPs. 
In fact, the matrix product $P R P^T$ in \eqref{eq:LS-sol2} can be identified as a Gram matrix w.\,r.\,t.\ the discrete inner product \eqref{eq:disc-inner-prod}:
\begin{equation}
  P R P^T = 
  \begin{pmatrix}
    [p_1,p_1]_N & \dots & 
    [p_1,p_K]_N \\ 
    \vdots & & \vdots \\ 
    [p_K,p_1]_N & \dots & 
    [p_K,p_K]_N \\ 
  \end{pmatrix}
\end{equation}
Let us formulate the Vandermonde matrix $P$ and the vector of moments $\mathbf{m}$ in  \eqref{eq:ex-cond-system} w.\,r.\,t.\ the basis of DOPs $\{\pi_k(\cdot,\mathbf{r})\}_{k=1}^K$.
Then, $P R P^T = I$ and therefore 
\begin{equation}\label{eq:LS-sol3}
  \mathbf{w}^{\mathrm{LS}} = R P^T \mathbf{m}.  
\end{equation}
Thus, the LS weights $\mathbf{w}^{\mathrm{LS}}$ are explicitly given by 
\begin{equation}\label{eq:LS-sol-explicit}
  w_n^{\mathrm{LS}} = r_n \sum_{k=1}^K \pi_k( \mathbf{x}_n ; \mathbf{r}) I[ \pi_k( \, \cdot \, ; \mathbf{r} ) ], \quad n=1,\dots,N. 
\end{equation}
In particular, this formula enables us to subsequently prove nonnegativity of the LS weights.

\subsection{Proof of the Main Results} 
\label{sub:proof}

Let us start with two preliminary results on the convergence of discrete inner products and the induced DOPs. 
Afterward, these will be used to prove our main result, Theorem \ref{thm:main}.

\begin{lemma}\label{lem:lem1}
  Assume that
  \begin{equation}\label{eq:cond1}
    \lim_{N \to \infty} [u,v]_N = \scp{u}{v} \quad \forall u,v \in \mathbb{P}_d(\R^q).
  \end{equation}
  Moreover, let $(u_N)_{N \in \N}$ and $(v_N)_{N \in \N}$ be two sequences in $\mathbb{P}_d(\R^q)$ with 
  \begin{equation}\label{eq:cond2}
    u_N \to u, \quad v_N \to v \text{ in } L^\infty(\Omega)
  \end{equation}
  for $N \to \infty$, where $u,v \in \mathbb{P}_d(\R^q)$ and $\Omega \subset \R^q$. 
  Then, 
  \begin{equation}\label{eq:assertion}
    \lim_{N \to \infty} [u_N,v_N]_N = \scp{u}{v}.
  \end{equation}
\end{lemma}

\begin{proof}
  Note that 
  \begin{equation} 
  \begin{aligned} 
    \left| \scp{u}{v} - [u_N,v_N]_N \right| 
      \leq & \left| \scp{u}{v} - [u,v]_N \right| + \left| [u,v]_N - [u_N,v]_N \right| \\ 
      & + \left| [u_N,v]_N - [u_N,v_N]_N 
\right|. 
  \end{aligned}
  \end{equation}
  The first term on the right-hand side converges to zero due to \eqref{eq:cond1}. 
  For the second term, the Cauchy--Schwarz inequality gives 
  \begin{equation}
    \left| [u,v]_N - [u_N,v]_N \right|^2  
      = \left| [u-u_N,v]_N \right|^2  
      \leq \norm{ u-u_N }_N^2 \norm{ v }_N^2.
  \end{equation} 
  Furthermore, \eqref{eq:cond1} implies $\norm{ v }_N^2 \to \norm{ v }^2$ for $N \to \infty$.
  Finally, the H\"older inequality and \eqref{eq:cond2} yield 
  \begin{equation}
    \norm{ u-u_N }_N^2 \leq \norm{ 1 }_N^2 \norm{ u-u_N }_{L^\infty(\Omega)}^2 \to 0, 
    \quad N \to \infty.
  \end{equation}
  Thus, the second term converges to zero as well. 
  A similar argument can be used to show that the third term converges to zero. 
\end{proof}

Next, we demonstrate that the DOPs 
$\pi_k(\cdot,\mathbf{r})$ converge uniformly to the COPs $\pi_k(\cdot,\omega)$ if the corresponding discrete inner product converges to the continuous one for all polynomials of degree at most $d$.

\begin{lemma}\label{lem:lem2}
  Assume that 
  \begin{equation}\label{eq:lem2-cond}
    \lim_{N \to \infty} [u,v]_N = \scp{u}{v} \quad \forall u,v \in \mathbb{P}_d(\R^q).
  \end{equation} 
  For $k=1,\dots,K$, let $\pi_k(\cdot;\mathbf{r})$ and $\pi_k(\cdot;\omega)$ respectively denote the $k$-th DOP and COP constructed by GS orthogonalization \eqref{eq:Gram-Schmidt}. 
  Then, we have 
  \begin{equation}
    \pi_k(\cdot;\mathbf{r}) \to \pi_k(\cdot;\omega) \text{ in } L^\infty(\Omega)
  \end{equation}
  for $N \to \infty$ and $k=1,\dots,K$.
\end{lemma}

\begin{proof}
  	The assertion is proven by induction. 
	For $k=1$ the assertion is trivial and essentially follows from ${\|1\|_N \to \|1\|}$ for ${N \to \infty}$. 
	In a second step, it is argued that if the assertion holds for the first $k-1$ OG polynomials, then it also holds for the $k$-th OG polynomial. 
  Thus, assume that 
  \begin{equation}
    \pi_l(\cdot;\mathbf{r}) \to \pi_l(\cdot;\omega) \text{ in } L^\infty(\Omega), \quad N \to \infty,
  \end{equation}
  holds for $l=1,\dots,k-1$.
  By the GS orthogonalization, the $k$-th OG polynomials are given by \eqref{eq:Gram-Schmidt}. 
  Lemma \ref{lem:lem1} implies   
  \begin{equation}
  	[e_k,\pi_l(\cdot;\mathbf{r})]_N \to \scp{e_k}{\pi_l(\cdot;\omega)}, \quad l=1,\dots,k-1,
  \end{equation}
  and therefore 
  \begin{equation}
    \tilde{\pi}_k(\cdot;\mathbf{r}) \to \tilde{\pi}_k(\cdot;\omega) \text{ in } L^\infty(\Omega), 
    \quad N \to \infty.
  \end{equation} 
  Furthermore, Lemma \ref{lem:lem1} yields 
  ${\|\tilde{\pi}_k(\cdot;\mathbf{r})\|_N \to \|\tilde{\pi}_k(\cdot;\omega)\|}$ 
  for $N \to \infty$. 
  This implies 
  \begin{equation}
    \pi_k(\cdot;\mathbf{r}) \to \pi_k(\cdot;\omega) \text{ in } L^\infty(\Omega), 
    \quad N \to \infty,
  \end{equation}
  which completes the proof. 
\end{proof}

The previous two lemmas can now be utilized to prove the main result. 
That is, Theorem \ref{thm:main}.

\begin{proof}[Proof of Theorem \ref{thm:main}]
  Recal that the LS weights $\mathbf{w}^{\mathrm{LS}}$ are explicitly given by \eqref{eq:LS-sol-explicit}. 
  Defining
  \begin{equation}
    \varepsilon_k := [\pi_k(\cdot;\mathbf{r}),1]_N - \scp{\pi_k(\cdot;\mathbf{r})}{1},
  \end{equation}
  the LS weights can be rewritten as 
  \begin{align}
    w_n^{\mathrm{LS}} 
      = r_n \left( \pi_1(\mathbf{x}_n;\mathbf{r}) [\pi_1(\cdot;\mathbf{r}),1]_N 
	- \sum_{k=1}^K \varepsilon_k \pi_k(\mathbf{x}_n;\mathbf{r}) \right).
  \end{align}
  Assuming the DOPs $\{ \pi_k(\cdot;\mathbf{r}) \}_{k=1}^K$ are ordered, one has $\pi_1(\cdot;\mathbf{r}) \equiv 1/\|1\|_N$. 
  This yields 
  \begin{equation}
    \pi_1(\mathbf{x}_n;\mathbf{r}) [\pi_1(\cdot;\mathbf{r}),1]_N 
      = [\pi_1(\cdot;\mathbf{r}),\pi_1(\cdot;\mathbf{r})]_N 
      = 1.
  \end{equation}
  The assertion $w_n^{\mathrm{LS}} \geq 0$ is therefore equivalent to 
  \begin{equation}
    \sum_{k=1}^K \varepsilon_k \pi_k(\mathbf{x}_n;\mathbf{r}) \leq 1.
  \end{equation}
  Next note that \eqref{eq:cond4} implies $[\cdot,\cdot]_N$ converging to 
  $\scp{\cdot}{\cdot}$ for all polynomials of degree at most $d$. 
  Hence, Lemma \ref{lem:lem2} provides us with 
  \begin{equation}\label{eq:proof1}
    \pi_k(\cdot;\mathbf{r}) \to \pi_k(\cdot;\omega) \text{ in } L^\infty(\Omega)
  \end{equation}
  for $N \to \infty$ and $k=1,\dots,K$. 
  In particular, the DOPs are uniformly bounded. 
  That is, there exists a constant $C>0$ such that 
  ${| \pi_k(\boldsymbol{x};\mathbf{r}) | \leq C}$ for all $\boldsymbol{x} \in \Omega$ and $k=1,\dots,K$. 
  Thus,  
  \begin{equation}
    \sum_{k=1}^K \varepsilon_k \pi_k(\mathbf{x}_n;\mathbf{r}) 
      \leq C \sum_{k=1}^K \left| \varepsilon_k \right|.
  \end{equation}
  Since uniform convergence \eqref{eq:proof1} holds, Lemma \ref{lem:lem1} yields $\varepsilon_k \to 0$ for $N \to \infty$ and $k=1,\dots,K$. 
  Hence, there exists an $N_1 \geq N_0$ such that 
  \begin{equation}
    \left| \varepsilon_k \right| \leq \frac{1}{CK}, \quad k=1,\dots,K,
  \end{equation}
  for $N \geq N_1$. 
  Finally, this implies 
  \begin{equation}
    \sum_{k=1}^K \varepsilon_k \pi_k(\mathbf{x}_n;\mathbf{r}) 
      \leq 1
  \end{equation}
  and therefore the assertion.
\end{proof}
 
\section{Connection to Other Cubature Formulas}
\label{sec:connection}

In what follows, connections of the proposed $\ell^1$- and LS-CF to some well-known CFs are discussed.  

\begin{remark}[Minimum Norm CFs]
	The proposed approach might be best compared to \emph{relative minimum-norm} CFs; 
	see \cite[Chapter 4 and 5]{sard1949best} or the review \cite{haber1970numerical} and references therein.
	Given a fixed set of data points $X$, these are constructed by considering the integration error ${I[f]-C_N[f]}$ as a linear functional and minimizing its operator norm. 
	In comparison, here we aim to minimize the operator norm of the CF considered as a linear functional, 
	\begin{equation} 
		C_N: ( \R^N, \|\cdot\| ) \to ( \R, |\cdot| ), \quad 
		\mathbf{f} \mapsto \mathbf{w} \cdot \mathbf{f} = \sum_{n=1}^N w_n f(\mathbf{x}_n).
	\end{equation}
	If $\| \cdot \| = \| \cdot \|_\infty$, the operator norm of $C_N$ is $\kappa(\mathbf{w})$, 
	resulting in the $\ell^1$-CFs. 
	For $\| \cdot\| = \|R^{-1/2} \cdot \|_2$, on the other hand, the operator norm of $C_N$ is $\|R^{-1/2} \mathbf{w}\|_2$. 
	This yields the LS-CFs.
\end{remark}

\begin{remark}[Alternative Stability Measures] 
\label{rem:other-stab-measures}
	Other choices for $\| \cdot \|$ are possible as well. 
	Note that by H\"older's inequality, 
	\begin{equation} 
		| C_N[f] | \leq \norm{ \mathbf{w} }_p \norm{ \mathbf{f} }_q
	\end{equation} 
	for $1 < p,q < \infty$ with $1/p + 1/q = 1$. 
	Equality holds if and only if the vectors $|\mathbf{w}|^p$ and $|\mathbf{f}|^q$ are linearly dependent. 
	Here, $|\mathbf{v}|^p = (|v_1|^p, \dots, |v_N|^p )^T$ for $\mathbf{v} \in \R^N$.
	The operator norm of $C_N$ is therefore given by $\|\mathbf{w}\|_p$. 
	In this work, however, only the cases $p = 1$ and $p = 2$ are considered.  
\end{remark}

\begin{remark}[Monte Carlo CFs]
	The LS-CF \eqref{eq:LS-CF} can be seen as a high-order correction to (Q)MC methods \cite{metropolis1949monte,caflisch1998monte,dick2013high}.
	In these, the data points are obtained by (uniform) random samples and the weights are simply $w_n = |\Omega| \omega(\mathbf{x}_n) / N$. 
	At the same time, recall (see \S \ref{sub:char}) that the LS weights are explicitly given by 
	$w_n^{\mathrm{LS}} = r_n \sum_{k=1}^K \pi_k( \mathbf{x}_n ; \mathbf{r}) I[ \pi_k( \, \cdot \, ; \mathbf{r} ) ]$, where $\{ \pi_k( \, \cdot \, ; \mathbf{r}) \}_{k=1}^K$ is a basis of DOPs. 
	For fixed $K$ and an increasing number of data points, $\pi_k( \mathbf{x}_n ; \mathbf{r}) I[ \pi_k( \, \cdot \, ; \mathbf{r} ) ]$ converges to the Kronecker delta $\delta_{1,k}$. 
	Hence, the difference between the (Q)MC and LS weights converges to zero. 
\end{remark}

\begin{remark}[Exact Integration of Discrete LS Approximations]
	The LS-CF $C^{\textit{LS}}_N[f]$ defined as in \eqref{eq:LS-CF} corresponds to exact integration of the following discrete LS approximation of $f$ from $\mathbb{P}_d(\R^q)$: 
	\begin{equation} 
		\hat{f}(x) = \sum_{k=1}^K c_k p_k(x) 
		\quad \text{s.t.} \quad 
		\norm{ R^{1/2} ( P^T \mathbf{c} - \mathbf{f} ) }_2 
		\ \text{is minimized}, 
	\end{equation} 
	where $\mathbf{c} = (c_1,\dots,c_K)^T$.
	That is, $C^{\textit{LS}}_N[f] = I[\hat{f}]$. 
	This can be noted by representing $\hat{f}$ w.\,r.\,t.\ to a basis of DOPs $\{ \pi_k(\cdot;\mathbf{r}) \}_{k=1}^K$ corresponding to the discrete inner product $[\cdot,\cdot]_N$ as in \eqref{eq:disc-inner-prod}. 
	Then, 
	\begin{equation} 
		\hat{f}(x) = \sum_{k=1}^K [ f, \pi_k(\cdot;\mathbf{r}) ]_N \pi_k(x;\mathbf{r}).
	\end{equation} 
	Integration therefore yields 
	\begin{equation} 
		I[\hat{f}] 
			= \sum_{k=1}^K [ f, \pi_k(\cdot;\mathbf{r}) ]_N I[\pi_k(\cdot;\mathbf{r})] 
			= \sum_{n=1}^N \left( \sum_{k=1}^K \pi_k(\mathbf{x}_n;\mathbf{r}) I[\pi_k(\cdot;\mathbf{r})] \right) f(\mathbf{x}_n) 
			= C^{\textit{LS}}_N[f]. 
	\end{equation} 
	The last equality follows from \eqref{eq:LS-sol-explicit}. 
	Building upon this connection, in \cite{migliorati2018stable} high-order CFs for independent random points were constructed. 
	These were shown to be positive with a high probability if the number of (random) data points is sufficiently larger than the DoE. 
	In particular, it was stated that the proportionality between $N$ and $K$ should be at least quadratic. 
	This is in accordance with the results presented here. 
	That said, we do not restrict ourselves to random points in the present manuscript. 
\end{remark}

\begin{remark}[Optimization-based CFs] 
	In the introduction, we mentioned some existing optimization strategies for computing high-order CFs \cite{taylor2000algorithm,taylor2007cardinal,ryu2015extensions,jakeman2018generation,keshavarzzadeh2018numerical}. 
	Besides relying on the ability to have full flexibility in placing data points inside the domain $\Omega$, the success of these often depends on the initial guess for the data points and convergence to an appropriate solution (otherwise, the weights might not always be ensured to be nonnegative). 
	In contrast, considering a fixed and prescribed set, in a sense, makes the corresponding optimization problem simpler. 
	For fixed data points, we 'only' have to optimize the cubature weights, rather than also optimizing for the data points. 
	At the same time, this seeming simplification comes with a new question: 
	``What is telling us that the fixed and prescribed set of data points actually supports a nonnegative and exact CF for some $d>0$?". 
	However, this is ensured by Corollary \ref{cor:cor_main}. 
	Roughly rephrasing this corollary, for any $d>0$ the solutions of the optimization problems \eqref{eq:l1-sol} and \eqref{eq:LS-sol}---corresponding to the weights of the LS- and $\ell^1$-CF---are nonnegative if a sufficiently large set of data points satisfying \eqref{eq:cond4} is considered. 
	That said, for determining the weights corresponding to the $\ell^1$-CF, we still rely on the numerical optimization method that is used to solve the basis pursuit problem \eqref{eq:l1-sol} to actually converge to one of these solutions.
\end{remark} 
\section{Efficient Construction and Implementation}
\label{sec:constr}

Below, some comments on computational aspects of the proposed LS- and $\ell^1$-CFs are provided. 
The Matlab code that was used to produce the subsequent numerical tests can be found at \cite{glaubitz2020github}.

\subsection{Computation of the Moments} 

In our implementation, we computed the matrix $P$ from a basis of monomials. 
For many domains and weight functions the corresponding moments can be found in the literature; see \cite{haber1970numerical,cools1997constructing,folland2001integrate,davis2007methods} and references therein.
In case the monomial's moments are not known, they can still be computed, for instance, by using some available CF on $\Omega$. 
Note that for the evaluation of the basis functions $\{p_k\}_{k=1}^K$, one is not limited to the set of data points.

\subsection{Finding Stable Cubature Formulas}

Algorithms \ref{algo:l1-CF} and \ref{algo:LS-CF} describe a simple procedure to determine nonnegative $\ell^1$- and LS-CFs. 
The idea behind these is to start with the highest possible $d$ such that the set of data points is still $\mathbb{P}_d(\R^q)$-unisolvent. 
Then, the DoE $d$ is decreased until the resulting CF is also ensured to be nonnegative.  
In the later numerical tests, however, the final DoE of these CFs is found to be significantly smaller than the highest possible $d$ such that the set of data points is $\mathbb{P}_d(\R^q)$-unisolvent. 
This is in accordance with prior works in one dimension \cite{wilson1970necessary,huybrechs2009stable,glaubitz2020shock,glaubitz2020stable}.
Hence, instead of Algorithms \ref{algo:l1-CF} and \ref{algo:LS-CF}, we utilized the following (more efficient) algorithm to construct $\ell^1$- and LS-CFs in our implementation. 

\begin{algorithm}
\caption{Efficient Construction of $\ell^1$- and LS-CFs}
\label{algo:CF-efficient}
\begin{algorithmic}[1]
    \State{$d, r, K, w_{\text{min}} = 0$} 
    \While{$r = K$ and $w_{\text{min}} \geq 0$} 
      \State{Increase the DoE: $d = d+1$} 
      \State{$K = \binom{d+q}{q}$}
      \State{Formulate the linear system $P \mathbf{w} = \mathbf{m}$ for DoE $d$} 
      \State{Compute the rank of $P$: $r = \text{rank}(P)$}
      \State{Compute the $\ell^1$/LS-solution $\mathbf{w}^{*}$} 
      \State{Determine the smallest weight: $w_{\text{min}} = \min( \mathbf{w}^{*} )$}
    \EndWhile 
    \State{Decrease the DoE: $d = d-1$} 
      \State{Formulate the linear system $P \mathbf{w} = \mathbf{m}$ for DoE $d$} 
      \State{Compute the $\ell^1$/LS-solution $\mathbf{w}^{*}$}
    \State{Approximate $I[f]$ by the corresponding $\ell^1$/LS-CF} 
\end{algorithmic}
\end{algorithm} 

Note that $\text{rank}(P) = K$ with $K = \binom{d+q}{q}$ is equivalent to $X$ being $\mathbb{P}_d(\R^q)$-unisolvent.

\subsection{Computation of the $\ell^1$- and LS-Solution} 

Recall that the $\ell^1$-solution is defined by the basis pursuit problem \eqref{eq:l1-sol}.
This problem can be rewritten as a linear programming problem \cite{bloomfield1983least,dantzig1998linear,dantzig2006linear,gill1991numerical,vanderbei2020linear}: 
\begin{equation}\label{eq:LP}
	\min_{\mathbf{w} \in \R^N} \mathbf{1}^T \mathbf{w} 
	\quad \text{s.t.} \quad 
	P \mathbf{w} = \mathbf{m}, \quad 
	\mathbf{w} \geq 0
\end{equation}
In our implementation, we therefore compute the $\ell^1$-solution based on \eqref{eq:LP} by  MATLAB's function \emph{linprog}. 
By default, this function uses a dual simplex algorithm. 
The LS solution \eqref{eq:LS-sol}, on the other hand, is computed by Matlab's function \emph{lsqminnorm}. 
This function uses a pivoted QR decomposition of $A = P R^{1/2}$; see \cite{trefethen1997numerical,golub2012matrix,horn2012matrix,strang2019linear}. 
\section{Numerical Results} 
\label{sec:tests}

In this section, the proposed stable high-order CFs are numerically investigated for a variety of different test cases. 
Among these are the (hyper-)cube ${C_q = [-1,1]^q}$ and ball ${B_q = \{ \mathbf{x} \in \R^q \mid \| \mathbf{x} \|_2 \leq 1 \}}$ in two and three dimensions ($q=2,3$). 
Here, $\|\mathbf{x}\|_2^2 = x_1^2 + \dots + x_q^2$ for $\mathbf{x} = (x_1,\dots,x_q)^T$. 
The domain's volume is respectively given by $|C_q| = 2^q$ and ${|B_q| = (\pi^\frac{q}{2})/\Gamma\left( \frac{q}{2}+1 \right)}$.
Here, $\Gamma$ denotes the usual gamma function \cite[Chapter 5]{dlmf2020digital}. 
The moments of the monomials corresponding to different weight functions can be found in \ref{sec:moments}. 

\begin{figure}[tb]
  \centering
  \begin{subfigure}[b]{0.32\textwidth}
    \includegraphics[width=\textwidth]{%
      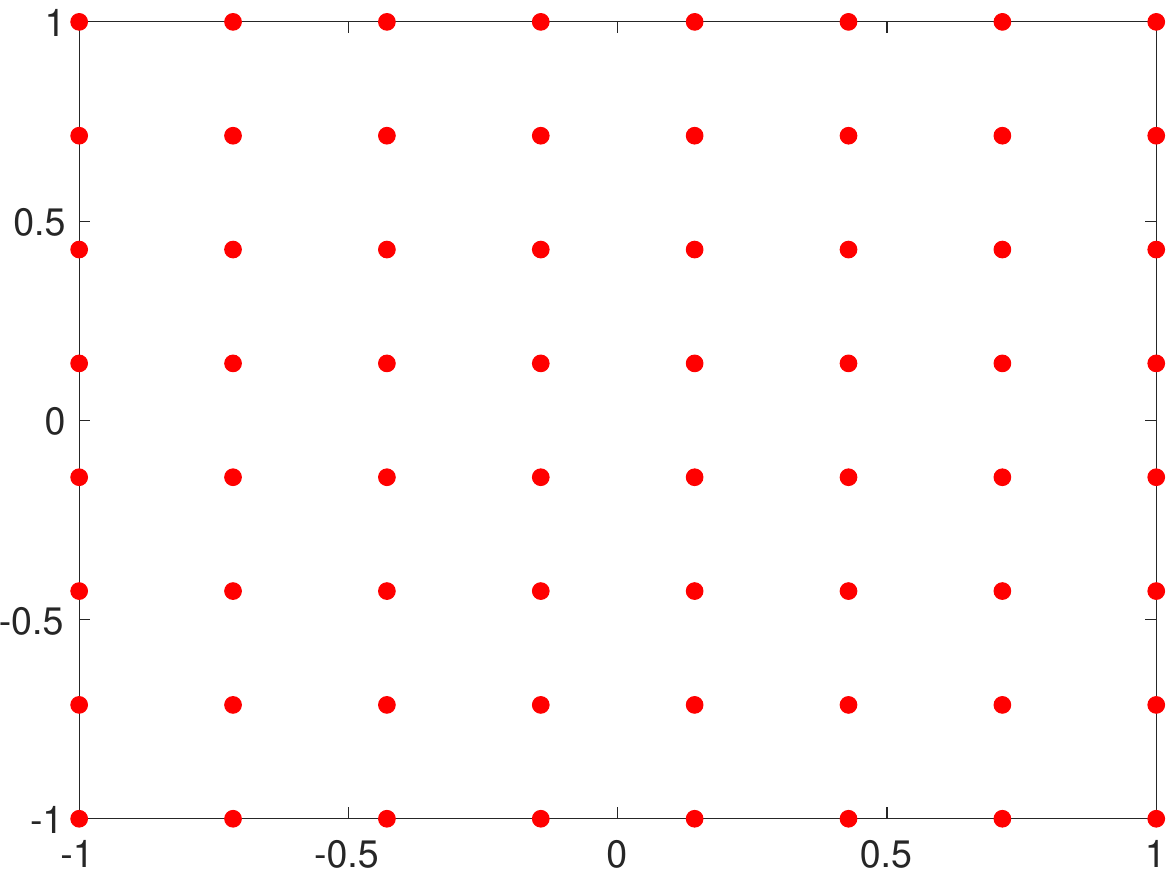} 
    \caption{Equidistant points}
    \label{fig:points_equid}
  \end{subfigure}%
  ~
  \begin{subfigure}[b]{0.32\textwidth}
    \includegraphics[width=\textwidth]{%
      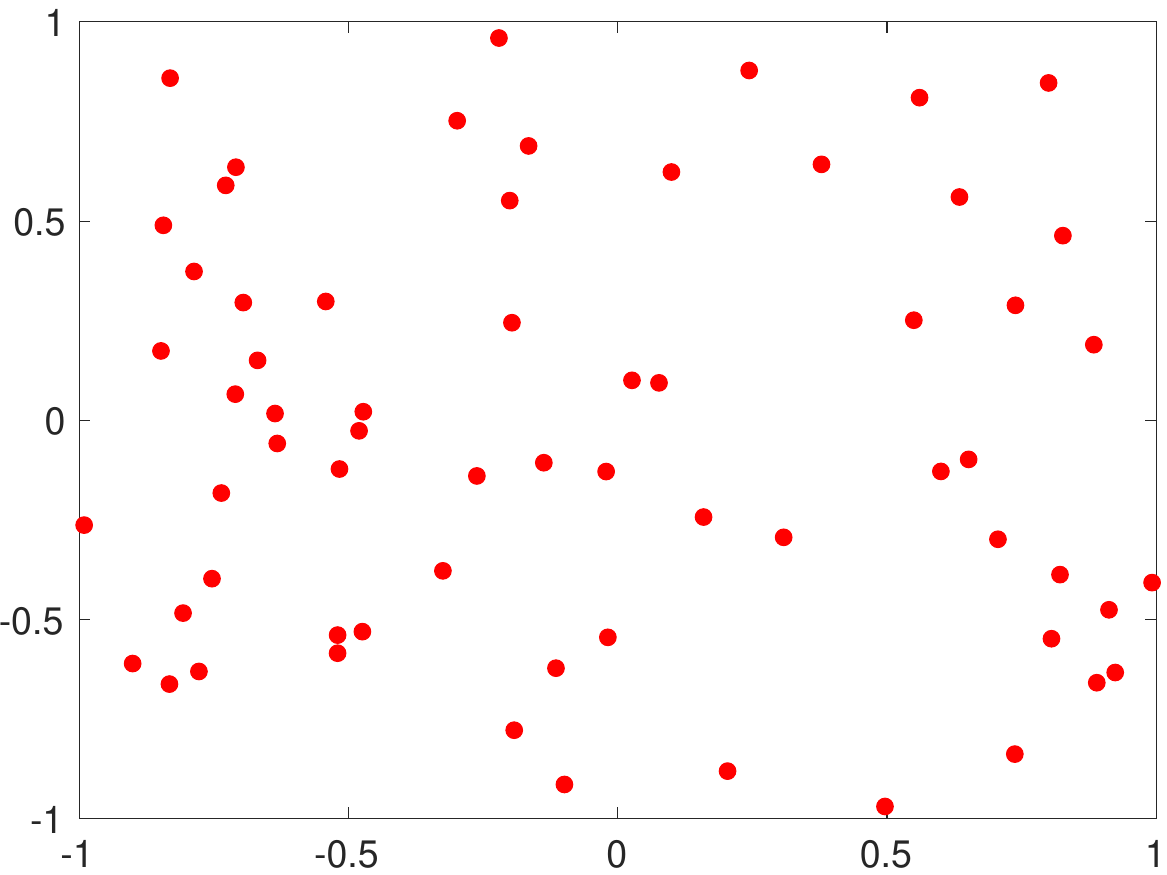} 
    \caption{Random points}
    \label{fig:points_uniform}
  \end{subfigure}%
  ~
  \begin{subfigure}[b]{0.32\textwidth}
    \includegraphics[width=\textwidth]{%
      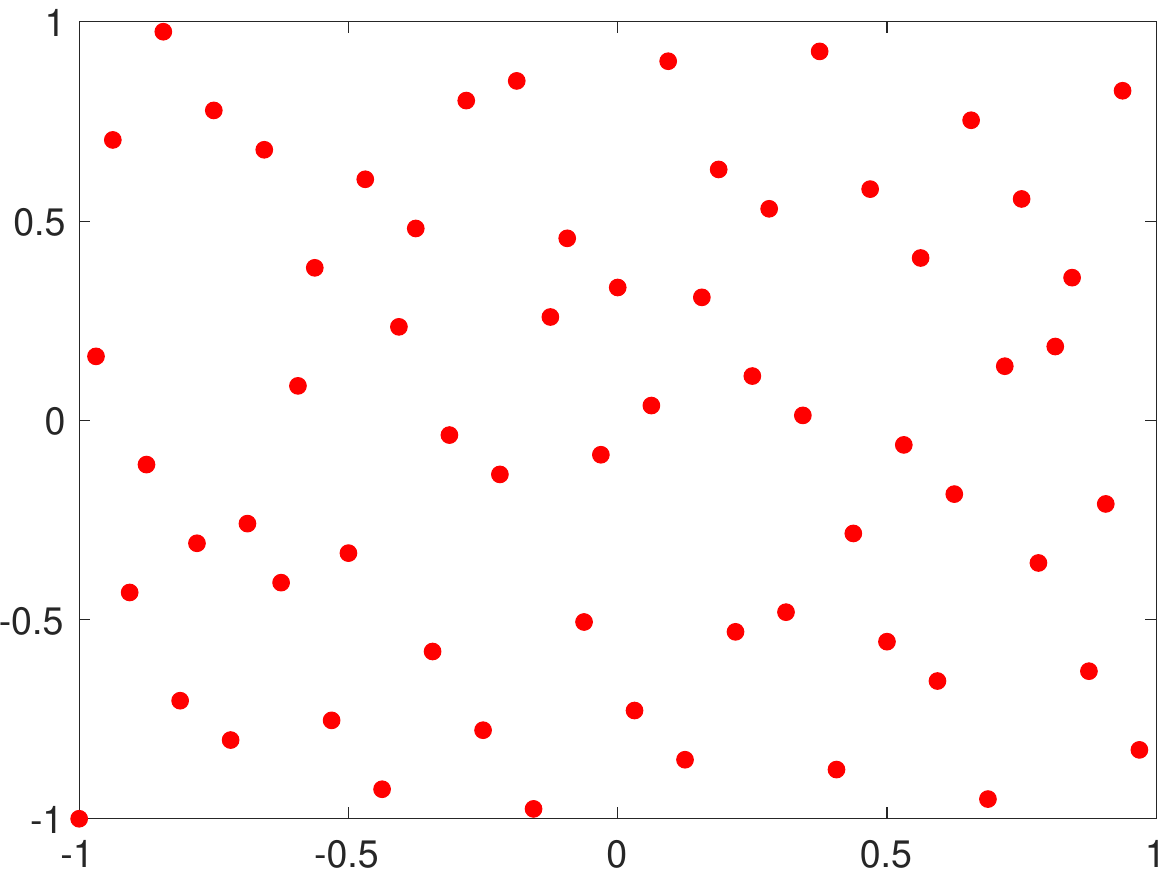} 
    \caption{Halton points}
    \label{fig:points_Halton}
  \end{subfigure}%
  \caption{Illustration of three different types of $N=64$ data points for the two-dimensional cube $C_2 = [-1,1]^2$}
  \label{fig:points}
\end{figure}

Furthermore, three different types of data points are considered: 
(1) equidistant points, which are fully deterministic; 
(2) Random points that are independent and identically distributed (i.\,i.\,d.)\ according to the uniform distribution $\mathcal{U}([-1,1]^q)$; and
(3) Halton points \cite{halton1960efficiency}, which are quasi-random and belong to the family of so-called \emph{low discrepancy sequences}.\footnote{The Halton points generalize the one-dimensional van der Corput points; see \cite[Erste Mitteilung]{van1935verteilungsfunktionen} or \cite{kuipers2012uniform}.}
Such points are developed to minimize the upper bound provided by the Koksma--Hlawka theorem \cite{hlawka1961funktionen,niederreiter1992random}. 
They yield the rate of convergence of the MC method to increase from $1/2$ to essentially $1$; see \cite{caflisch1998monte,dick2013high,trefethen2017cubature} and references therein. 
An illustration of these points in two dimensions is provided by Figure \ref{fig:points}. 
For the corresponding (hyper-)ball, the subset of points with radius not greater than $1$ is selected.

\subsection{The Ratio Between $N$ and $d$} 
\label{sub:ratio}

First, the ratio between the number of data points $N$ and the number $K$ of basis functions spanning $\mathbb{P}_d(\R^q)$ is investigated.
Recall that $K = \binom{d+q}{q}$. 
Hence, the asymptotic relation $K \sim d^q/(q!)$ holds. 
That is, $\lim_{d \to \infty} K / d^q = 1/q!$; see \cite{de1981asymptotic}.

\begin{table}[!tb]
\renewcommand{\arraystretch}{1.25}
\centering 
  	\begin{tabular}{c c c c c c c c c c c} 
    	\multicolumn{11}{c}{LS-CF on the Cube} \\ \hline 
    \multicolumn{3}{c}{} & \multicolumn{4}{c}{$\omega \equiv 1$} & & \multicolumn{3}{c}{$\omega(\boldsymbol{x}) = (1-x_1^2)^{1/2} \dots (1-x_q^2)^{1/2}$} \\ \hline 
    $q$ &\multicolumn{2}{c}{} & Legendre & equidistant & random & Halton & & equidistant & random & Halton \\ \hline 
	$2$ & s  & & 1.9	  	& 1.9   		&  9.0e-1  & 1.3     	& & 1.7      	& 1.3 		& 1.3 \\
    		   & C & & 3.0e-1 	& 4.4e-1	& 2.3e+1	& 1.4			& & 5.1e-1	& 9.6    		& 1.0 \\ \hline 
	$3$ & s  & & 2.8 		& 1.1			& 2.1			& 1.6			& & 1.6			& 8.8e-1	& 1.5 \\
    		   & C & & 2.1e-1 	& 9.0e+1	& 2.3e-1	& 2.2e-1	& & 2.2			& 6.9e+1	&  5.9e-1 \\ \hline 
		   \\	  
	\multicolumn{11}{c}{$\ell^1$-CF on the Cube} \\ \hline 
    \multicolumn{3}{c}{} & \multicolumn{4}{c}{$\omega \equiv 1$} & & \multicolumn{3}{c}{$\omega(\boldsymbol{x}) = (1-x_1^2)^{1/2} \dots (1-x_q^2)^{1/2}$} \\ \hline 
    $q$ &\multicolumn{2}{c}{} & Legendre & equidistant & random & Halton & & equidistant & random & Halton \\ \hline 
	$2$ & s  & & 1.9		& 1.5      	& 1.4 		& 1.6			& & 1.7 		& 1.6       	& 1.0 \\
    		   & C & & 3.0e-1	& 3.6e-1	& 7.4e-1 	& 2.7e-1  	& & 2.0e-1	& 2.0e-1	& 3.6 \\ \hline 
	$3$ & s  & & 2.8 		& 1.5			& 1.7			& 1.7			& & 1.5			& 9.4e-1	& 1.8 \\
    		   & C & & 2.1e-1 	& 3.9e-1	& 9.3e-2	& 5.3e-2	& & 9.4e-1	& 8.1			& 2.6e-2 \\ \hline 
		   \\	  
	\end{tabular} 
  	\begin{tabular}{c c c c c c c c c c} 
    	\multicolumn{10}{c}{LS-CF on the Ball} \\ \hline 
    \multicolumn{3}{c}{} & \multicolumn{3}{c}{$\omega \equiv 1$} & & \multicolumn{3}{c}{$\omega(\boldsymbol{x}) = \sqrt{ \|\boldsymbol{x}\|_2}$} \\ \hline 
    $q$ &\multicolumn{2}{c}{} & equidistant & random & Halton & & equidistant & random & Halton \\ \hline 
	$2$ & s  & & 1.4 		& 8.8e-1 	& 1.8 		& & 1.4 		& 8.7e-1 	& 1.6 \\
    		   & C & & 1.2 		& 2.0e+1	& 2.3e-1	& & 1.2			& 2.2e+1	& 4.6e-1 \\ \hline 
	$3$ & s  & & 5.4e-1	& 1.0 		& 1.0			& & 4.6e-1	& 1.0			& 1.0 \\
    		   & C & & 5.0e+1	& 7.9			& 7.3			& & 7.5e+1	& 7.9			& 7.7 \\ \hline 
		   \\	  
    	\multicolumn{10}{c}{$\ell^1$-CF on the Ball} \\ \hline 
    \multicolumn{3}{c}{} & \multicolumn{3}{c}{$\omega \equiv 1$} & & \multicolumn{3}{c}{$\omega(\boldsymbol{x}) = \sqrt{ \|\boldsymbol{x}\|_2}$} \\ \hline 
    $q$ &\multicolumn{2}{c}{} & equidistant & random & Halton & & equidistant & random & Halton \\ \hline 
	$2$ & s  & & 1.5 		& 1.5       	& 1.1			& & 1.8 		&1.4			& 1.9 \\
    		   & C & &  5.4e-1 	& 5.6e-1	& 3.2			& & 1.8e-1 	& 9.4e-1	& 7.2e-2 \\ \hline 
	$3$ & s  & & 7.6e-1	& 1.0			& 1.6			& & 6.8e-1	& 1.5			& 1.2 \\
    		   & C & & 1.4e+1	& 3.2			& 1.5e-1	& & 2.2e+1	& 2.9e-1	& 1.3 \\ \hline 
\end{tabular} 
\caption{LS fit for the parameters $C$ and $s$ in the model $N = C K^s$}
\label{tab:LS-fit}
\end{table}

\begin{figure}[tb]
  \centering 
  \captionsetup{justification=centering}
  \begin{subfigure}[b]{0.33\textwidth}
    \includegraphics[width=\textwidth]{%
      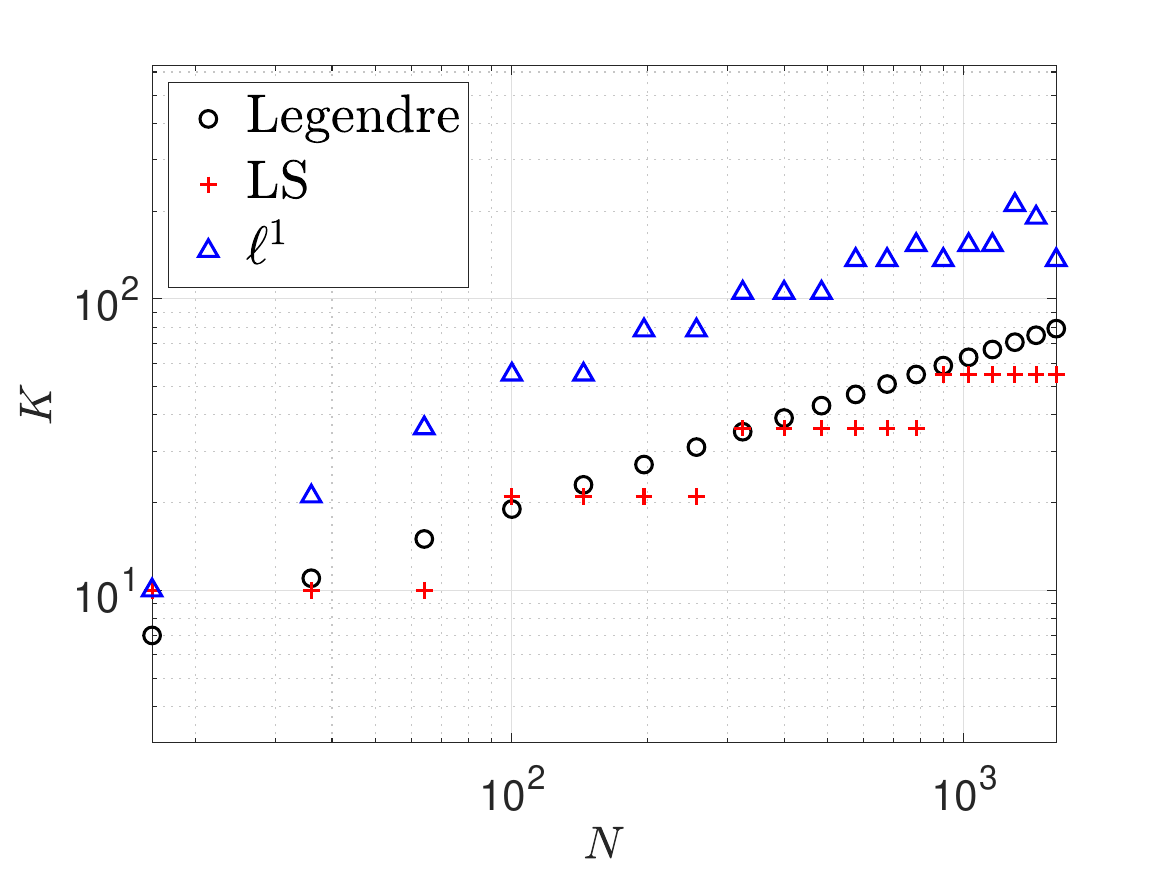} 
    \caption{$C_2$, equidistant points, $\omega \equiv 1$}
    \label{fig:ratio_cube_equid_dim2}
  \end{subfigure}%
  ~
  \begin{subfigure}[b]{0.33\textwidth}
    \includegraphics[width=\textwidth]{%
      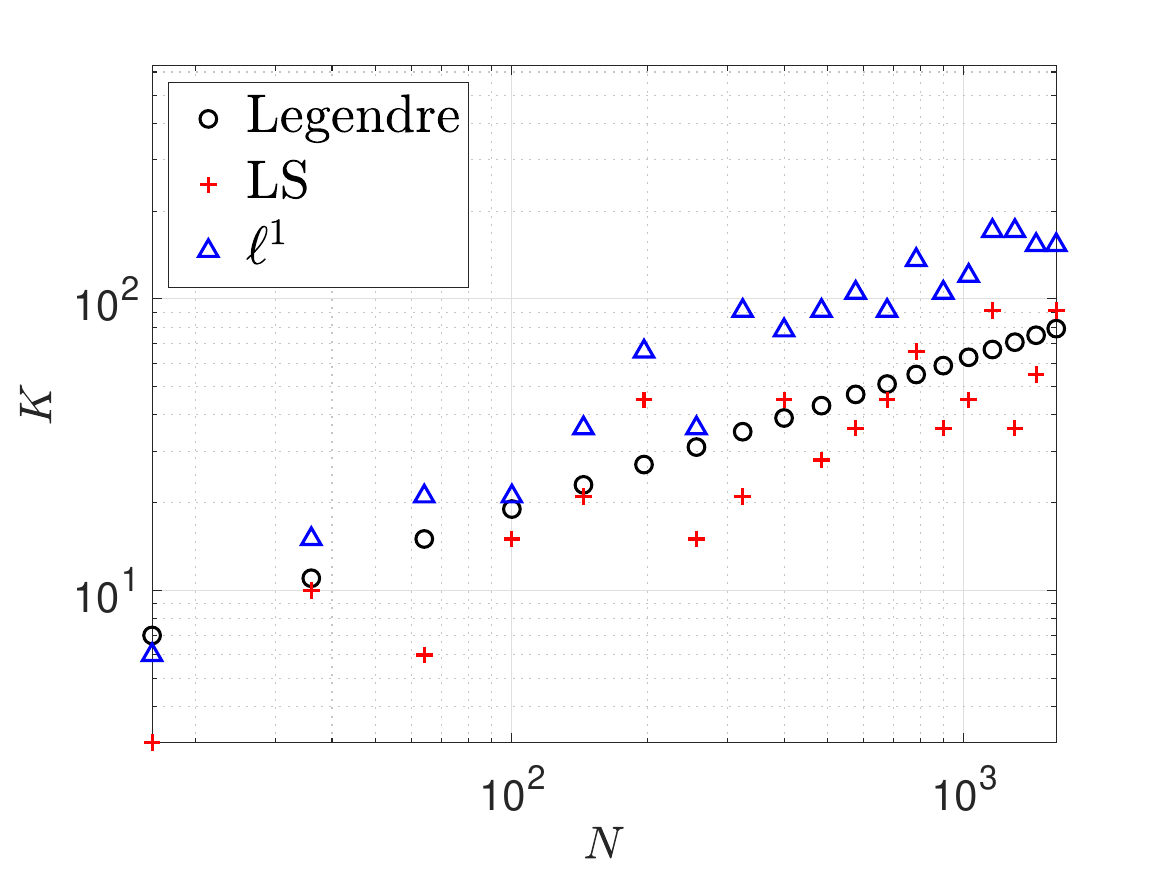} 
    \caption{$C_2$, random points, $\omega \equiv 1$}
    \label{fig:ratio_cube_uniform_dim2}
  \end{subfigure}%
  ~
  \begin{subfigure}[b]{0.33\textwidth}
    \includegraphics[width=\textwidth]{%
      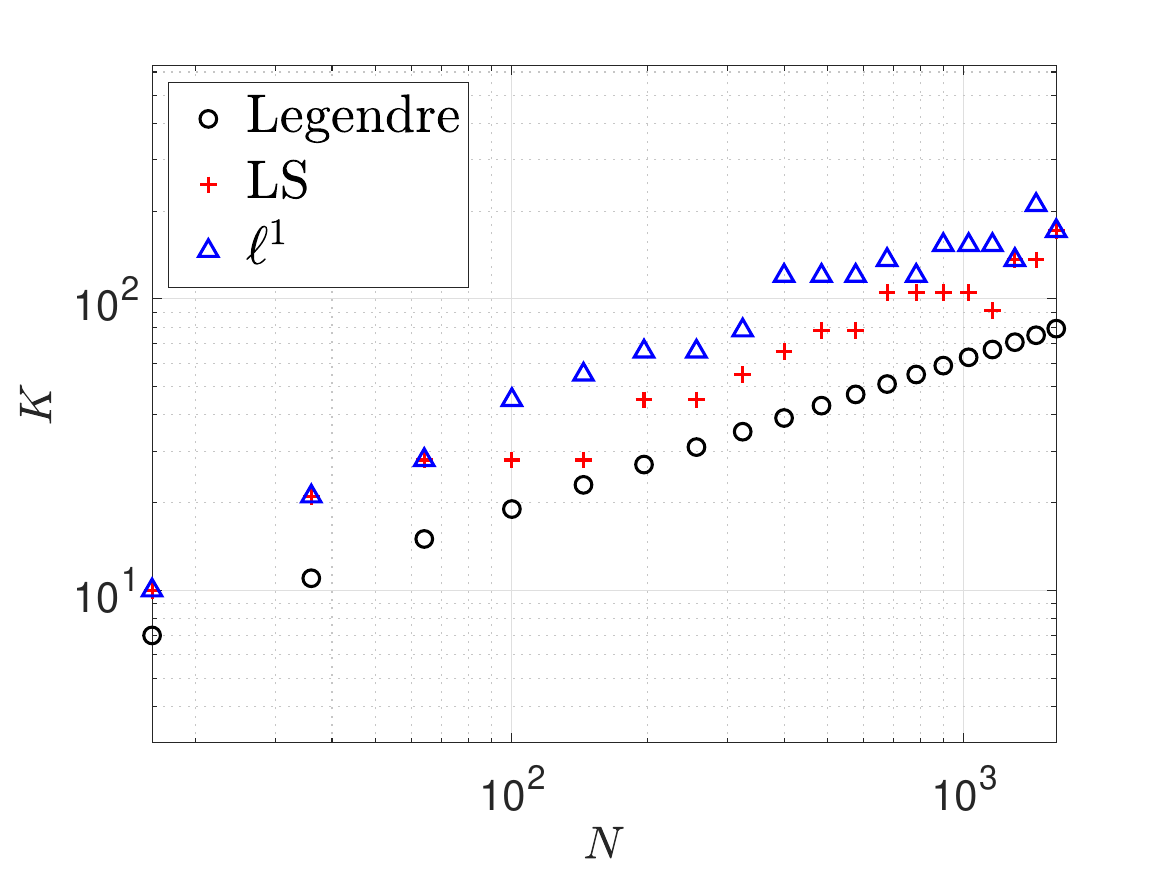} 
    \caption{$C_2$,  Halton points, $\omega \equiv 1$}
    \label{fig:ratio_cube_Halton_dim2}
  \end{subfigure}%
  \\ 
  \begin{subfigure}[b]{0.33\textwidth}
    \includegraphics[width=\textwidth]{%
      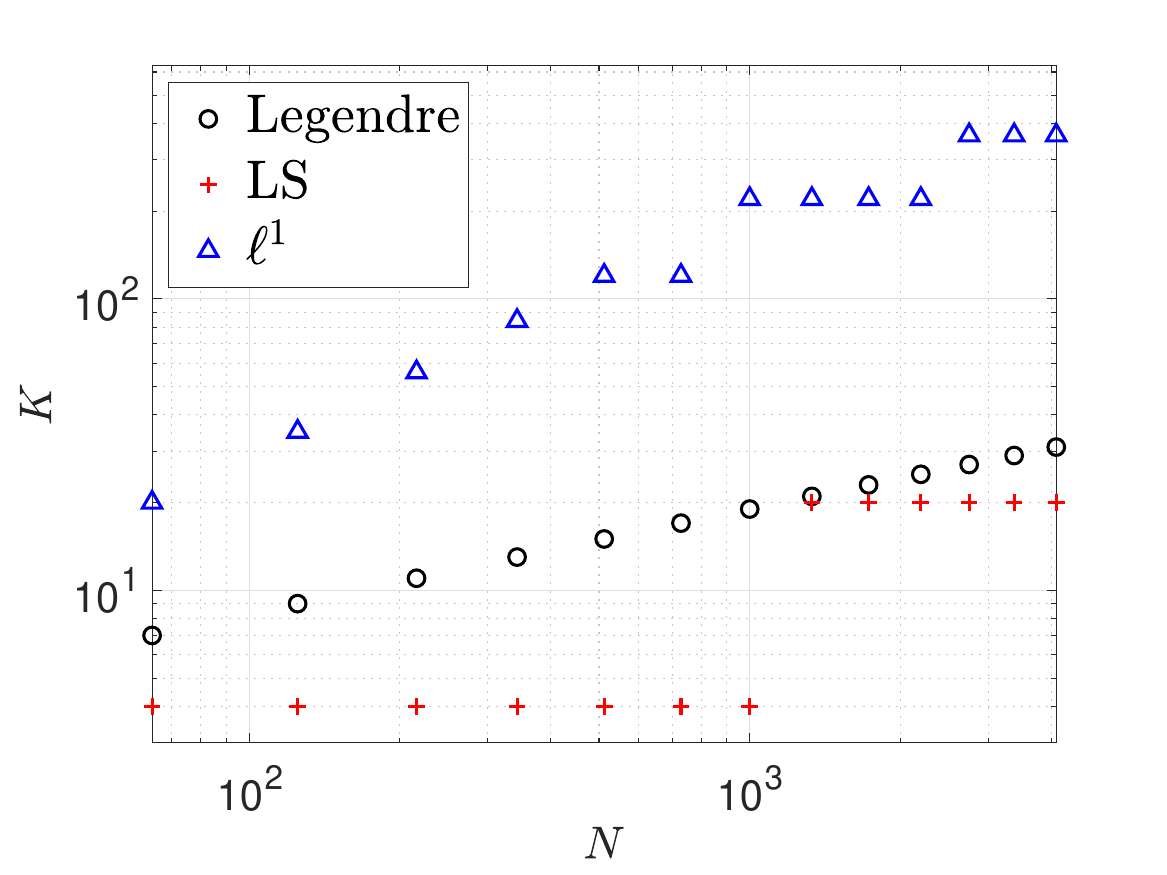} 
    \caption{$C_3$, equidistant points, $\omega \equiv 1$}
    \label{fig:ratio_cube_equid_dim3}
  \end{subfigure}%
  ~
  \begin{subfigure}[b]{0.33\textwidth}
    \includegraphics[width=\textwidth]{%
      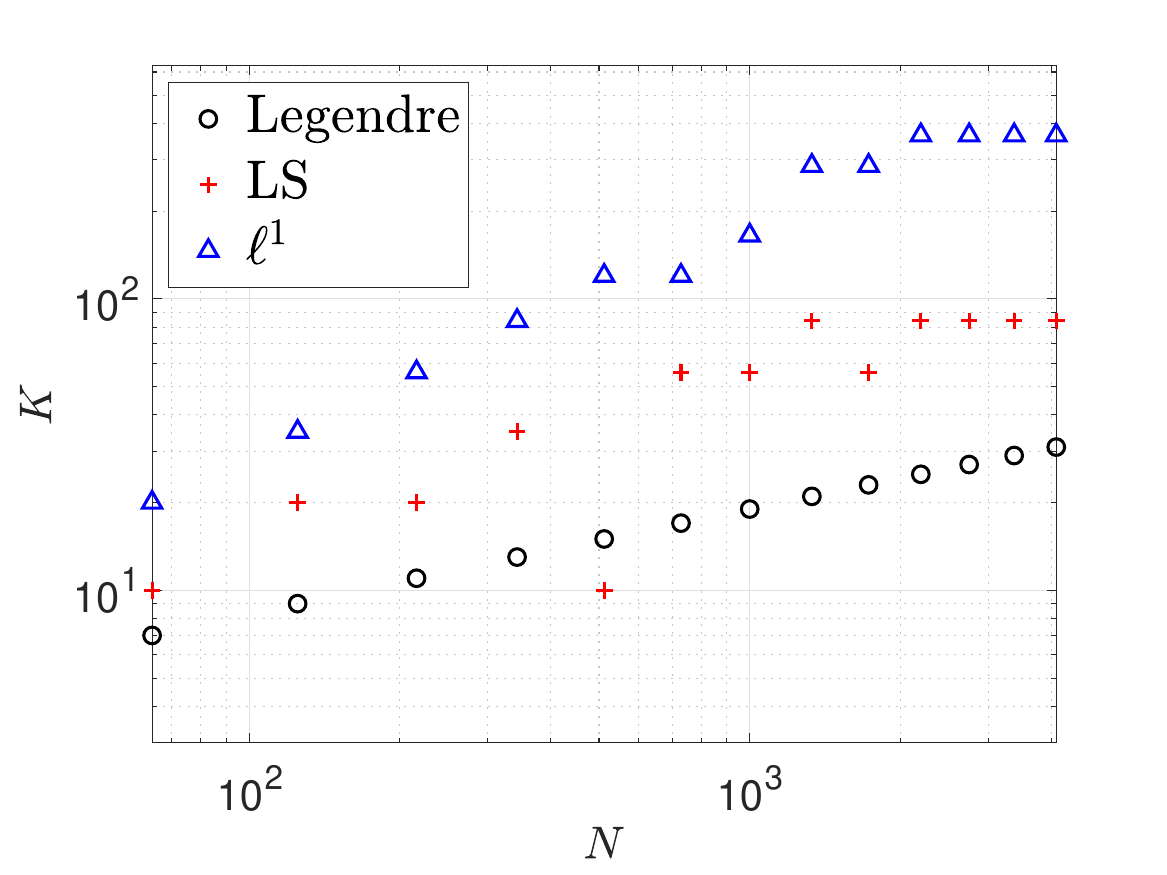} 
    \caption{$C_3$, random points, $\omega \equiv 1$}
    \label{fig:ratio_cube_uniform_dim3}
  \end{subfigure}%
  ~
  \begin{subfigure}[b]{0.33\textwidth}
    \includegraphics[width=\textwidth]{%
      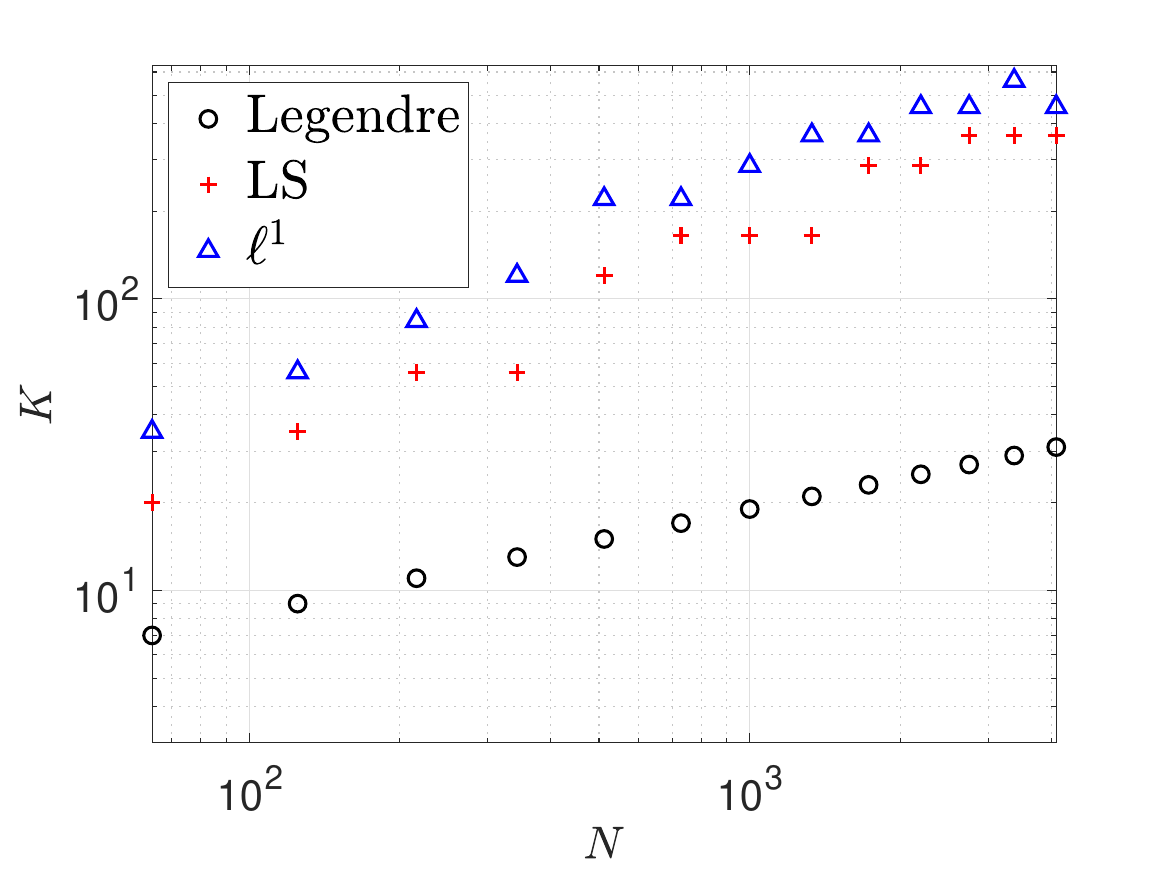} 
    \caption{$C_3$, Halton points, $\omega \equiv 1$}
    \label{fig:ratio_cube_Halton_dim3}
  \end{subfigure}%
  \\ 
  \begin{subfigure}[b]{0.33\textwidth}
    \includegraphics[width=\textwidth]{%
      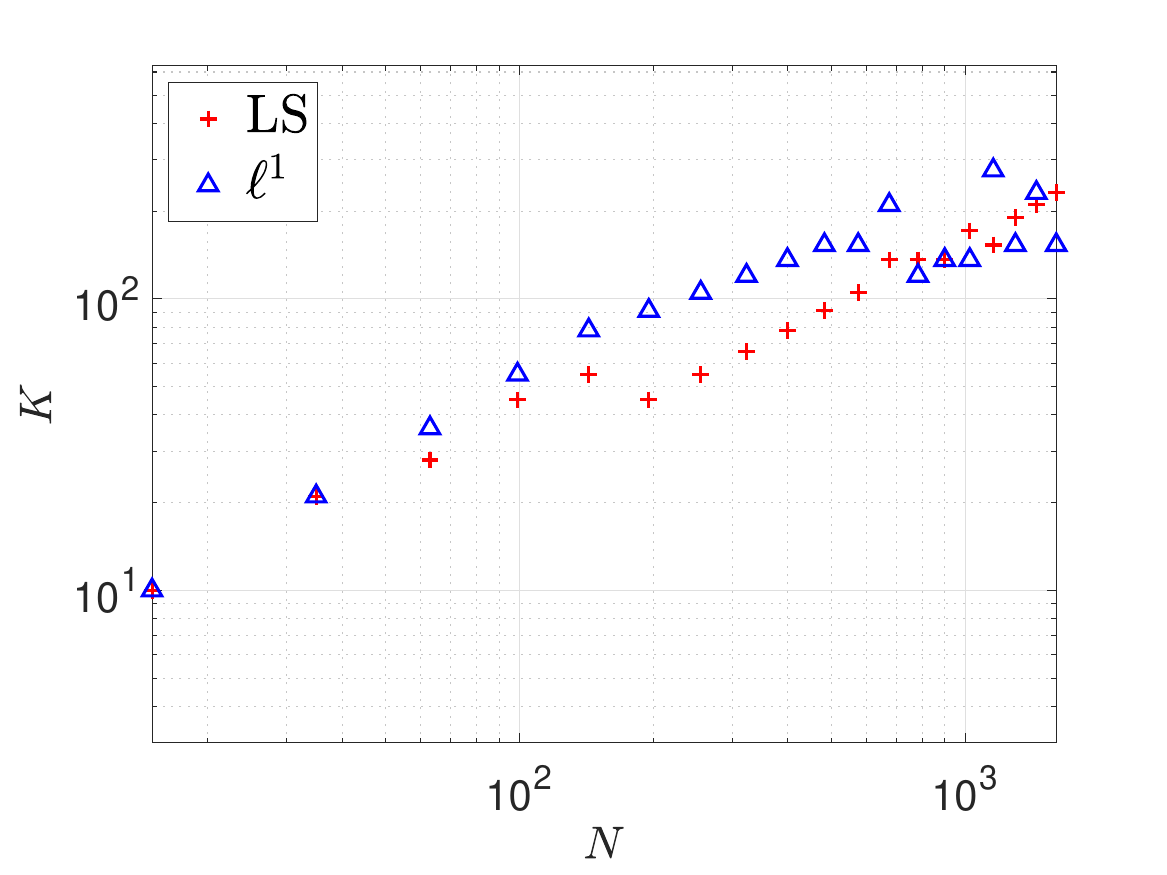} 
    \caption{$C_2$, Halton points, \\ $\omega(x,y) = (1-x^2)^{1/2} (1-y^2)^{1/2}$}
    \label{fig:ratio_cube_Halton_dim2_C2k}
  \end{subfigure}%
  ~
  \begin{subfigure}[b]{0.33\textwidth}
    \includegraphics[width=\textwidth]{%
      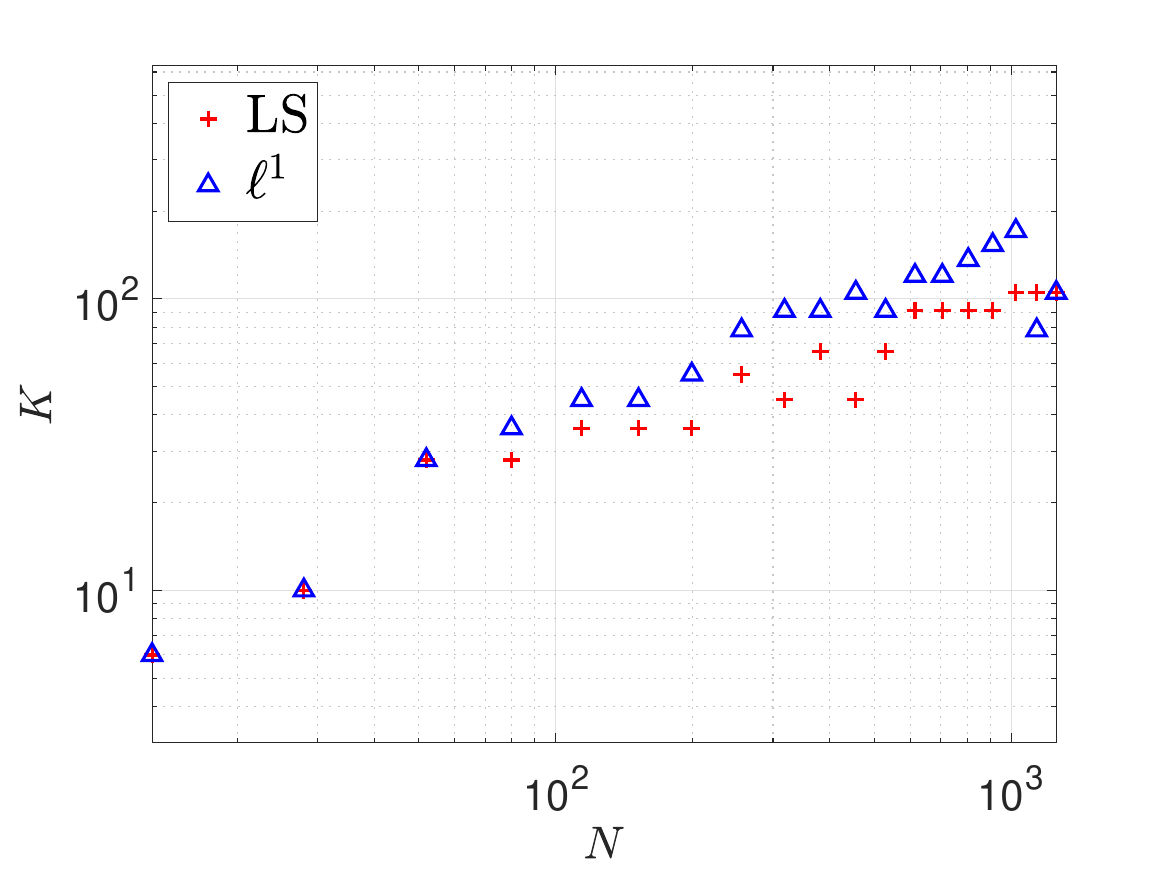} 
    \caption{$B_2$, Halton points, \\ $\omega \equiv 1$}
    \label{fig:ratio_ball_Halton_dim2_1}
  \end{subfigure}%
  ~
  \begin{subfigure}[b]{0.33\textwidth}
    \includegraphics[width=\textwidth]{%
      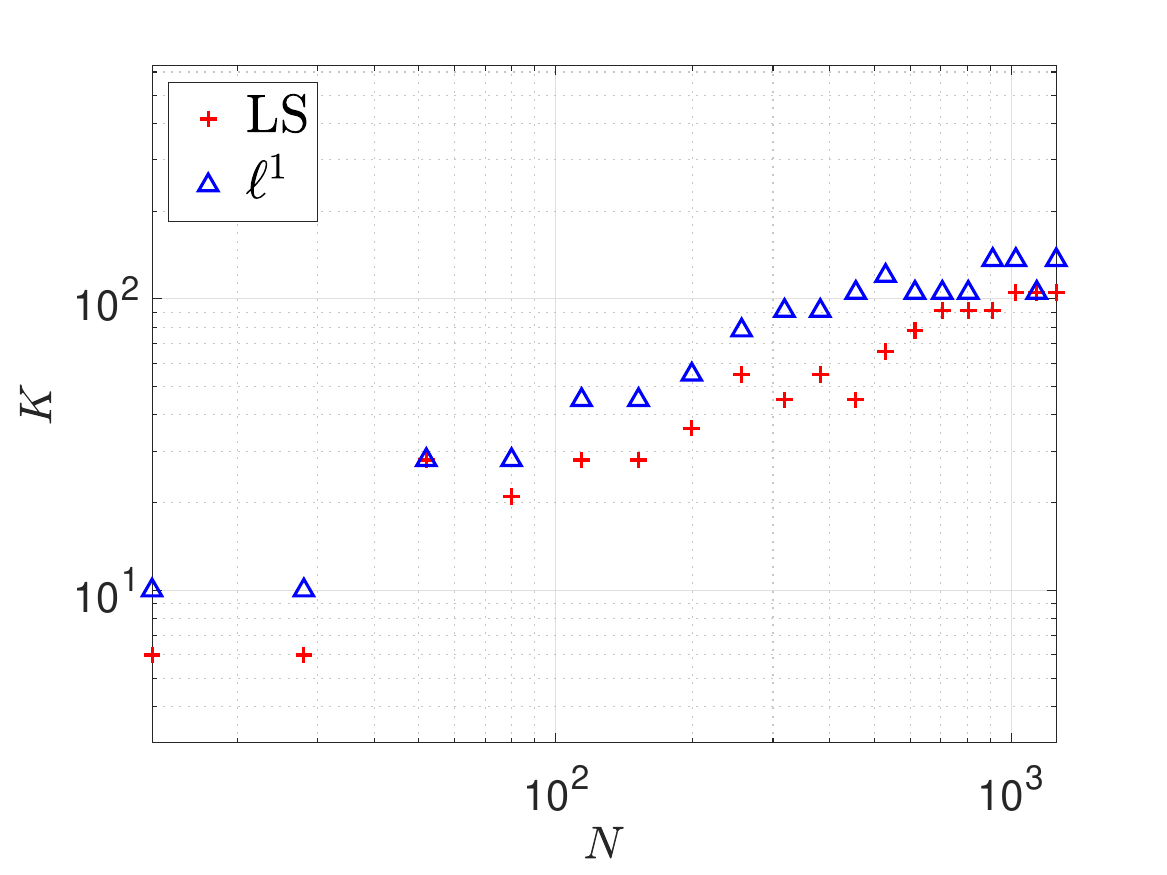} 
    \caption{$B_2$, Halton points, \\ $\omega(\boldsymbol{x}) = \sqrt{ \|\boldsymbol{x}\|_2}$}
    \label{fig:ratio_ball_Halton_dim2_sqrt}
  \end{subfigure}%
  \caption{$N$ versus $K$ for $C_2$, $C_3$, and $B_2$}
  \label{fig:ratio}
\end{figure}

Table \ref{tab:LS-fit} and Figure \ref{fig:ratio} report on the relation between $N$ and $K$ for the LS- and $\ell^1$-CF for $C_q$ and $B_q$ with $q=2,3$. 
The cube $C_q$ is considered together with the weight functions $\omega \equiv 1$ and ${\omega(\boldsymbol{x}) = \prod_{i=1}^q \sqrt{ 1 - x_i^2 }}$ (corresponding to products of  \emph{Chebyshev functions of second kind} \cite[Chapter 18]{dlmf2020digital}; apart from a multiplicative constant, also known as the \emph{Wigner semicircle distribution}).
For the ball $B_q$, the weight functions $\omega \equiv 1$ and ${ \omega(\boldsymbol{x}) = \sqrt{ \|\boldsymbol{x}\|_2} }$ are considered. 
The relation between $N$ and $K$ is assumed to be of the form $N \approx C K^s$. 
Here, the values for $s$ and $C$ have been determined numerically by performing an LS fit for the constants $C$ and $s$. 
Note that in almost all cases the $\ell^1$-CF achieves at least the same DoE---if not even a higher---compared to the LS-CF. 
This is in accordance with the $\ell^1$ weights to minimize the stability measure $\kappa$ while the LS weights minimize a weighted $2$-norm. 
Furthermore, for the cube and the weight function $\omega \equiv 1$, the asymptotic ratio of the LS- and $\ell^1$-CF is compared with the one of the product Legendre rule. 
This rule is known to provide DoE $d = 2n-1$ if $n$ Legendre points are used in every direction ($N=n^q$). 
Hence, $N \sim 2^{-q} d^q$ and therefore $N \sim q! 2^{-q} K$ for the product Legendre rule.  
We observe that the LS-CF and $\ell^1$-CF are observed to yield a smaller parameter $s$ than the product Legendre rule in many cases. 
This indicates that---asymptotically---these formulas require a smaller number of data points to achieve the same DoE. 
It should be stressed, however, that Cartesian products of Legendre (as well as other Gaussian) formulas appear to be more accurate in practice than one might expect based on their (total) DoE. 
This might be related to the fact that these formulas are not only exact for polynomials of total degree at most $d$, but also many others. 
Already in \cite[Remark 2]{haber1970numerical} it was argued that total DoE might not be suitable in higher dimensions: "Perhaps one should consider sets of monomials having the property that whenever $(x^1)^{n_1} \dots (x^s)^{n_s}$ is in the set, so is $(x^1)^{m_1} \dots (x^s)^{m_s}$ if $0 \leq m_i \leq n_i$, $i=1,\dots,s$;". 
Recently, this discussion has been revitalized in \cite{trefethen2017cubature,trefethen2017multivariate,trefethen2021exactness}. 
There, it was proposed to construct CFs on the hypercube based on the Euclidean rather than the total degree. 
Future work might provide a numerical investigation of this in the context of $\ell^1$- and LS-CFs. 

\newpage
\subsection{Accuracy for Exact Data}
\label{sub:num_exact}

Next, the accuracy of the proposed LS- and $\ell^1$-CF for two different test cases without any noise is investigated. 

\begin{figure}[tb]
  \centering 
  \captionsetup{justification=centering}
  \begin{subfigure}[b]{0.33\textwidth}
    \includegraphics[width=\textwidth]{%
      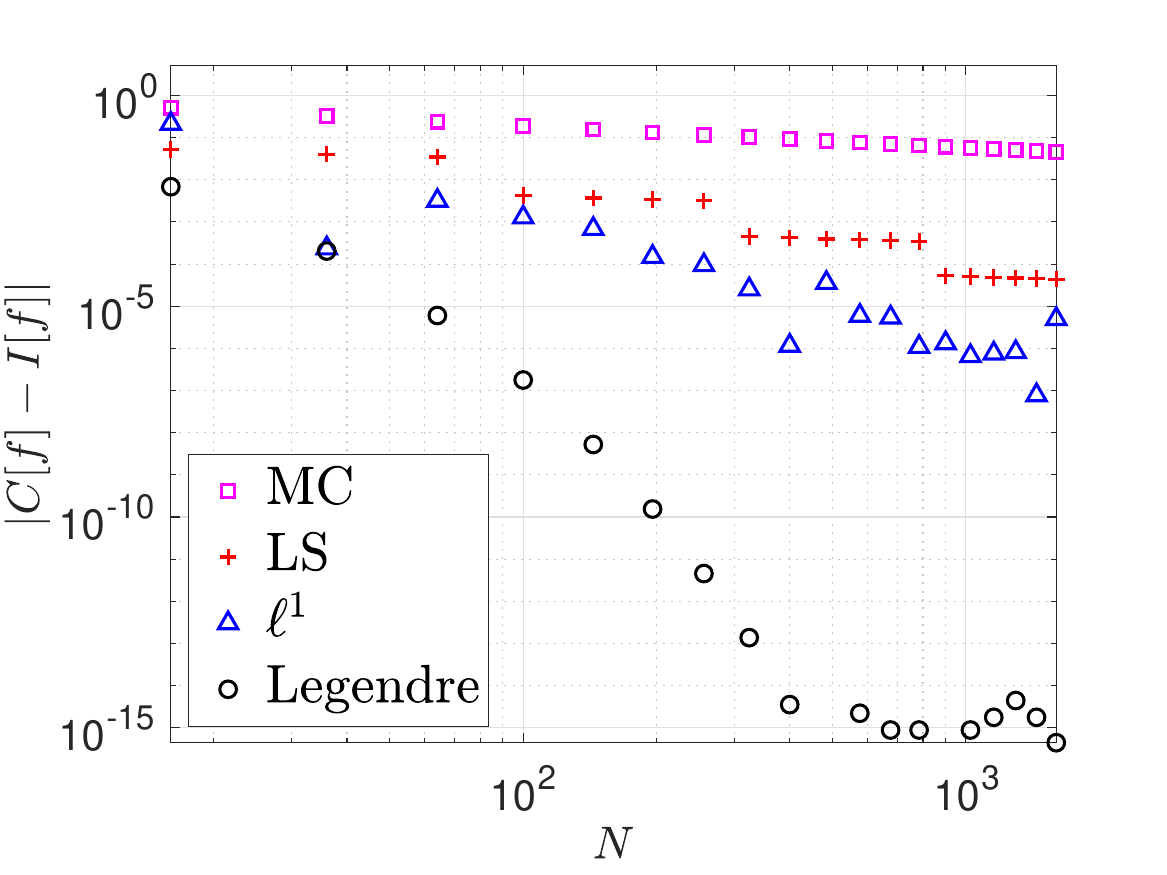} 
    \caption{$C_2$, equidistant points}
    \label{fig:accuracy_test1_dim2_equid}
  \end{subfigure}%
  ~
  \begin{subfigure}[b]{0.33\textwidth}
    \includegraphics[width=\textwidth]{%
      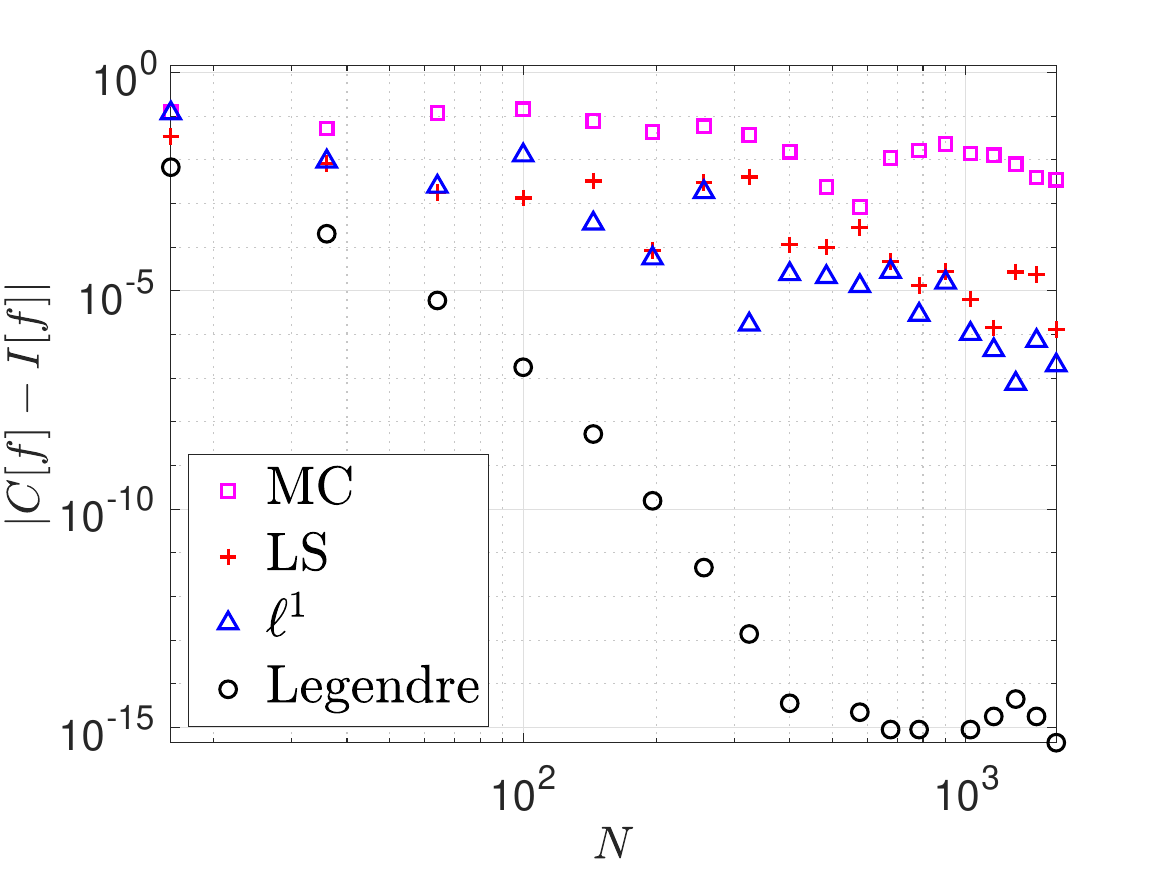} 
    \caption{$C_2$, random points}
    \label{fig:accuracy_test1_dim2_uniform}
  \end{subfigure}%
  ~
  \begin{subfigure}[b]{0.33\textwidth}
    \includegraphics[width=\textwidth]{%
      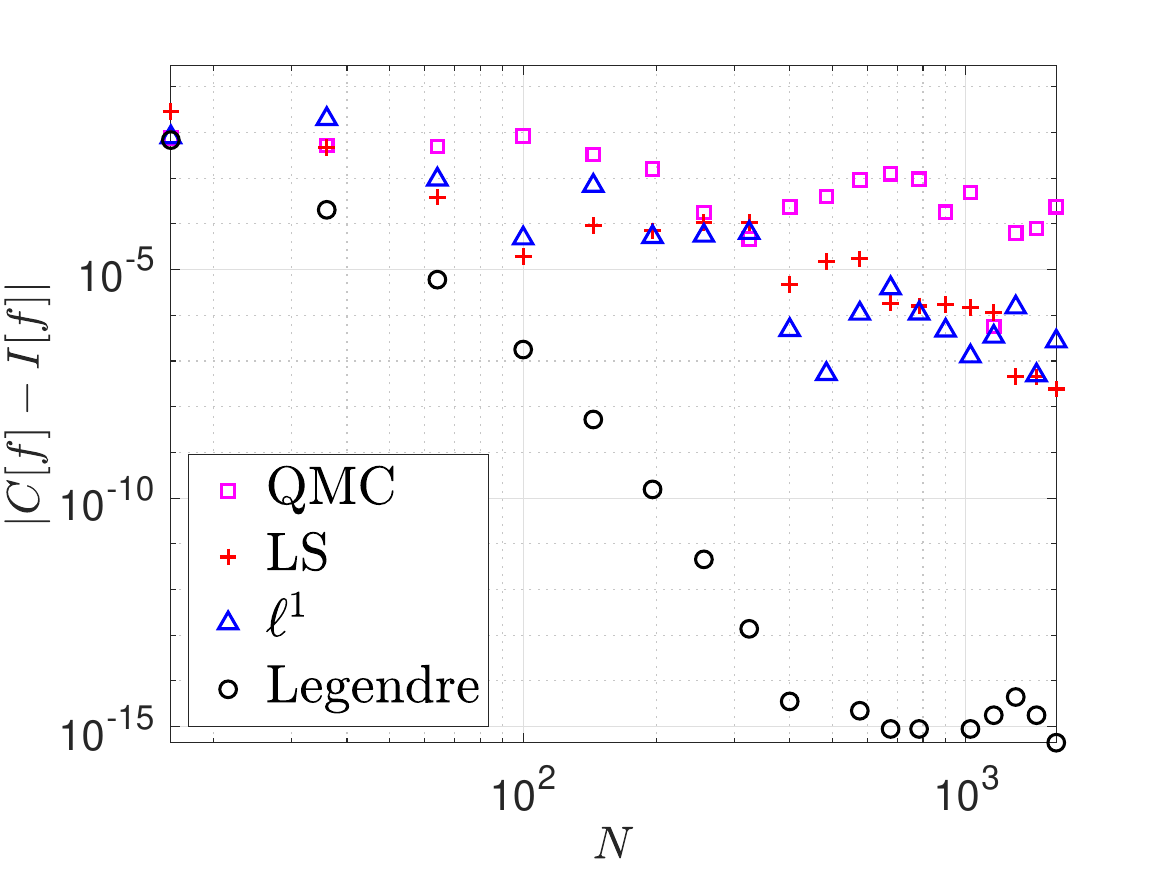} 
    \caption{$C_2$, Halton points}
    \label{fig:accuracy_test1_dim2_Halton}
  \end{subfigure}%
  \\
  \begin{subfigure}[b]{0.33\textwidth}
    \includegraphics[width=\textwidth]{%
      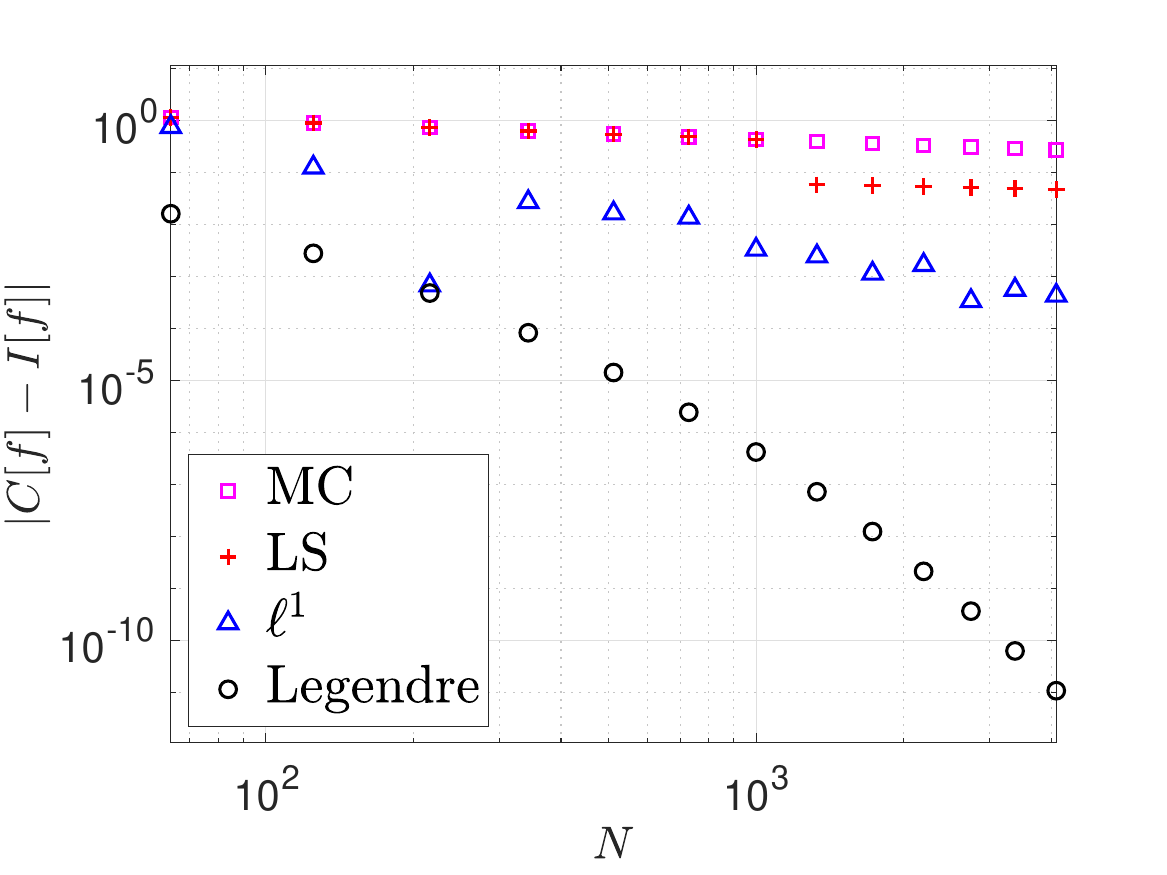} 
    \caption{$C_3$, equidistant points}
    \label{fig:accuracy_test1_dim3_equid}
  \end{subfigure}%
  ~
  \begin{subfigure}[b]{0.33\textwidth}
    \includegraphics[width=\textwidth]{%
      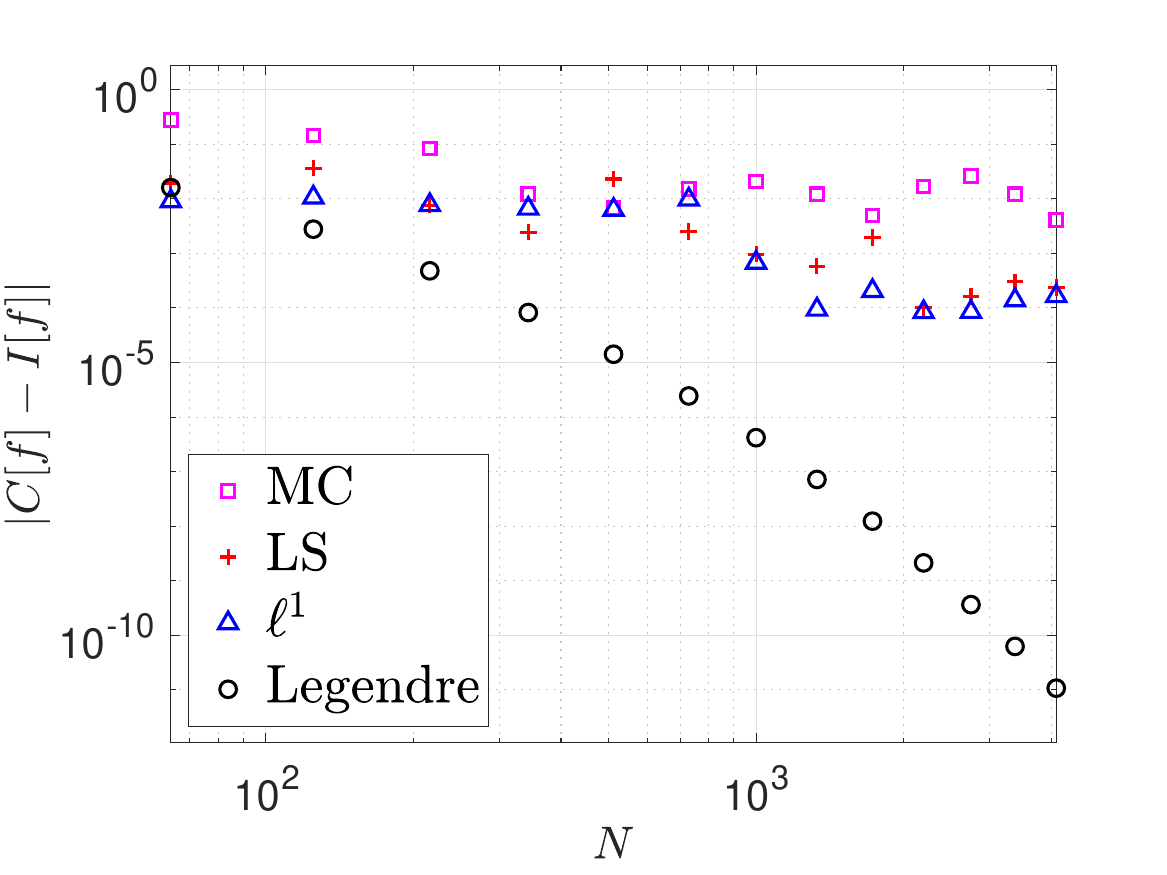} 
    \caption{$C_3$, random points}
    \label{fig:accuracy_test1_dim3_uniform}
  \end{subfigure}%
  ~
  \begin{subfigure}[b]{0.33\textwidth}
    \includegraphics[width=\textwidth]{%
      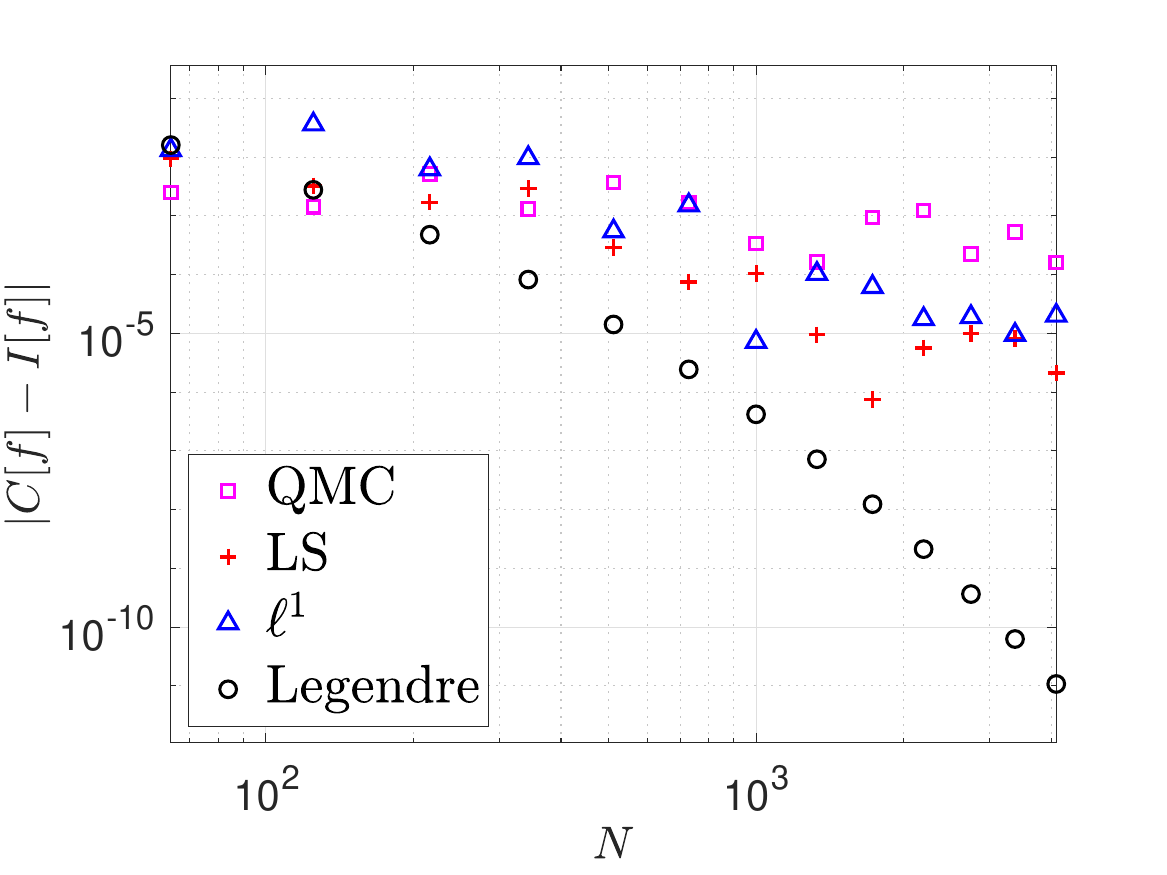} 
    \caption{$C_3$, Halton points}
    \label{fig:accuracy_test1_dim3_Halton}
  \end{subfigure}%
  \caption{Errors for $C_2$ and $C_3$ with $\omega$ and $f$ as in \eqref{eq:test1}.}
  \label{fig:accuracy_test1}
\end{figure}

\begin{figure}[tb]
  \centering 
  \captionsetup{justification=centering}
  \begin{subfigure}[b]{0.33\textwidth}
    \includegraphics[width=\textwidth]{%
      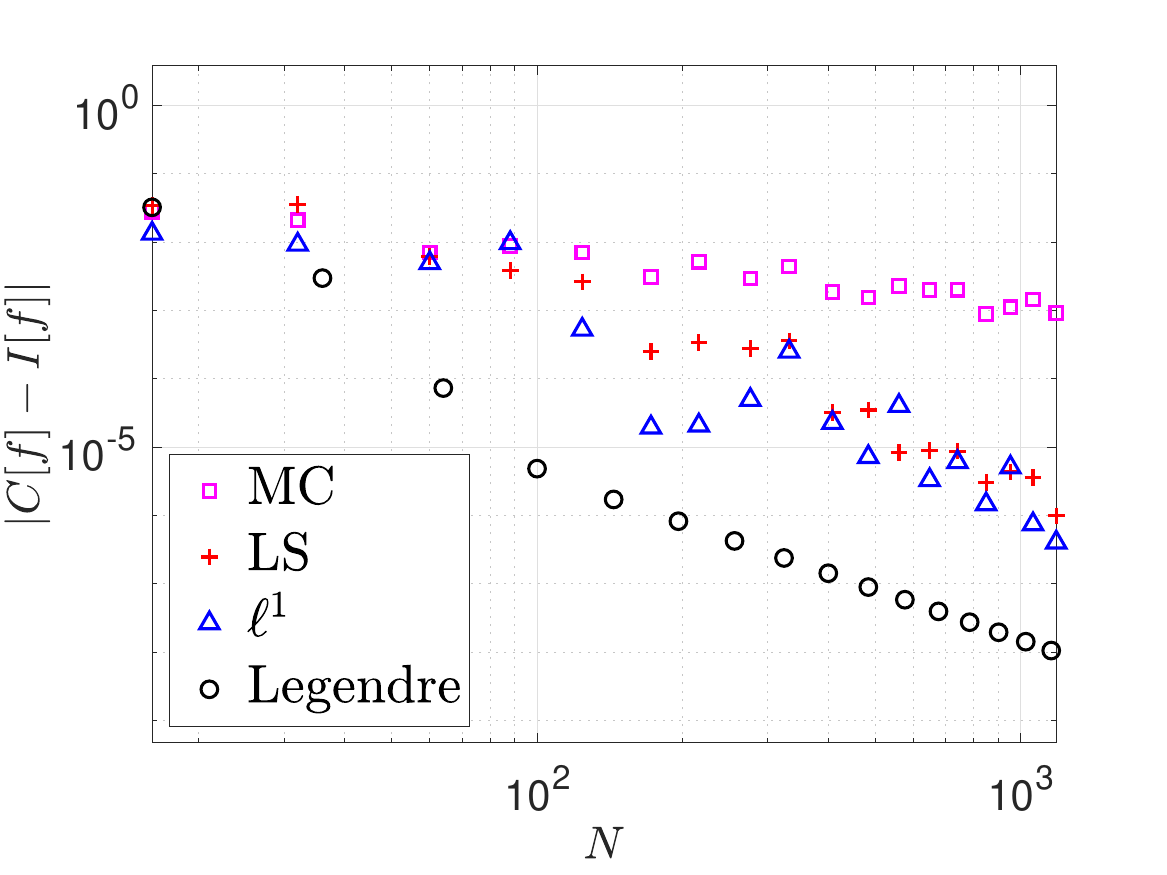} 
    \caption{$B_2$, equidistant points}
    \label{fig:accuracy_test2_dim2_equid}
  \end{subfigure}%
  ~
  \begin{subfigure}[b]{0.33\textwidth}
    \includegraphics[width=\textwidth]{%
      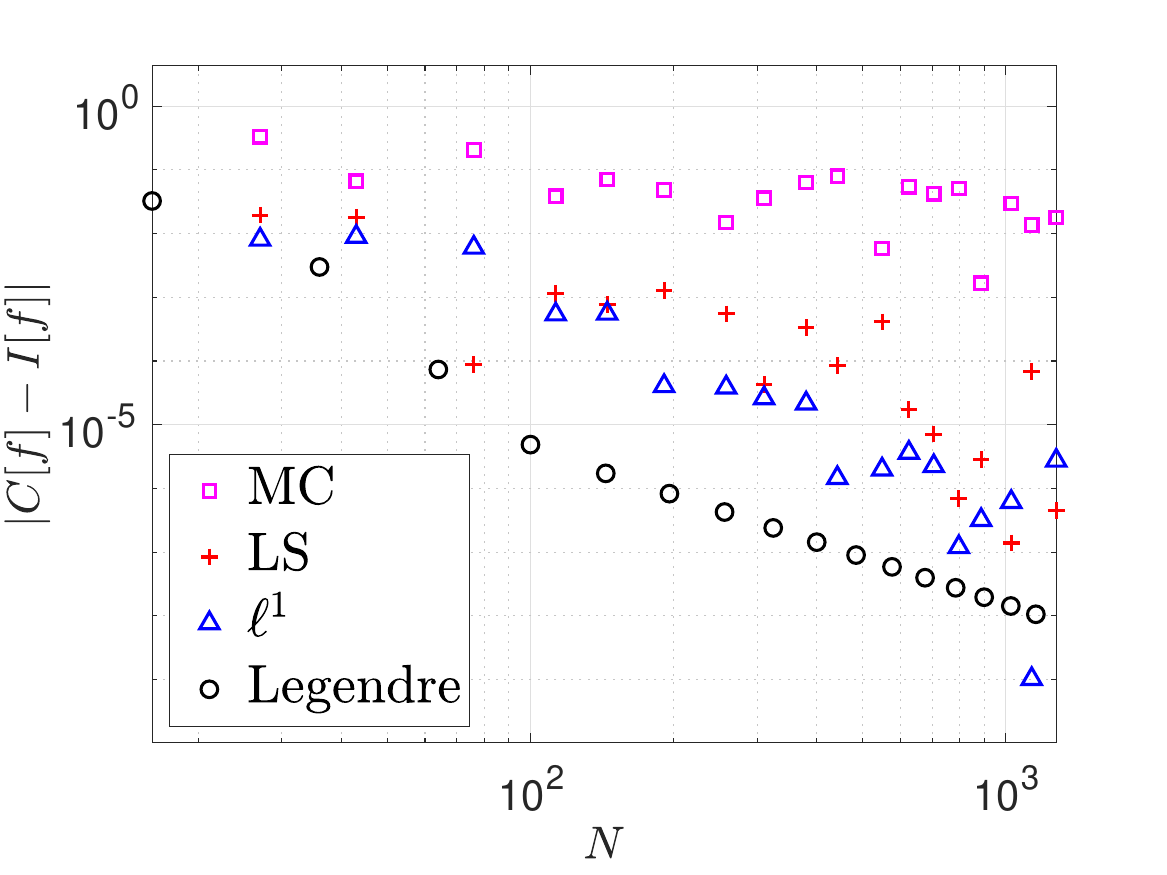} 
    \caption{$B_2$, random points}
    \label{fig:accuracy_test2_dim2_uniform}
  \end{subfigure}%
  ~
  \begin{subfigure}[b]{0.33\textwidth}
    \includegraphics[width=\textwidth]{%
      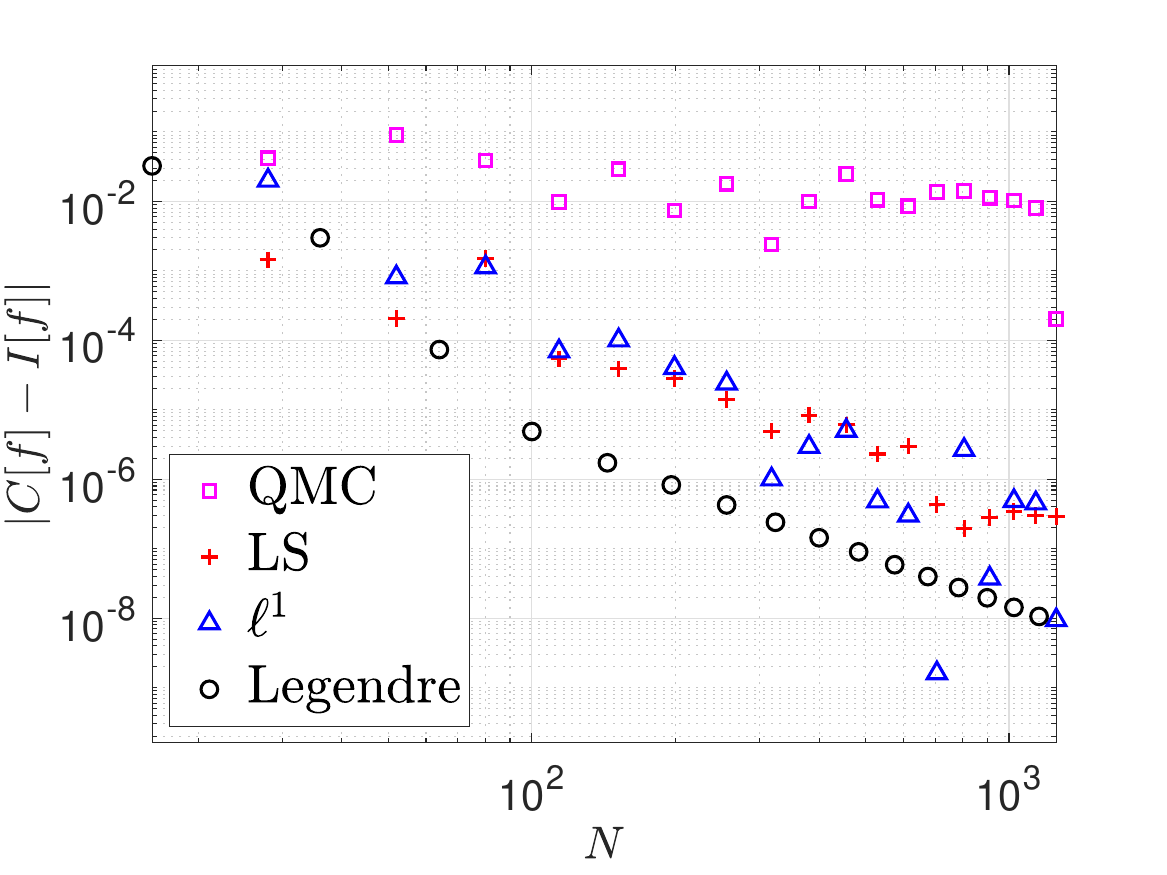} 
    \caption{$B_2$, Halton points}
    \label{fig:accuracy_test2_dim2_Halton}
  \end{subfigure}%
  \\
  \begin{subfigure}[b]{0.33\textwidth}
    \includegraphics[width=\textwidth]{%
      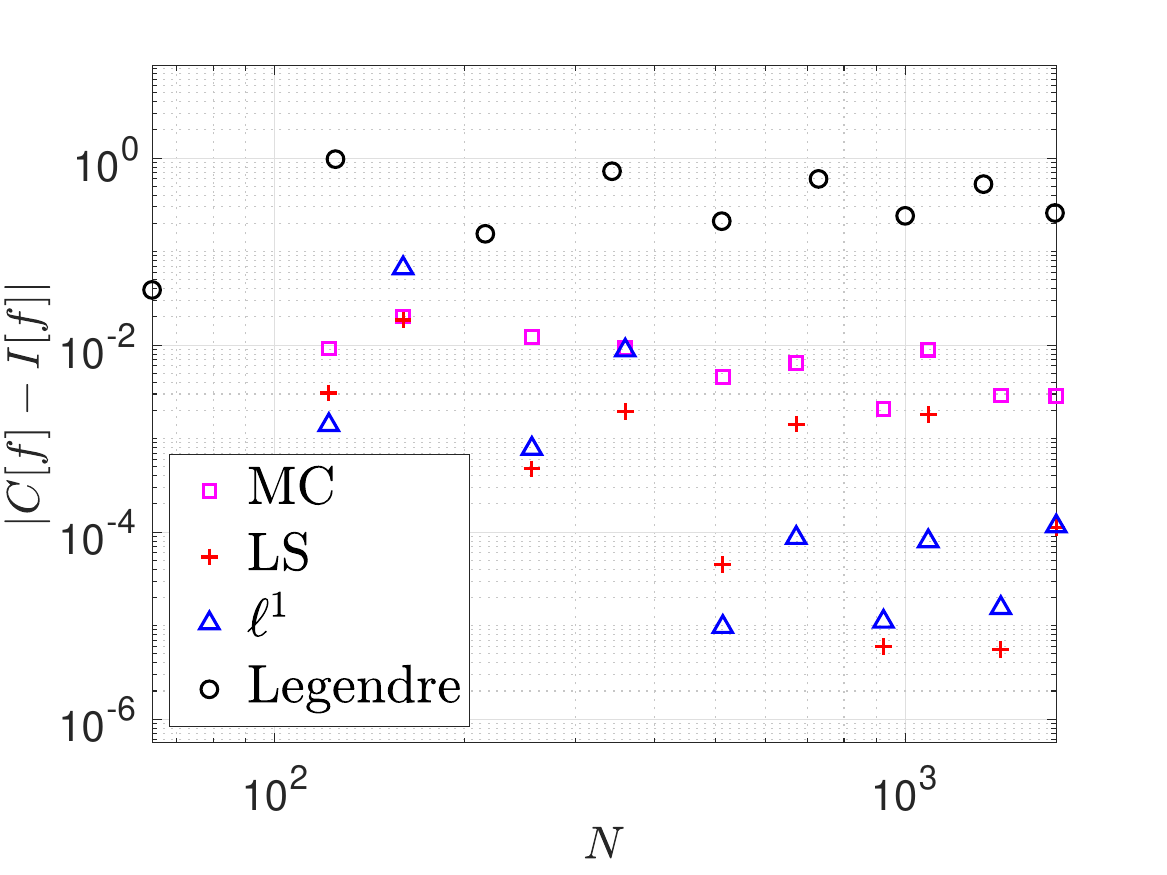} 
    \caption{$B_3$, equidistant points}
    \label{fig:accuracy_test2_dim3_equid}
  \end{subfigure}%
  ~
  \begin{subfigure}[b]{0.33\textwidth}
    \includegraphics[width=\textwidth]{%
      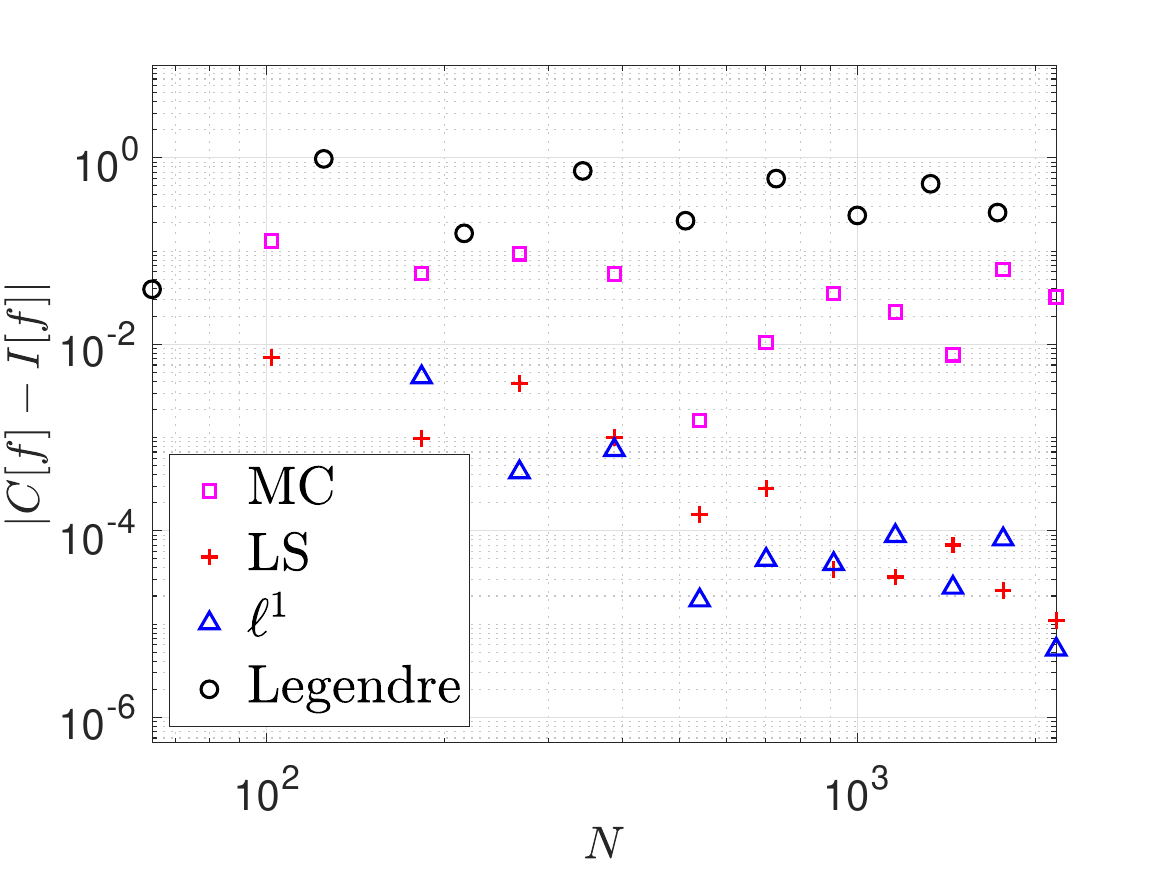} 
    \caption{$B_3$, random points}
    \label{fig:accuracy_test2_dim3_uniform}
  \end{subfigure}%
  ~
  \begin{subfigure}[b]{0.33\textwidth}
    \includegraphics[width=\textwidth]{%
      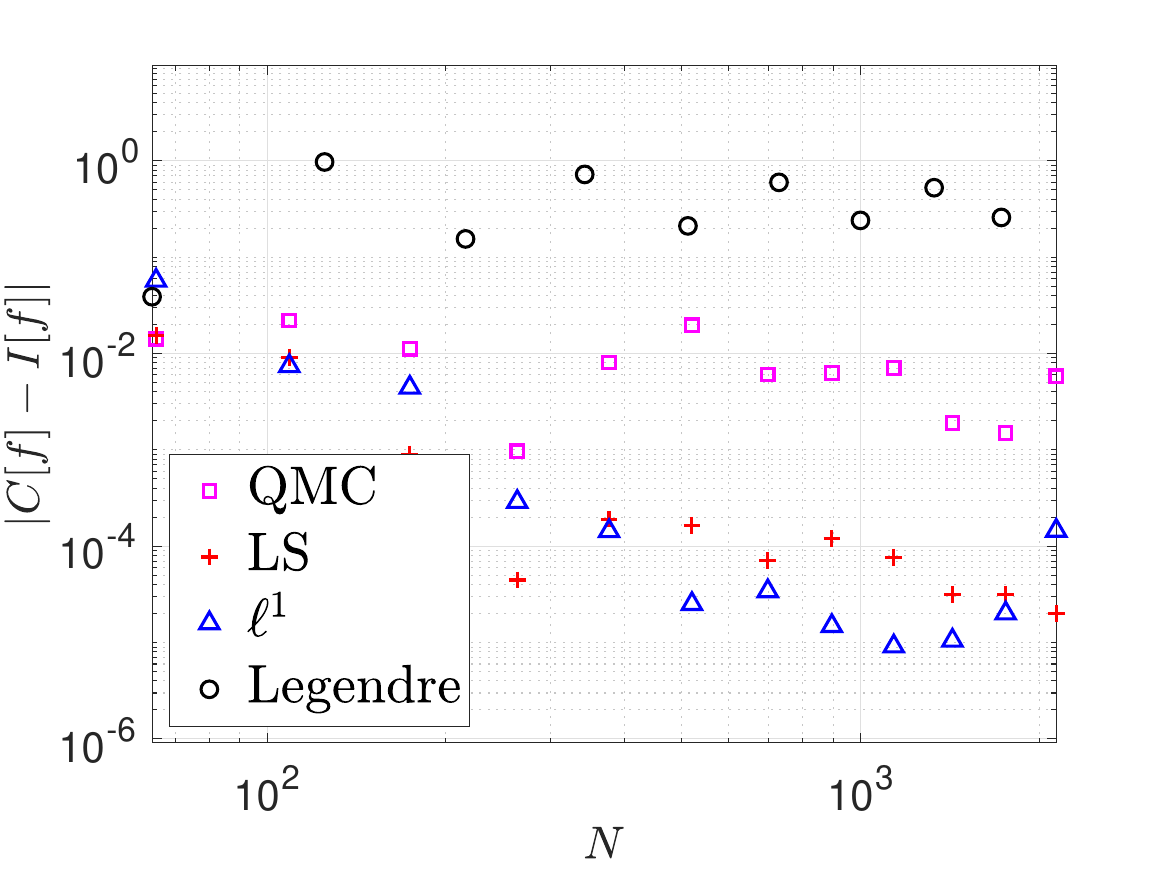} 
    \caption{$B_3$, Halton points}
    \label{fig:accuracy_test2_dim3_Halton}
  \end{subfigure}%
  \caption{Errors for $B_2$ and $B_3$ with $\omega$ and $f$ as in \eqref{eq:test2}.}
  \label{fig:accuracy_test2}
\end{figure}

In the first test case, we consider $B_2$ and $B_3$ with weight function $\omega$ and test function $f$ given by 
\begin{equation}\label{eq:test1}
	\omega \equiv 1, \quad 
	f(\boldsymbol{x}) = \frac{1}{(1 + x_1^2) \dots (1 + x_q^2)}, 
	\quad q=2,3.
\end{equation} 
In the second test case we consider $C_2$, $C_3$ together with 
\begin{equation}\label{eq:test2}
	\omega(\boldsymbol{x}) = \sqrt{ \|\boldsymbol{x}\|_2}, \quad 
	f(\boldsymbol{x}) = \frac{1}{1 + \|\boldsymbol{x}\|_2^2} + \sin(x_1).
\end{equation}
The results for equidistant, random, and Halton points are reported in Figure \ref{fig:accuracy_test1} (first test case) and Figure \ref{fig:accuracy_test2} (second test case).   
Besides the $\ell^1$- and LS-CF also the (quasi-)MC method (denoted by QMC for Halton points and MC otherwise) as well as the (transformed) Cartesian product Legendre formula are considered. 
The (quasi-)MC method is applied to the same set of data points as the LS- and $\ell^1$-CF. 
The (transformed) Cartesian product Legendre formula, on the other hand, requires a specific set of data points and is only included to provide a reference.

\subsection{Accuracy for Noisy Data}

The same test as in \S \ref{sub:num_exact}is considered. 
Yet, i.\,i.\,d.\ uniform noise supported on $[-10^{-6},10^{-6}]$ is added to the function values.
That is, noisy data $\mathbf{f}^{\epsilon}$ given by 
\begin{equation}
	f^{\epsilon}_n = f(\mathbf{x}_n) + Z_n, 
	\quad Z_n \in \mathcal{U}( -10^{-6},10^{-6} ), 
	\quad n=1,\dots,N,
\end{equation}
is considered. 
Thus, $\| \mathbf{f} - \mathbf{f}^{\epsilon} \|_{\infty} \leq 10^{-6}$. 
Moreover, the noise is assumed to not be correlated to the data point $\mathbf{x}_n$ or measurement $f(\mathbf{x}_n)$. 
In this case, none of the methods can be expected to yield an accuracy significantly lower than this uniform error bound. 

\begin{figure}[tb]
  \centering 
  \captionsetup{justification=centering}
  \begin{subfigure}[b]{0.33\textwidth}
    \includegraphics[width=\textwidth]{%
      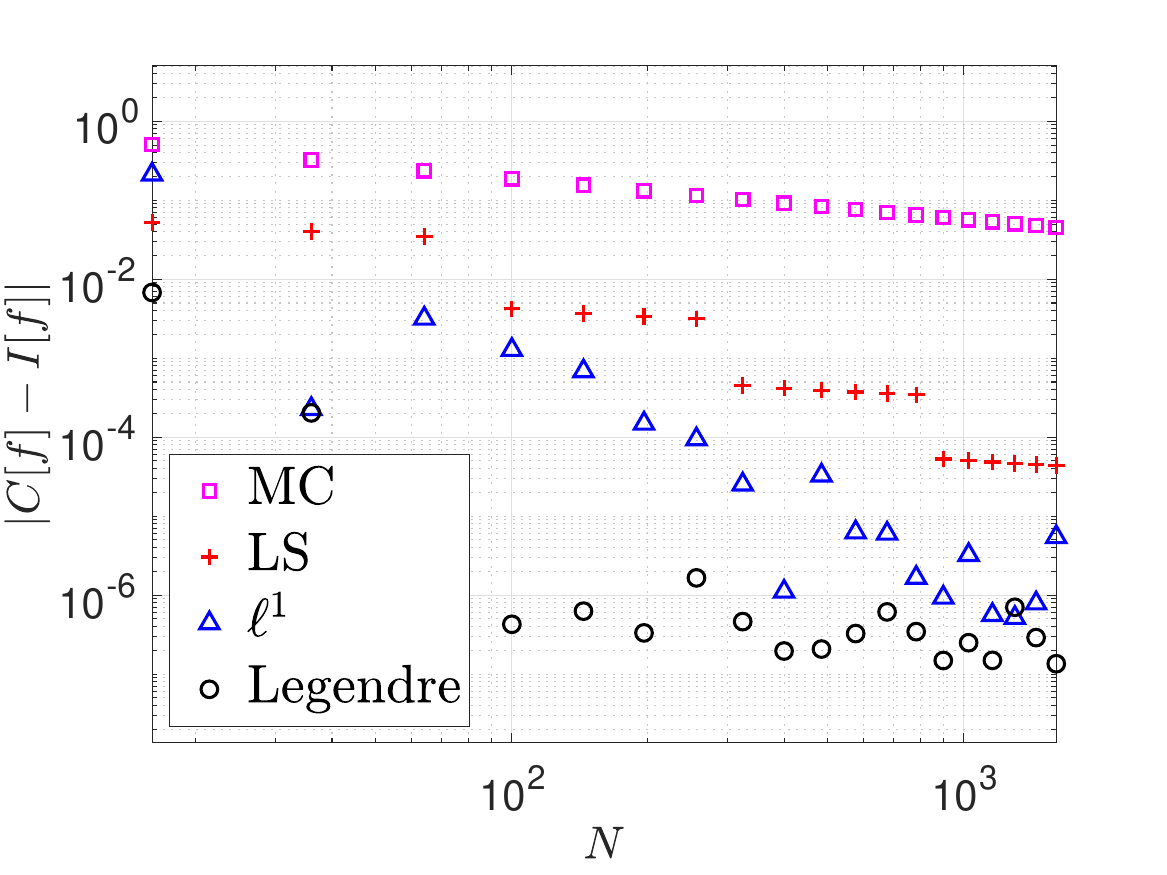} 
    \caption{$C_2$, equidistant points}
    \label{fig:accuracy_test1_noisy_dim2_equid}
  \end{subfigure}%
  ~
  \begin{subfigure}[b]{0.33\textwidth}
    \includegraphics[width=\textwidth]{%
      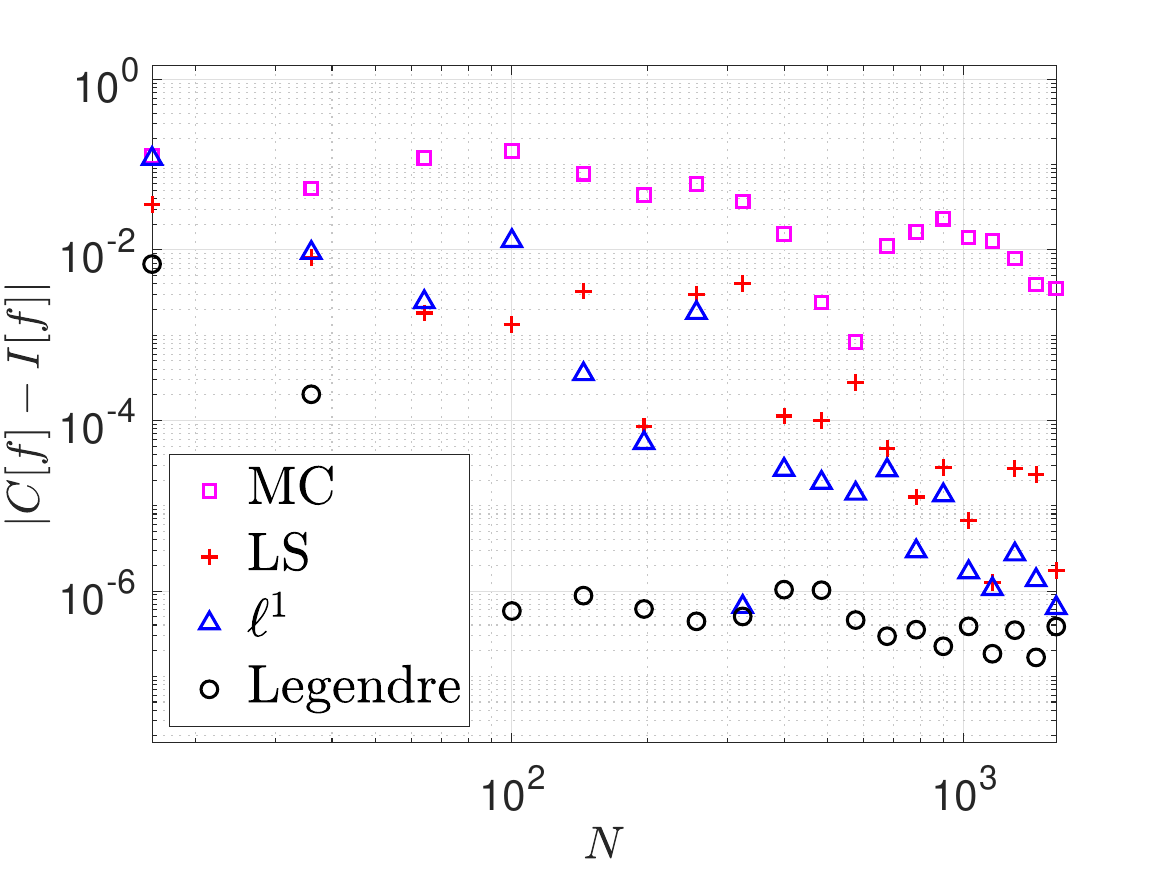} 
    \caption{$C_2$, random points}
    \label{fig:accuracy_test1_noisy_dim2_uniform}
  \end{subfigure}%
  ~
  \begin{subfigure}[b]{0.33\textwidth}
    \includegraphics[width=\textwidth]{%
      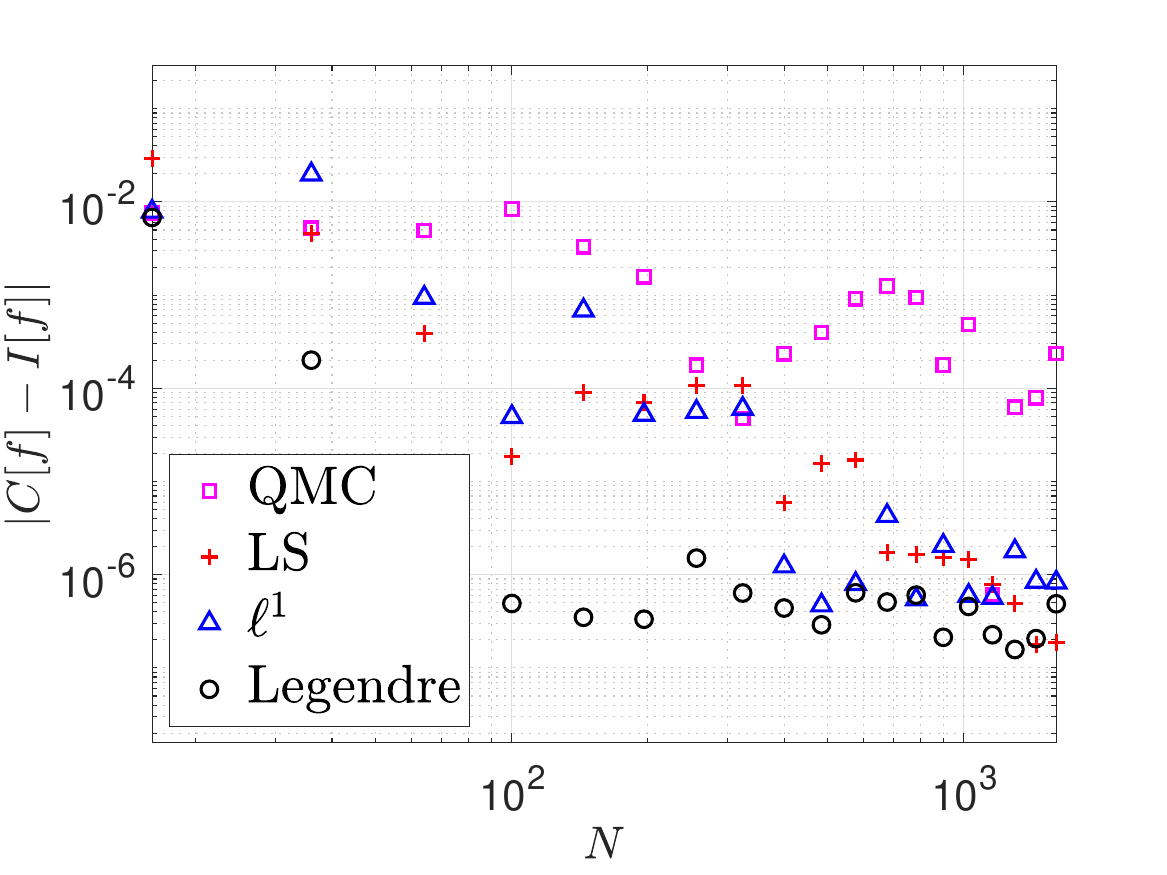} 
    \caption{$C_2$, Halton points}
    \label{fig:accuracy_test1_noisy_dim2_Halton}
  \end{subfigure}%
  \\
  \begin{subfigure}[b]{0.33\textwidth}
    \includegraphics[width=\textwidth]{%
      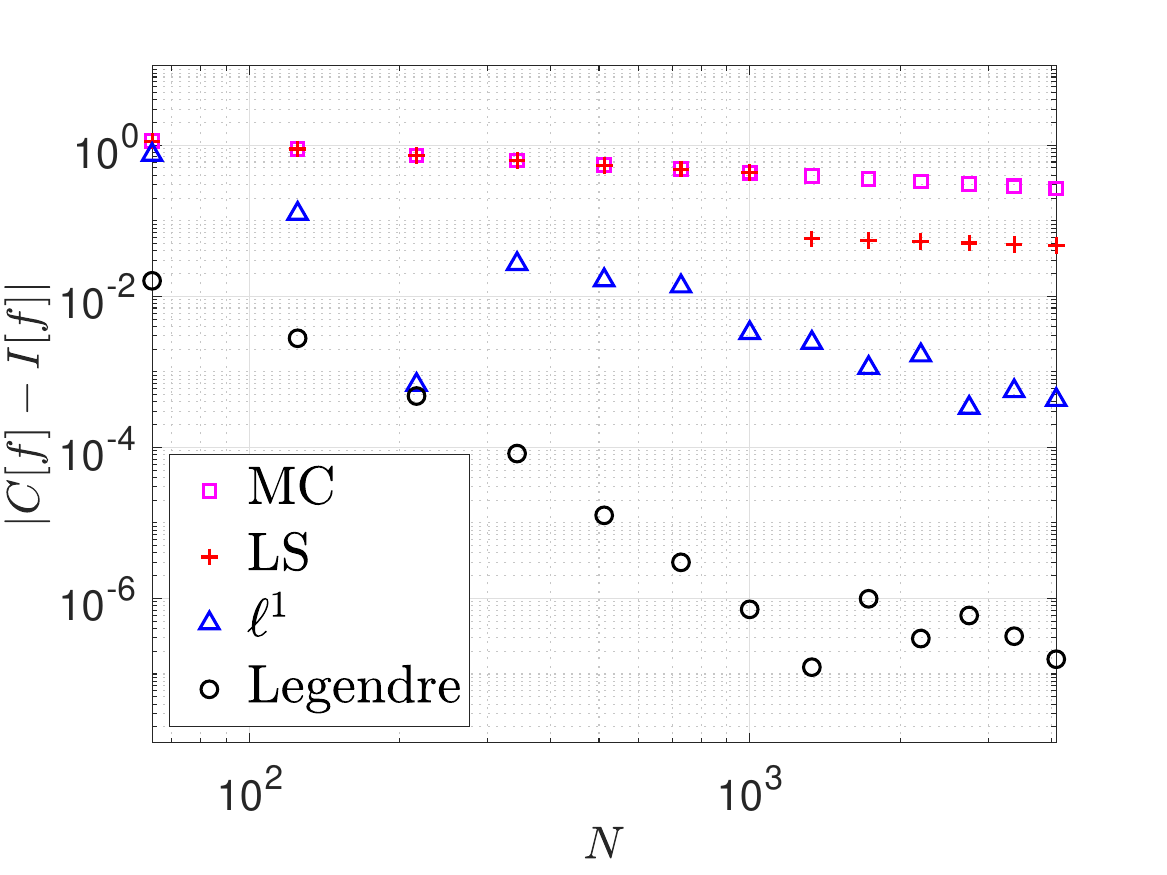} 
    \caption{$C_3$, equidistant points}
    \label{fig:accuracy_test1_noisy_dim3_equid}
  \end{subfigure}%
  ~
  \begin{subfigure}[b]{0.33\textwidth}
    \includegraphics[width=\textwidth]{%
      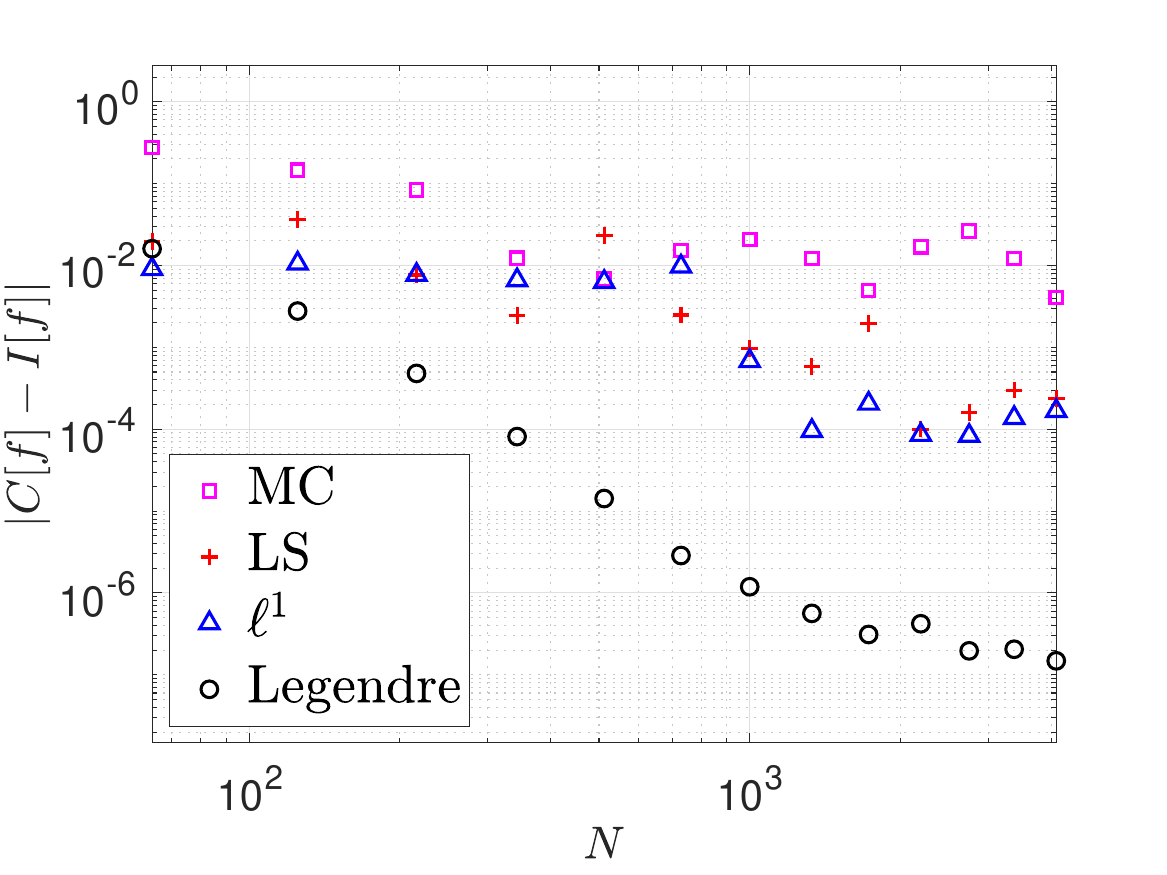} 
    \caption{$C_3$, random points}
    \label{fig:accuracy_test1_noisy_dim3_uniform}
  \end{subfigure}%
  ~
  \begin{subfigure}[b]{0.33\textwidth}
    \includegraphics[width=\textwidth]{%
      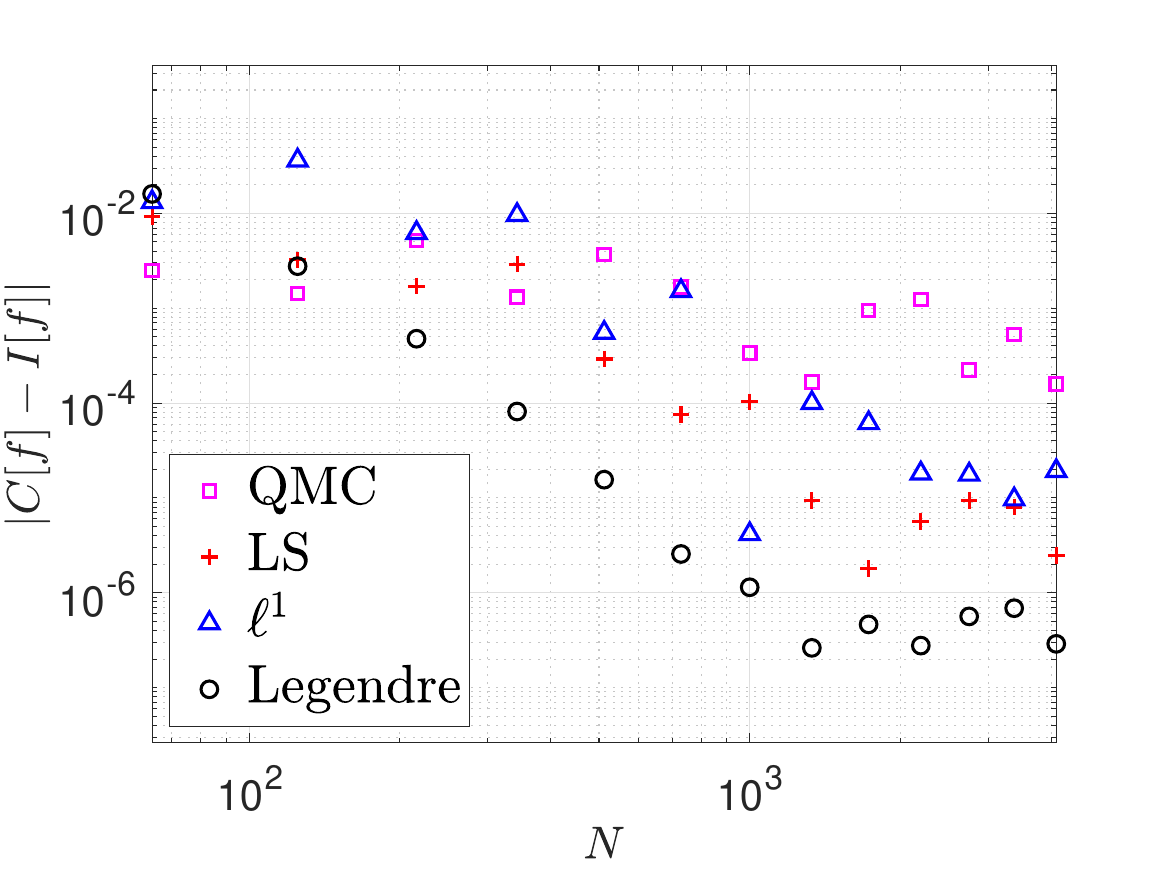} 
    \caption{$C_3$, Halton points}
    \label{fig:accuracy_test1_noisy_dim3_Halton}
  \end{subfigure}%
  \caption{Errors for $C_2$ and $C_3$ with $\omega$ and $f$ as in \eqref{eq:test1}. 
  I.\,i.\,d.\ uniform noise ${Z_n \in \mathcal{U}( -10^{-6},10^{-6} )}$ was added.}
  \label{fig:accuracy_test1_noisy}
\end{figure}

\begin{figure}[tb]
  \centering 
  \captionsetup{justification=centering}
  \begin{subfigure}[b]{0.33\textwidth}
    \includegraphics[width=\textwidth]{%
      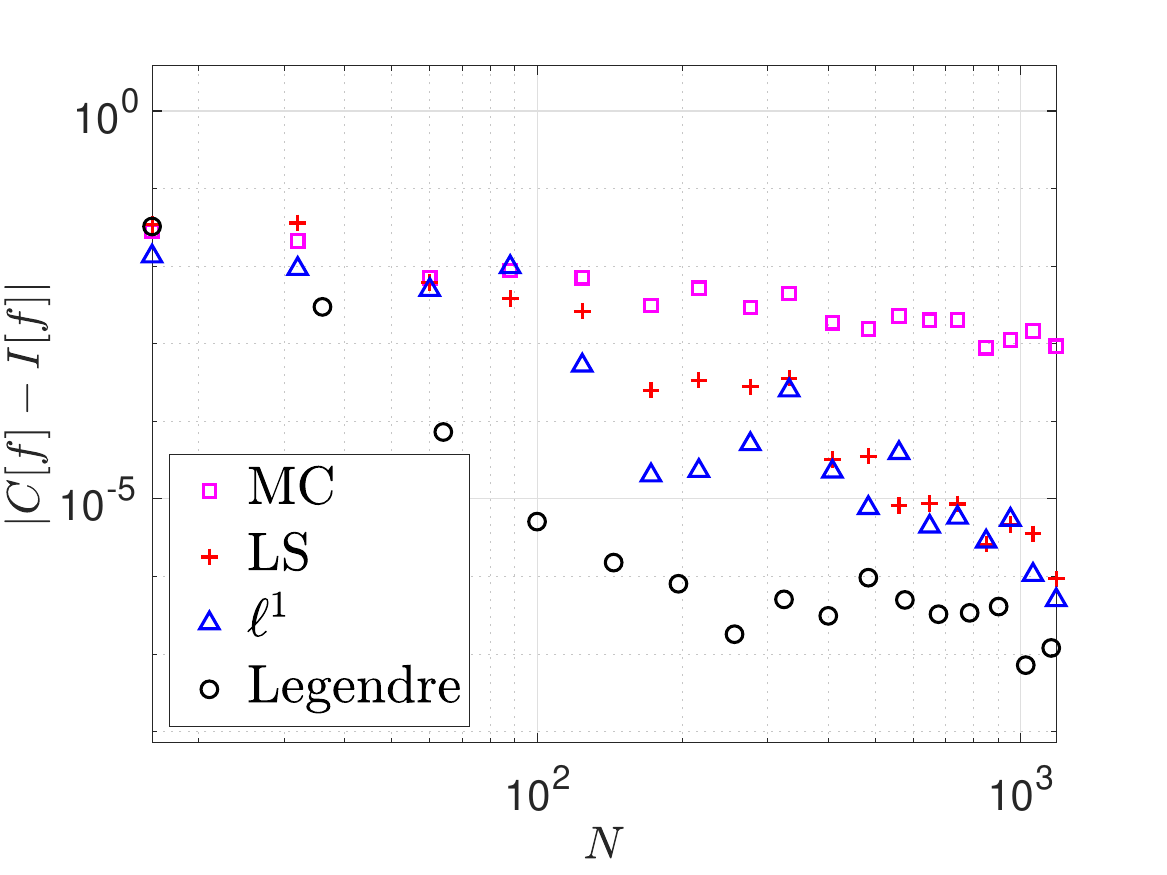} 
    \caption{$B_2$, equidistant points}
    \label{fig:accuracy_test2_noisy_dim2_equid}
  \end{subfigure}%
  ~
  \begin{subfigure}[b]{0.33\textwidth}
    \includegraphics[width=\textwidth]{%
      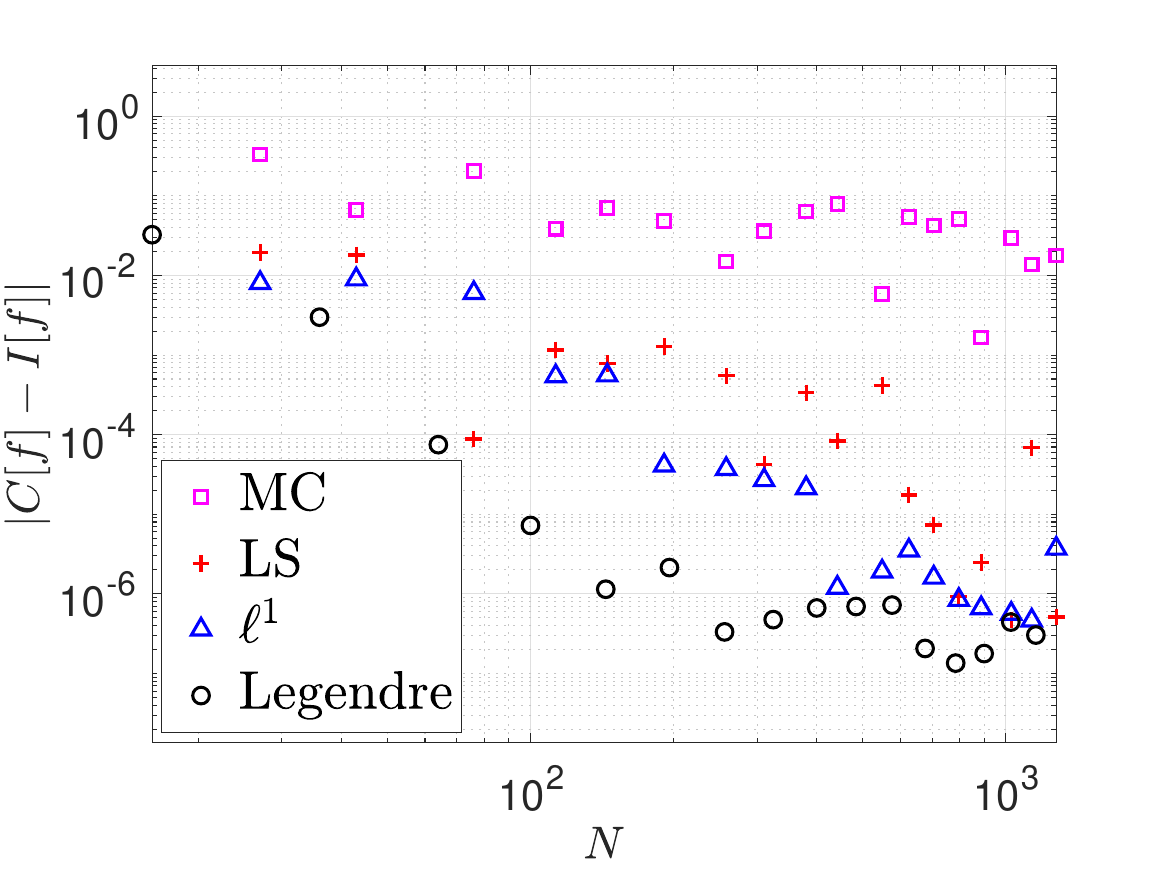} 
    \caption{$B_2$, random points}
    \label{fig:accuracy_test2_noisy_dim2_uniform}
  \end{subfigure}%
  ~
  \begin{subfigure}[b]{0.33\textwidth}
    \includegraphics[width=\textwidth]{%
      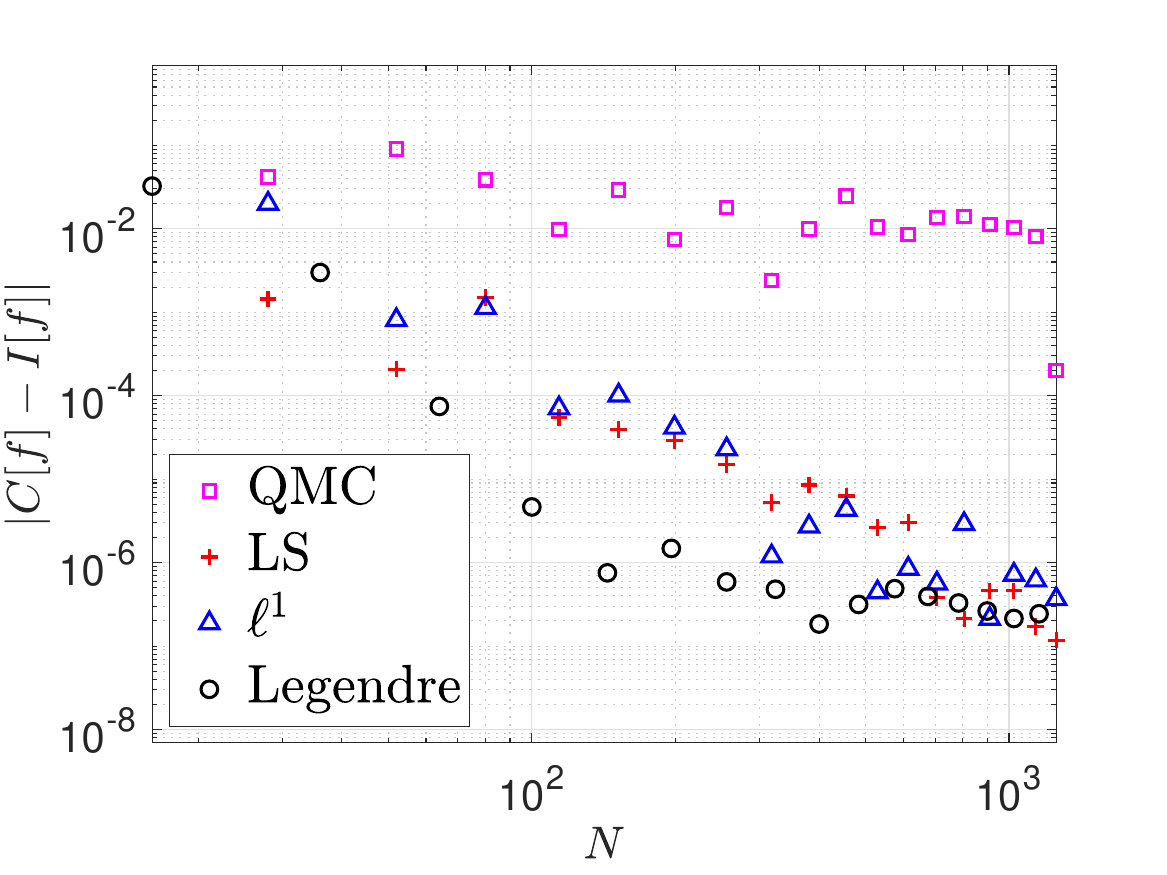} 
    \caption{$B_2$, Halton points}
    \label{fig:accuracy_test2_noisy_dim2_Halton}
  \end{subfigure}%
  \\
  \begin{subfigure}[b]{0.33\textwidth}
    \includegraphics[width=\textwidth]{%
      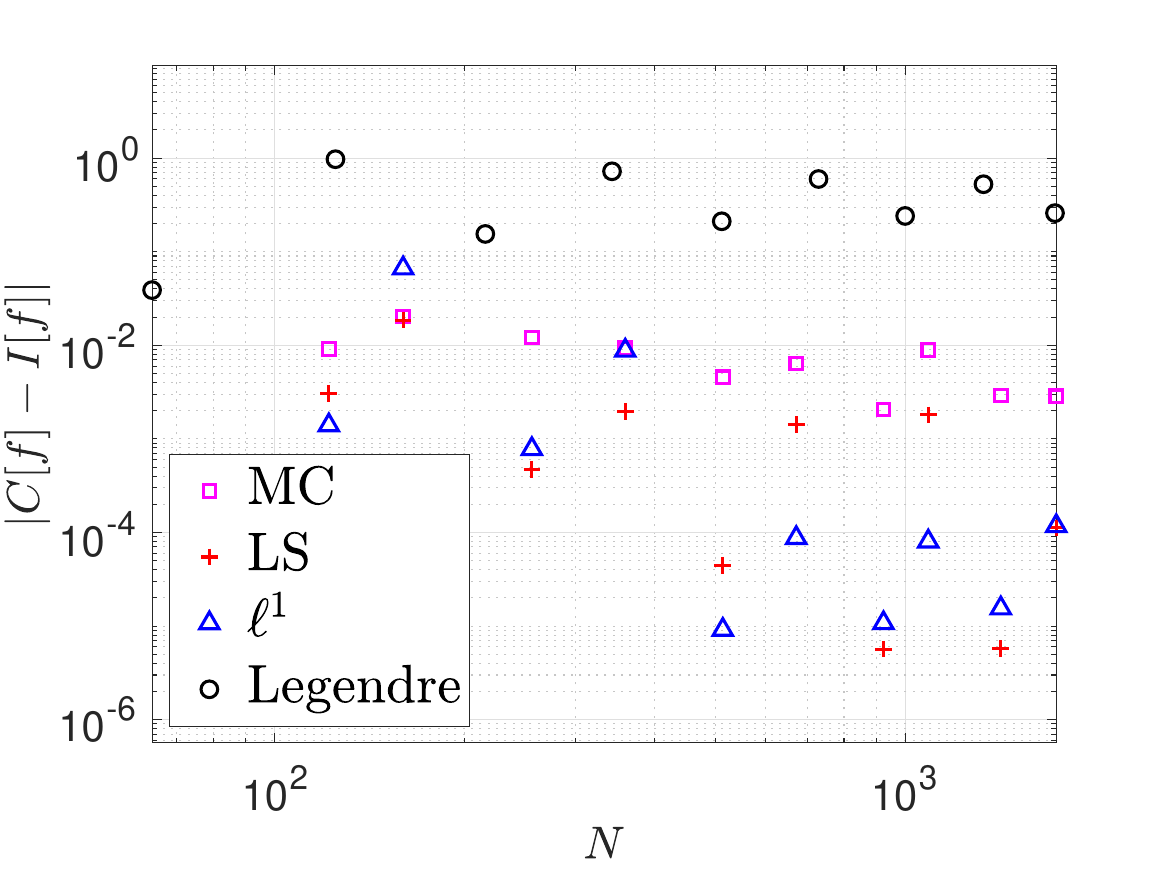} 
    \caption{$B_3$, equidistant points}
    \label{fig:accuracy_test2_noisy_dim3_equid}
  \end{subfigure}%
  ~
  \begin{subfigure}[b]{0.33\textwidth}
    \includegraphics[width=\textwidth]{%
      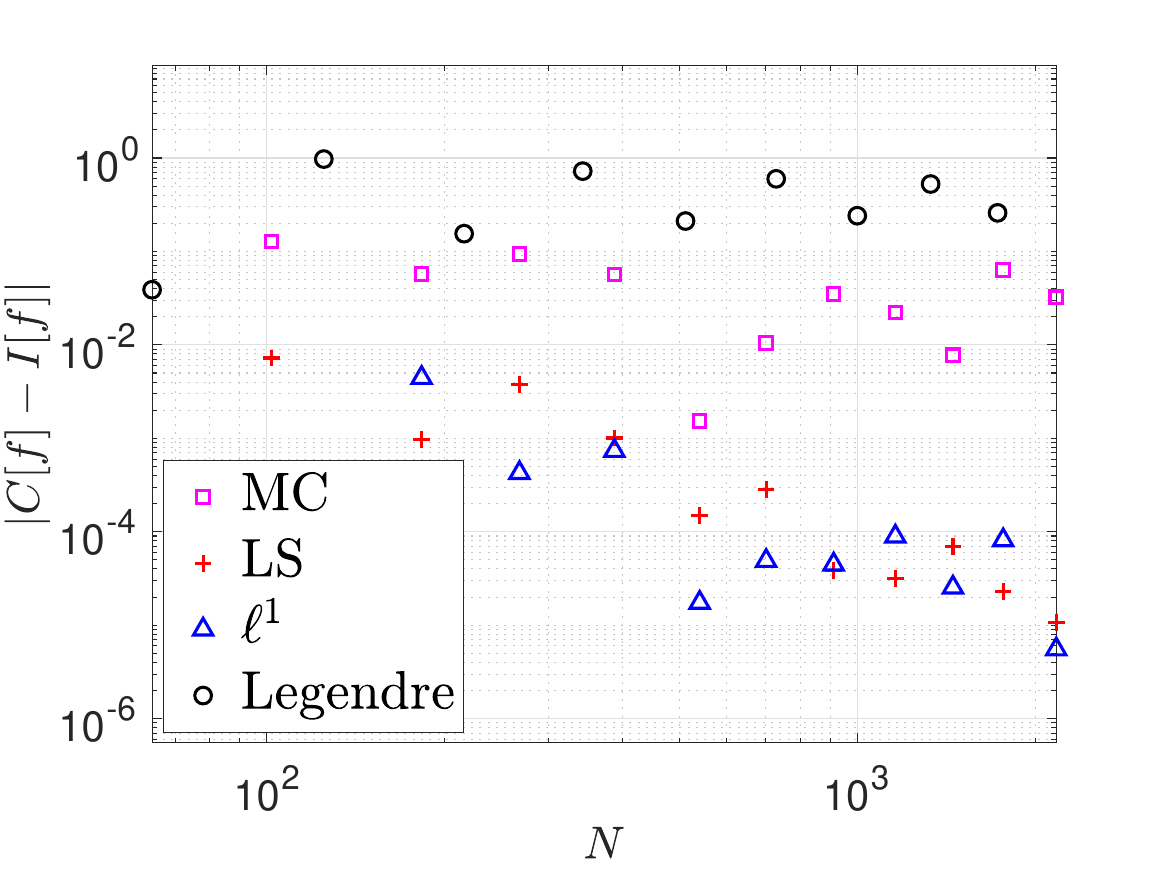} 
    \caption{$B_3$, random points}
    \label{fig:accuracy_test2_noisy_dim3_uniform}
  \end{subfigure}%
  ~
  \begin{subfigure}[b]{0.33\textwidth}
    \includegraphics[width=\textwidth]{%
      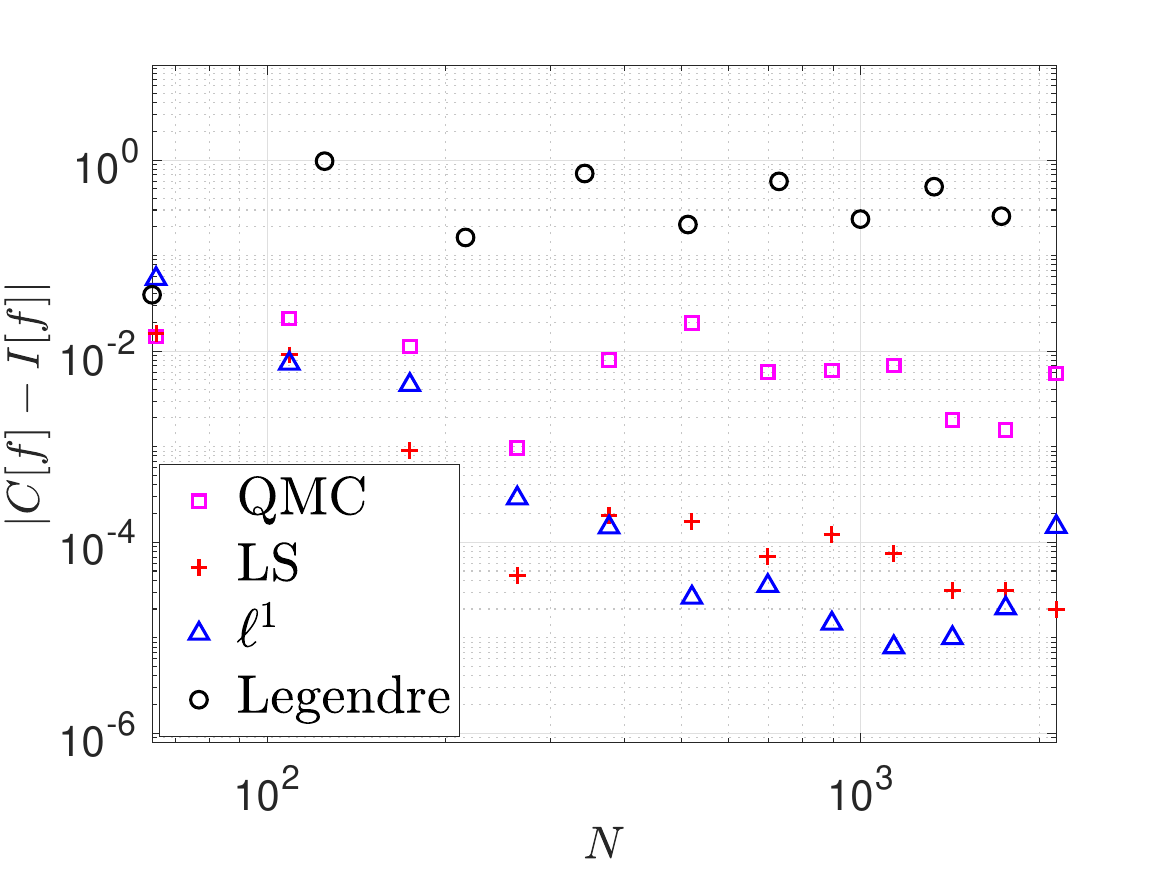} 
    \caption{$B_3$, Halton points}
    \label{fig:accuracy_test2_noisy_dim3_Halton}
  \end{subfigure}%
  \caption{Errors for $B_2$ and $B_3$ with $\omega$ and $f$ as in \eqref{eq:test2}.
  I.\,i.\,d.\ uniform noise ${Z_n \in \mathcal{U}( -10^{-6},10^{-6} )}$ was added.}
  \label{fig:accuracy_test2_noisy}
\end{figure}

This is also reflected in the corresponding numerical results reported in Figure \ref{fig:accuracy_test1_noisy} and Figure \ref{fig:accuracy_test2_noisy}. 
For each case, the experiments were repeated $50$ times and the reported errors are averaged. 
For each experiment, a new noise vector was drawn.
Note that all considered CFs behave fairly robust against the introduction of noise. 
This is in accordance with all considered CFs having nonnegative weights.

\subsection{Genz Test Functions} 

Below, some of Genz's test functions \cite{genz1984testing} (also see \cite{van2020adaptive}) are considered. 
Let $\Omega = [0,1]^q$ with $q=1,2$ and $\omega \equiv 1$. 
Then, Genz's tests functions are defined as follows: 
\begin{equation}\label{eq:Genz}
\begin{aligned}
	g_1(\boldsymbol{x}) 
		& = \cos\left( 2 \pi b_1 + \sum_{i=1}^q a_i x_i \right) \quad 
		&& \text{(oscillatory)}, \\
	g_2(\boldsymbol{x}) 
		& = \prod_{i=1}^q \left( a_i^{-2} + (x_i - b_i)^2 \right)^{-1} \quad 
		&& \text{(product peak)}, \\
	g_3(\boldsymbol{x}) 
		& = \left( 1 + \sum_{i=1}^q a_i x_i \right)^{-(q+1)} \quad 
		&& \text{(corner peak)}, \\
	g_4(\boldsymbol{x}) 
		& = \exp \left( - \sum_{i=1}^q a_i^2 ( x_i - b_i )^2 \right) \quad 
		&& \text{(Gaussian)}
\end{aligned}
\end{equation} 
These functions are crafted such that they have different difficult characteristics for numerical integration routines.
The vectors $\mathbf{a} = (a_1,\dots,a_q)^T$ and $\mathbf{b} = (b_1,\dots,b_q)^T$ respectively contain shape and translation parameters. 

\begin{figure}[tb]
  \centering 
  \captionsetup{justification=centering}
  \begin{subfigure}[b]{0.32\textwidth}
    \includegraphics[width=\textwidth]{%
      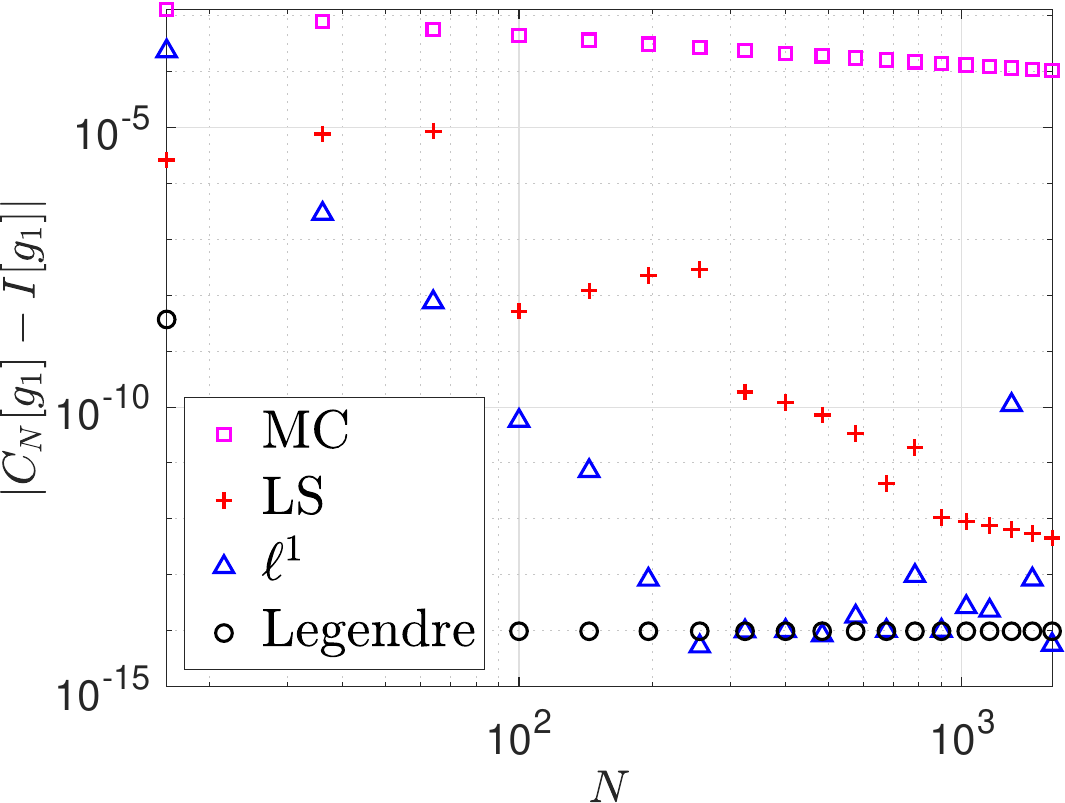} 
    \caption{$\Omega = [0,1]^2$, equidistant points}
    \label{fig:accuracy_Genz1_dim2_equid}
  \end{subfigure}%
  ~
  \begin{subfigure}[b]{0.32\textwidth}
    \includegraphics[width=\textwidth]{%
      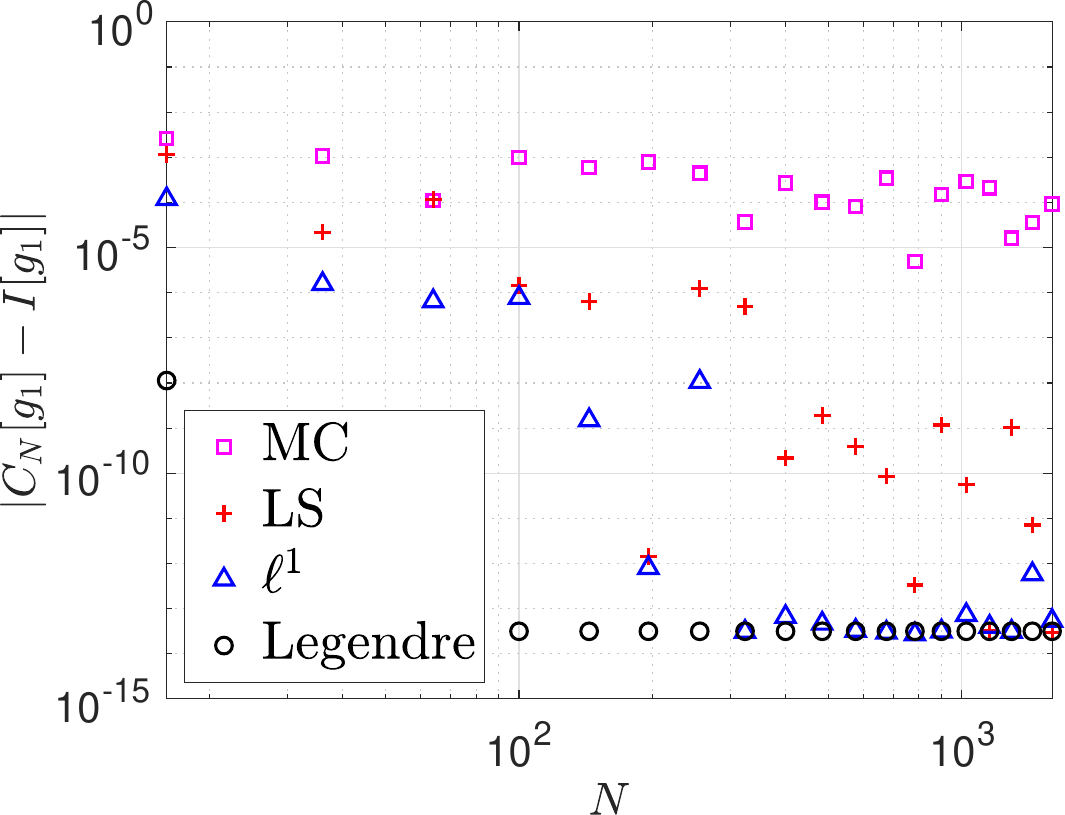} 
    \caption{$\Omega = [0,1]^2$, random points}
    \label{fig:accuracy_Genz1_dim2_uniform}
  \end{subfigure}%
  ~
  \begin{subfigure}[b]{0.32\textwidth}
    \includegraphics[width=\textwidth]{%
      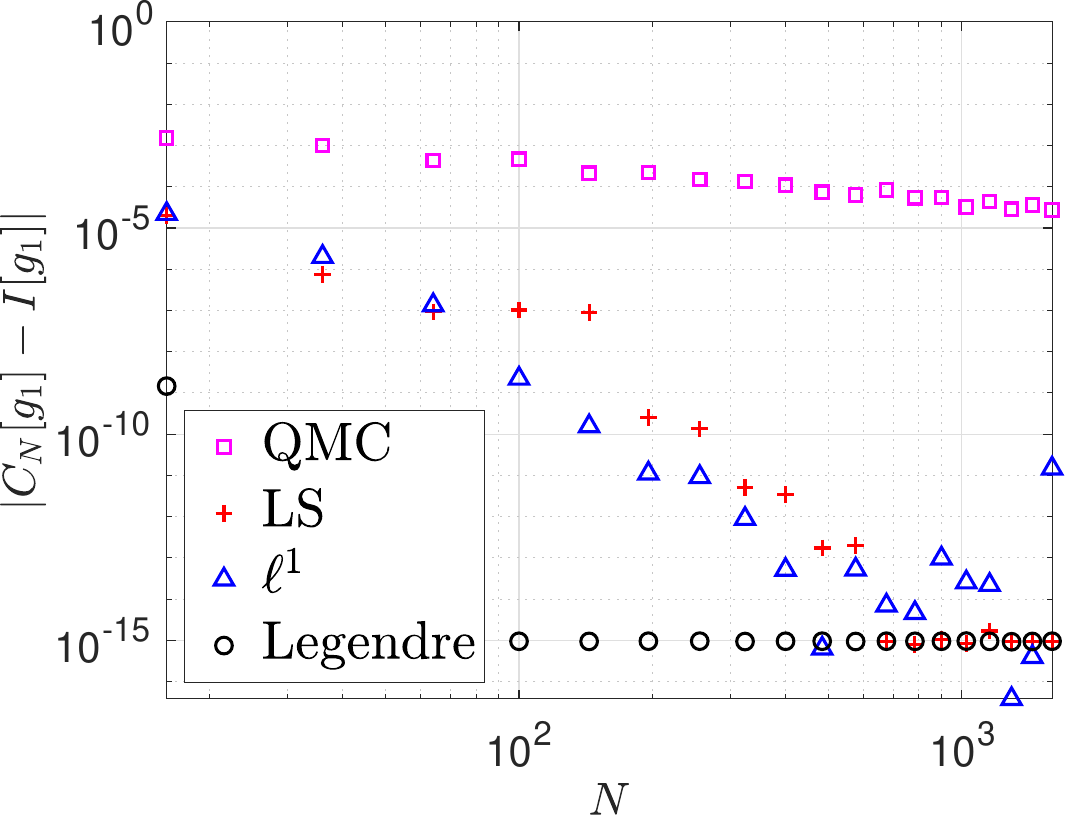} 
    \caption{$\Omega = [0,1]^2$, Halton points}
    \label{fig:accuracy_Genz1_dim2_Halton}
  \end{subfigure}%
  \\
  \begin{subfigure}[b]{0.32\textwidth}
    \includegraphics[width=\textwidth]{%
      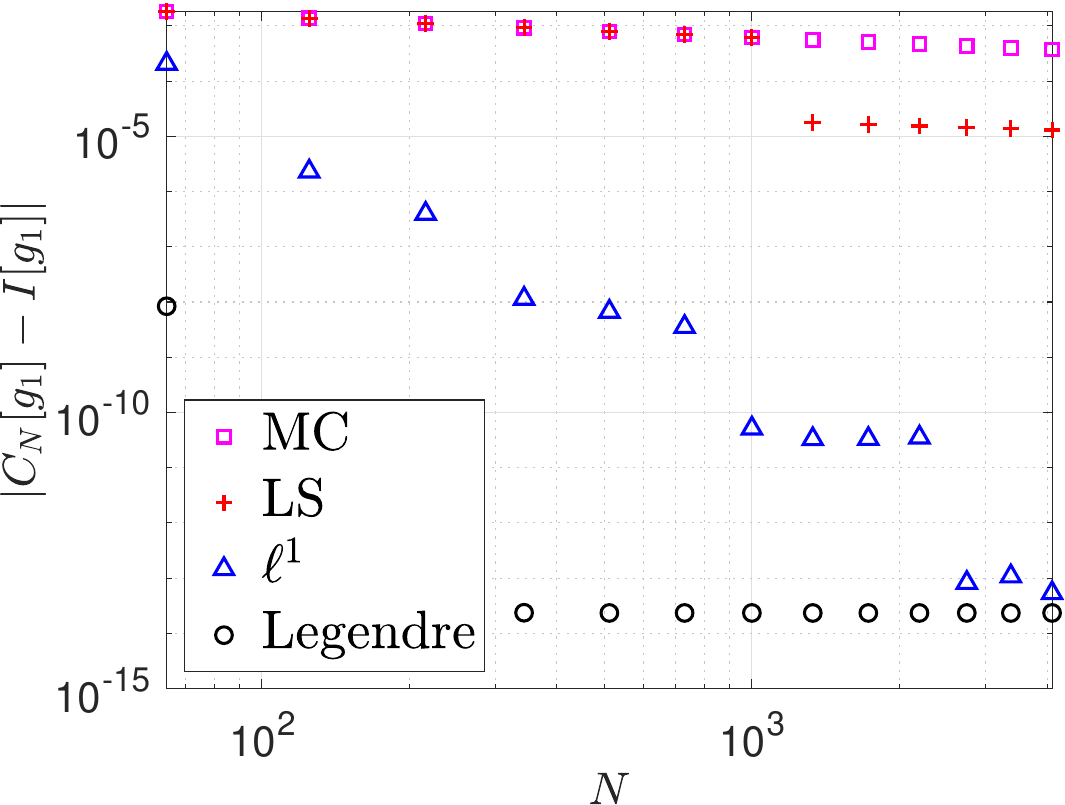} 
    \caption{$\Omega = [0,1]^3$, equidistant points}
    \label{fig:accuracy_Genz1_dim3_equid}
  \end{subfigure}%
  ~
  \begin{subfigure}[b]{0.32\textwidth}
    \includegraphics[width=\textwidth]{%
      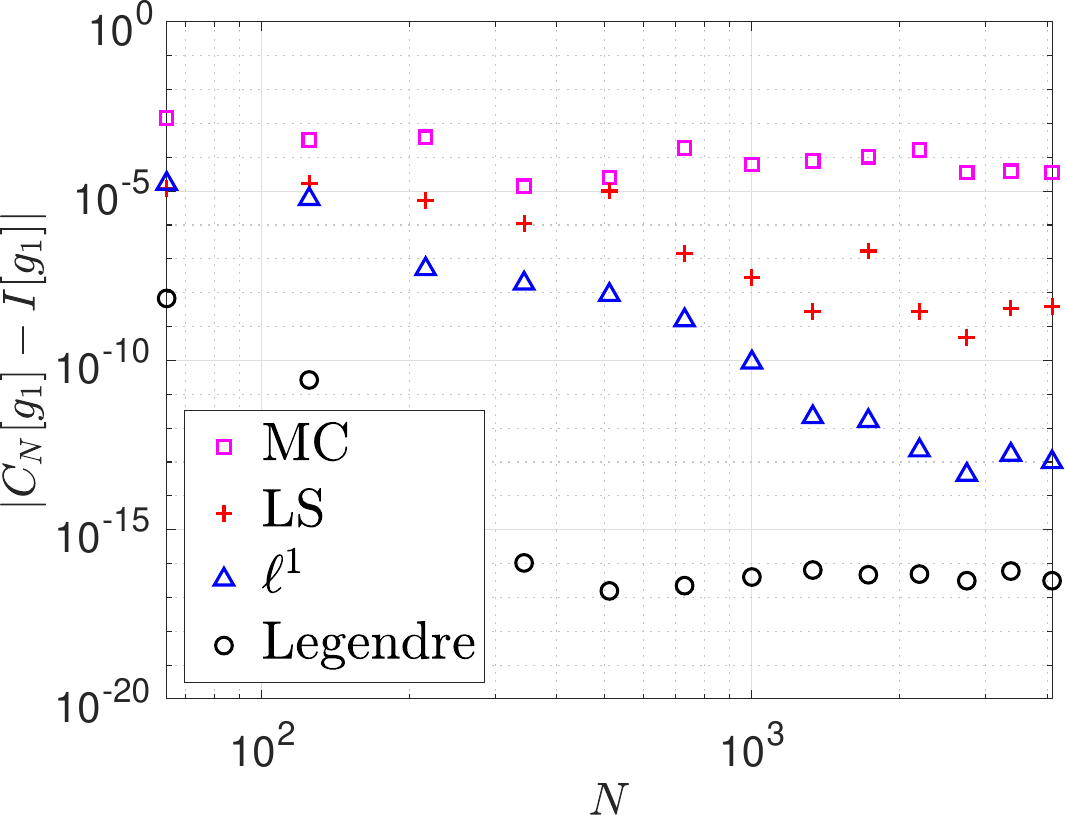} 
    \caption{$\Omega = [0,1]^3$, random points}
    \label{fig:accuracy_Genz1_dim3_uniform}
  \end{subfigure}%
  ~
  \begin{subfigure}[b]{0.32\textwidth}
    \includegraphics[width=\textwidth]{%
      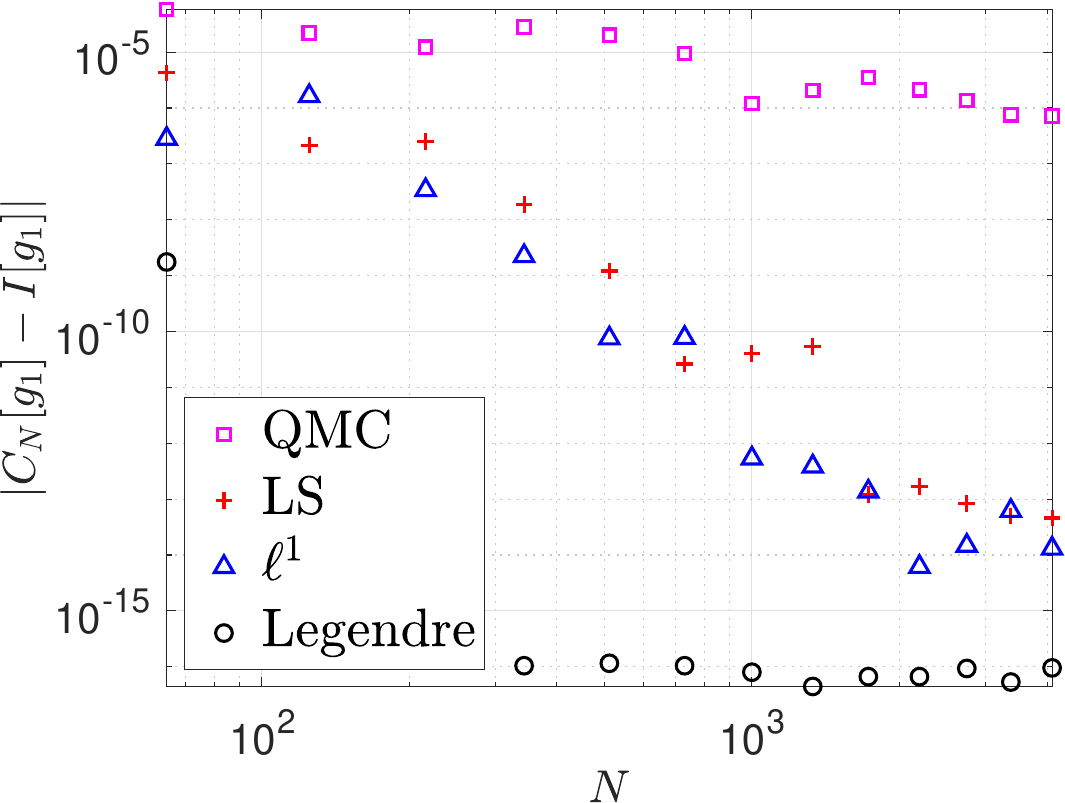} 
    \caption{$\Omega = [0,1]^3$, Halton points}
    \label{fig:accuracy_Genz1_dim3_Halton}
  \end{subfigure}%
  \caption{Errors for the first Genz function $g_1$ (oscillatory) as in \eqref{eq:Genz}.}
  \label{fig:Genz1}
\end{figure}

\begin{figure}[!htb]
  \centering 
  \captionsetup{justification=centering}
  \begin{subfigure}[b]{0.32\textwidth}
    \includegraphics[width=\textwidth]{%
      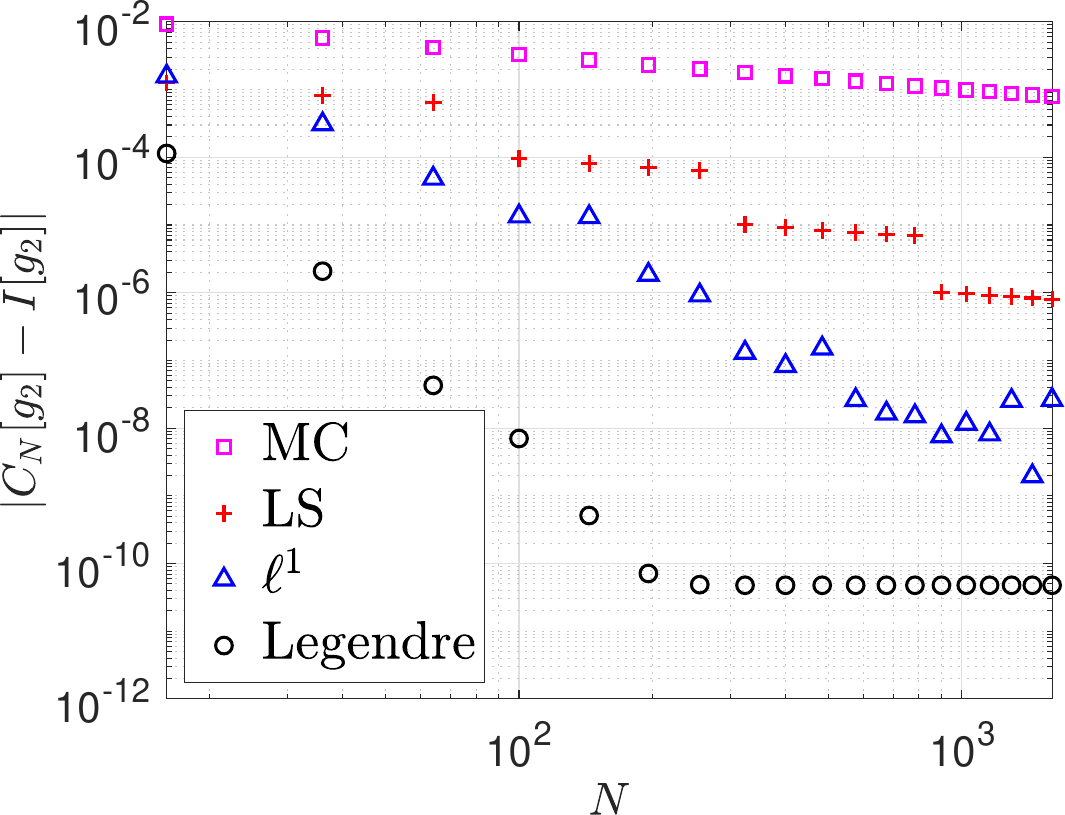} 
    \caption{$\Omega = [0,1]^2$, equidistant points}
    \label{fig:accuracy_Genz2_dim2_equid}
  \end{subfigure}%
  ~
  \begin{subfigure}[b]{0.32\textwidth}
    \includegraphics[width=\textwidth]{%
      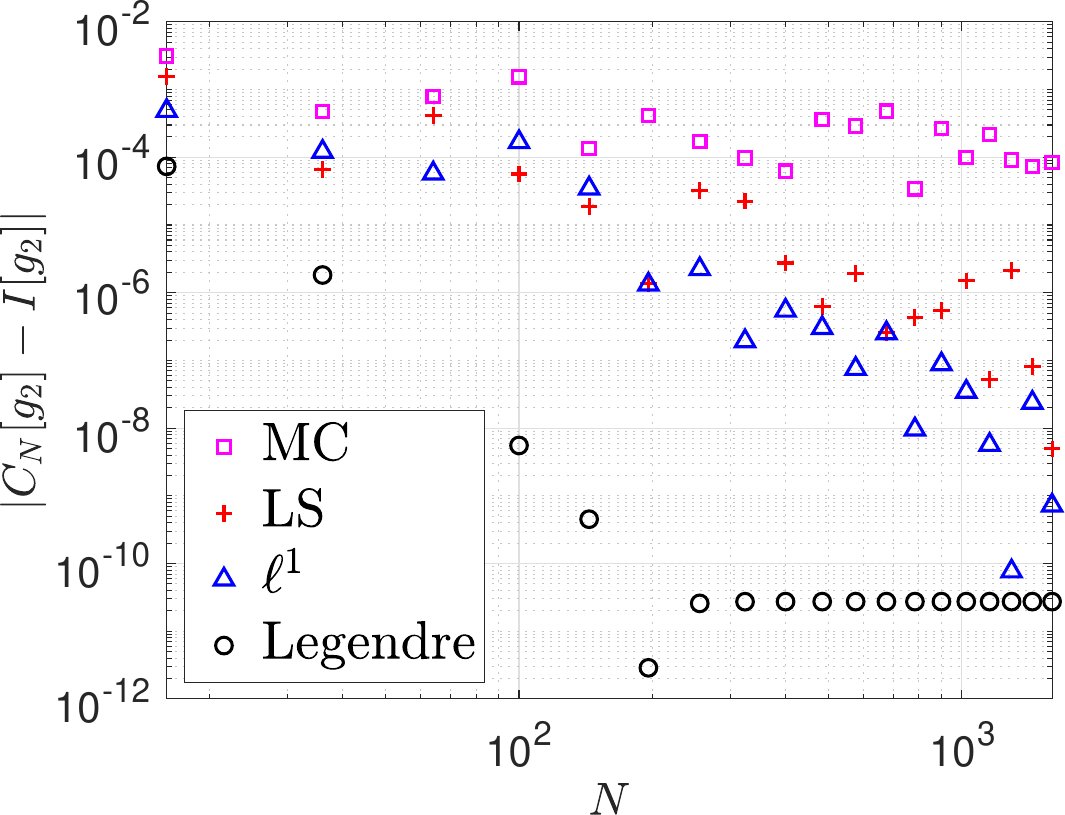} 
    \caption{$\Omega = [0,1]^2$, random points}
    \label{fig:accuracy_Genz2_dim2_uniform}
  \end{subfigure}%
  ~
  \begin{subfigure}[b]{0.32\textwidth}
    \includegraphics[width=\textwidth]{%
      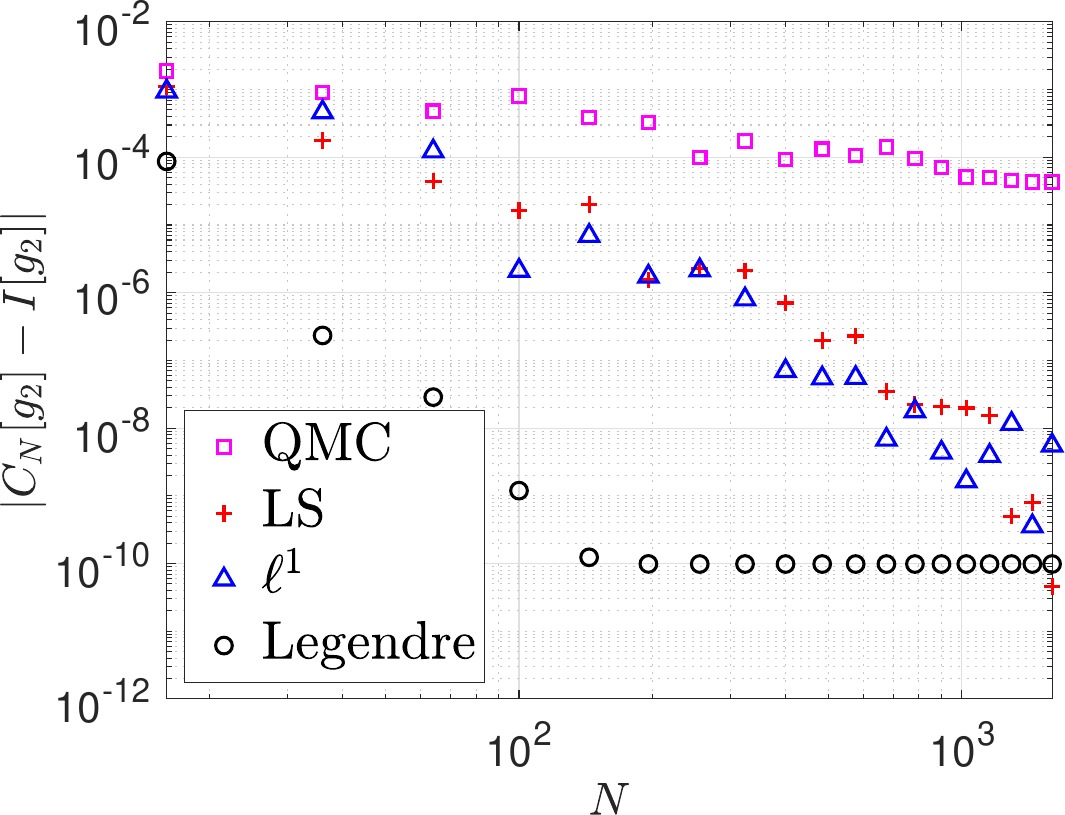} 
    \caption{$\Omega = [0,1]^2$, Halton points}
    \label{fig:accuracy_Genz2_dim2_Halton}
  \end{subfigure}%
  \\
  \begin{subfigure}[b]{0.32\textwidth}
    \includegraphics[width=\textwidth]{%
      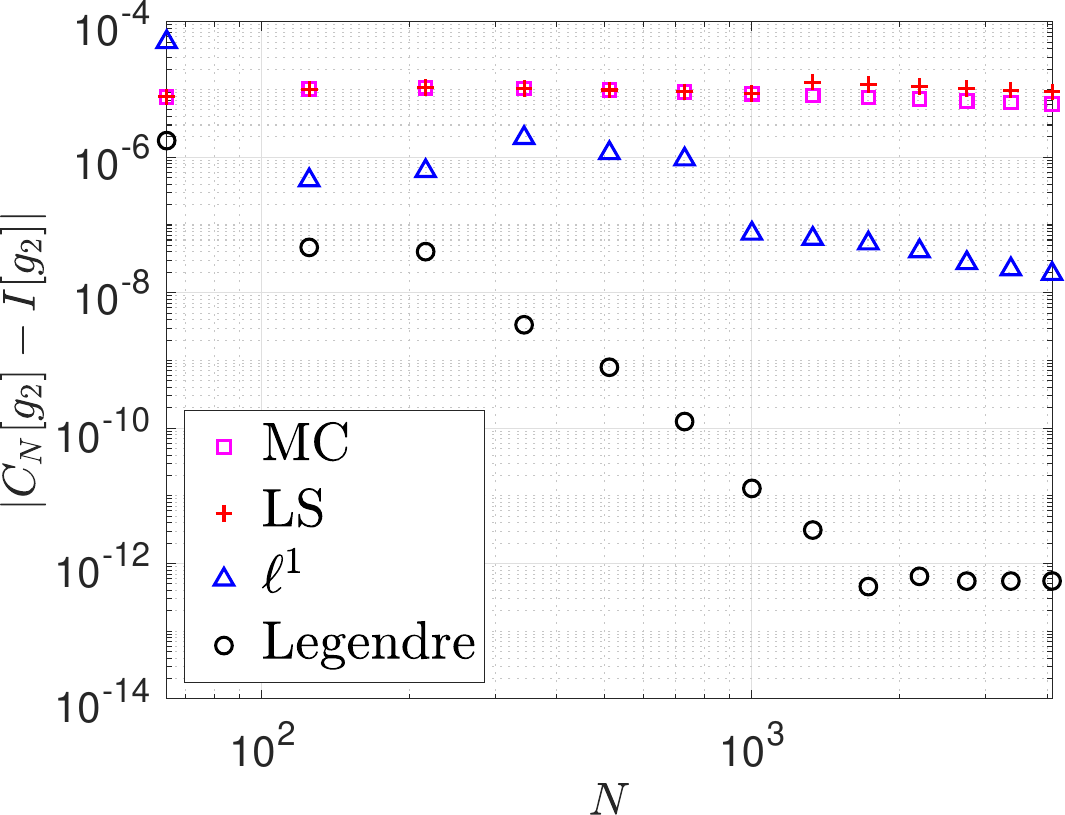} 
    \caption{$\Omega = [0,1]^3$, equidistant points}
    \label{fig:accuracy_Genz2_dim3_equid}
  \end{subfigure}%
  ~
  \begin{subfigure}[b]{0.32\textwidth}
    \includegraphics[width=\textwidth]{%
      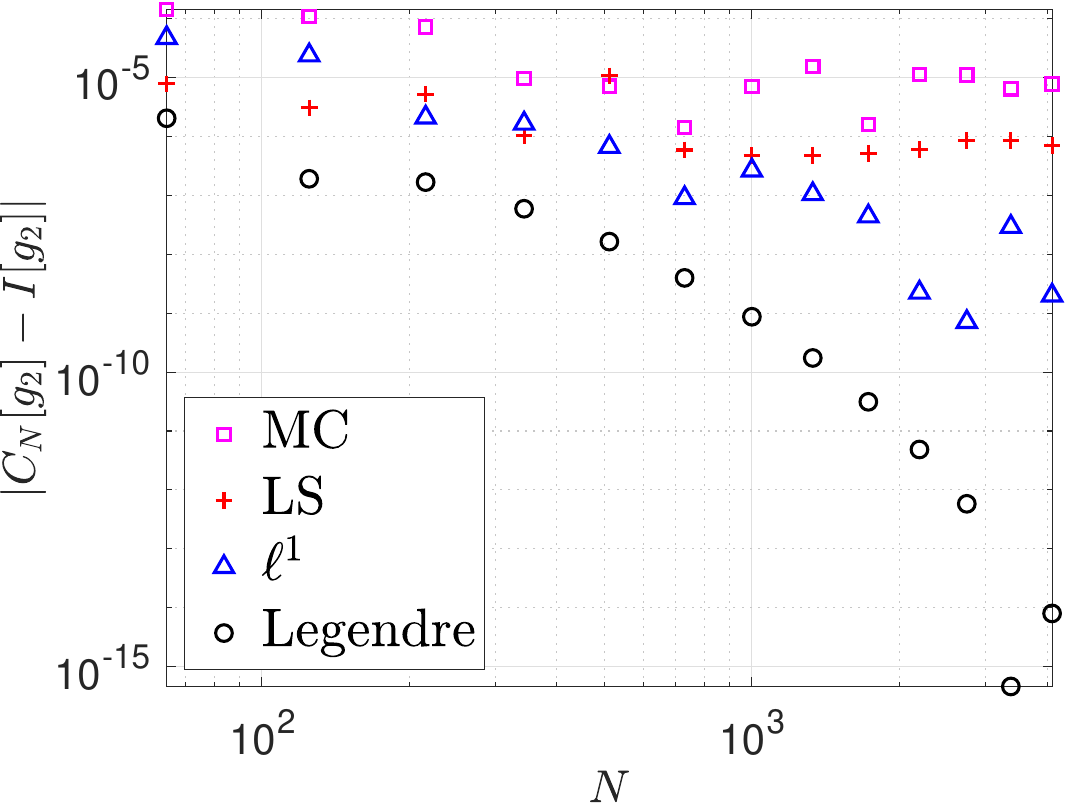} 
    \caption{$\Omega = [0,1]^3$, random points}
    \label{fig:accuracy_Genz2_dim3_uniform}
  \end{subfigure}%
  ~
  \begin{subfigure}[b]{0.32\textwidth}
    \includegraphics[width=\textwidth]{%
      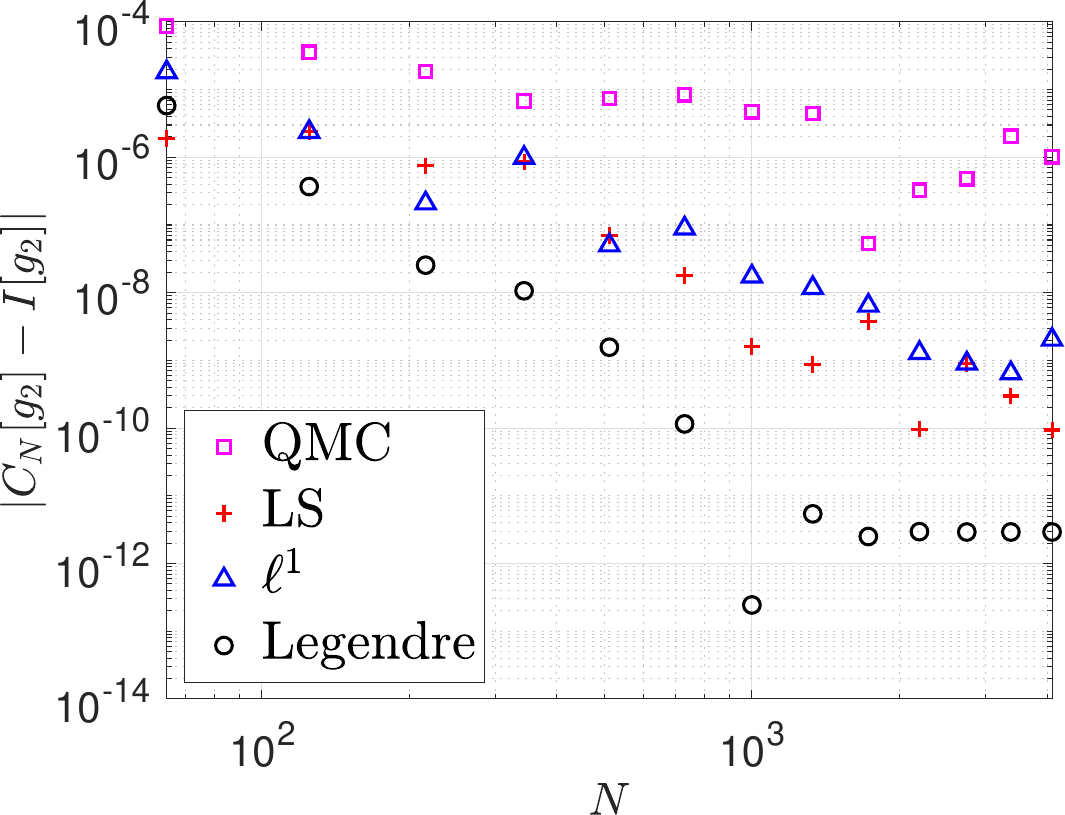} 
    \caption{$\Omega = [0,1]^3$, Halton points}
    \label{fig:accuracy_Genz2_dim3_Halton}
  \end{subfigure}%
  \caption{Errors for the second Genz function $g_2$ (product peak) as in \eqref{eq:Genz}.}
  \label{fig:Genz2}
\end{figure}

\begin{figure}[tb]
  \centering 
  \captionsetup{justification=centering}
  \begin{subfigure}[b]{0.32\textwidth}
    \includegraphics[width=\textwidth]{%
      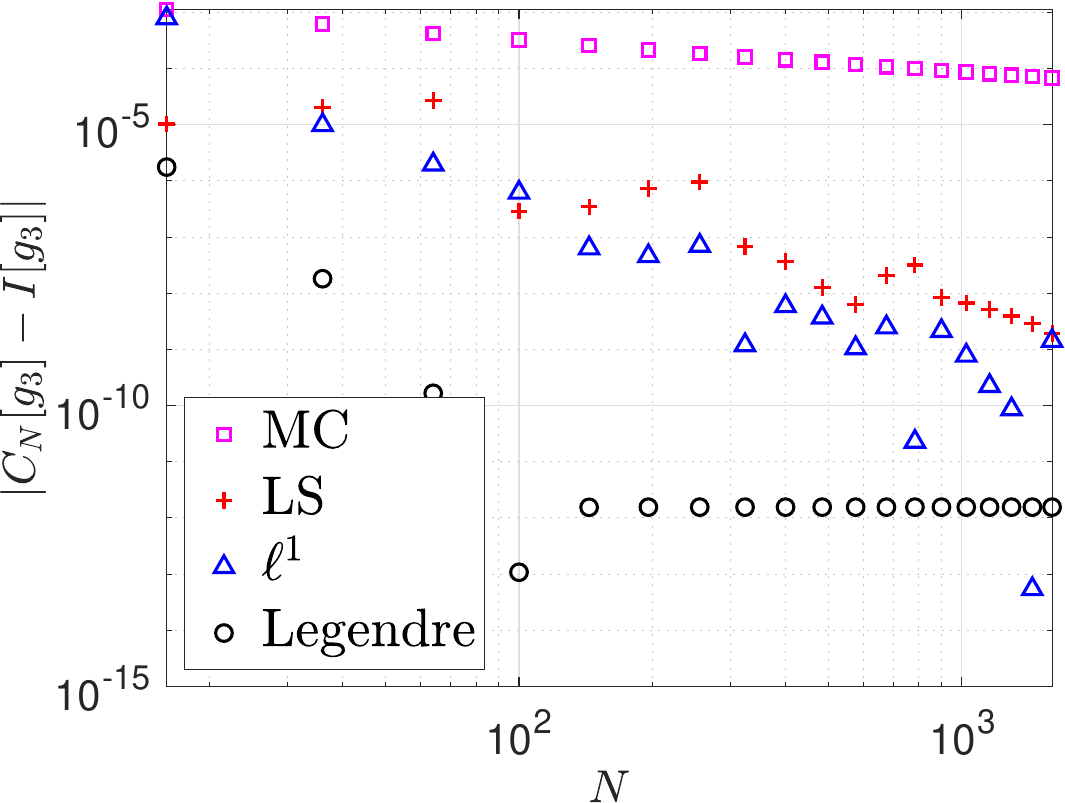} 
    \caption{$\Omega = [0,1]^2$, equidistant points}
    \label{fig:accuracy_Genz3_dim2_equid}
  \end{subfigure}%
  ~
  \begin{subfigure}[b]{0.32\textwidth}
    \includegraphics[width=\textwidth]{%
      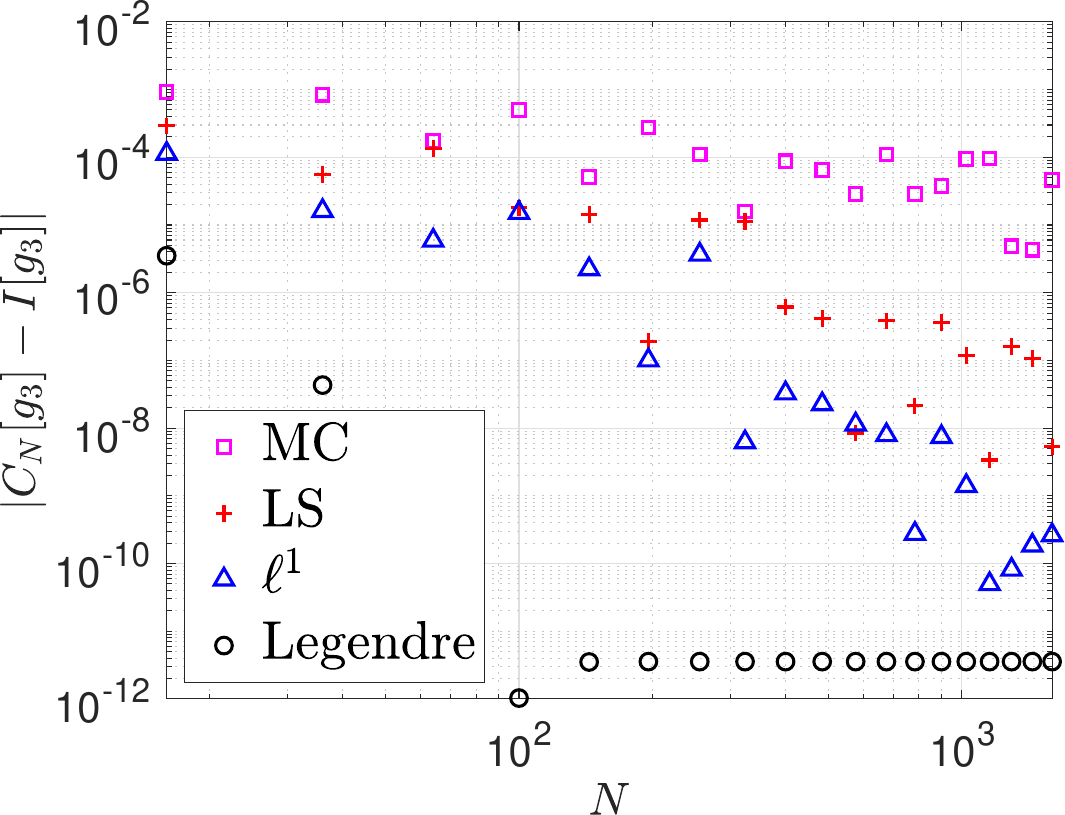} 
    \caption{$\Omega = [0,1]^2$, random points}
    \label{fig:accuracy_Genz3_dim2_uniform}
  \end{subfigure}%
  ~
  \begin{subfigure}[b]{0.32\textwidth}
    \includegraphics[width=\textwidth]{%
      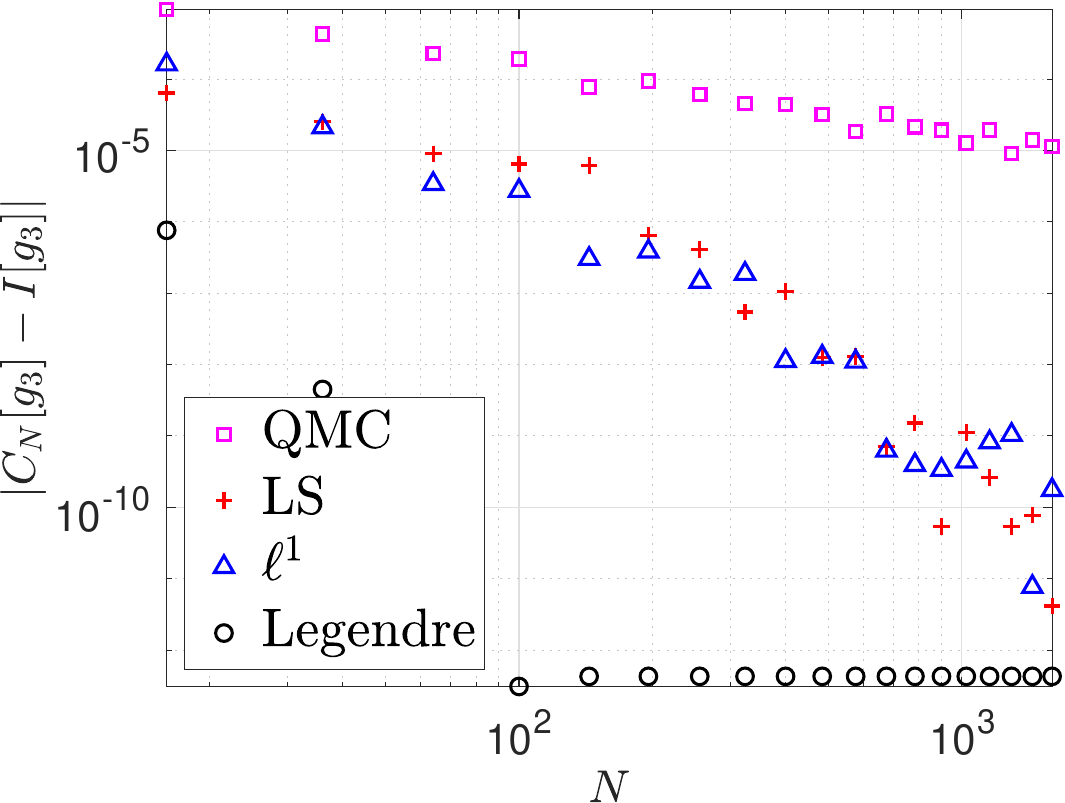} 
    \caption{$\Omega = [0,1]^2$, Halton points}
    \label{fig:accuracy_Genz3_dim2_Halton}
  \end{subfigure}%
  \\
  \begin{subfigure}[b]{0.32\textwidth}
    \includegraphics[width=\textwidth]{%
      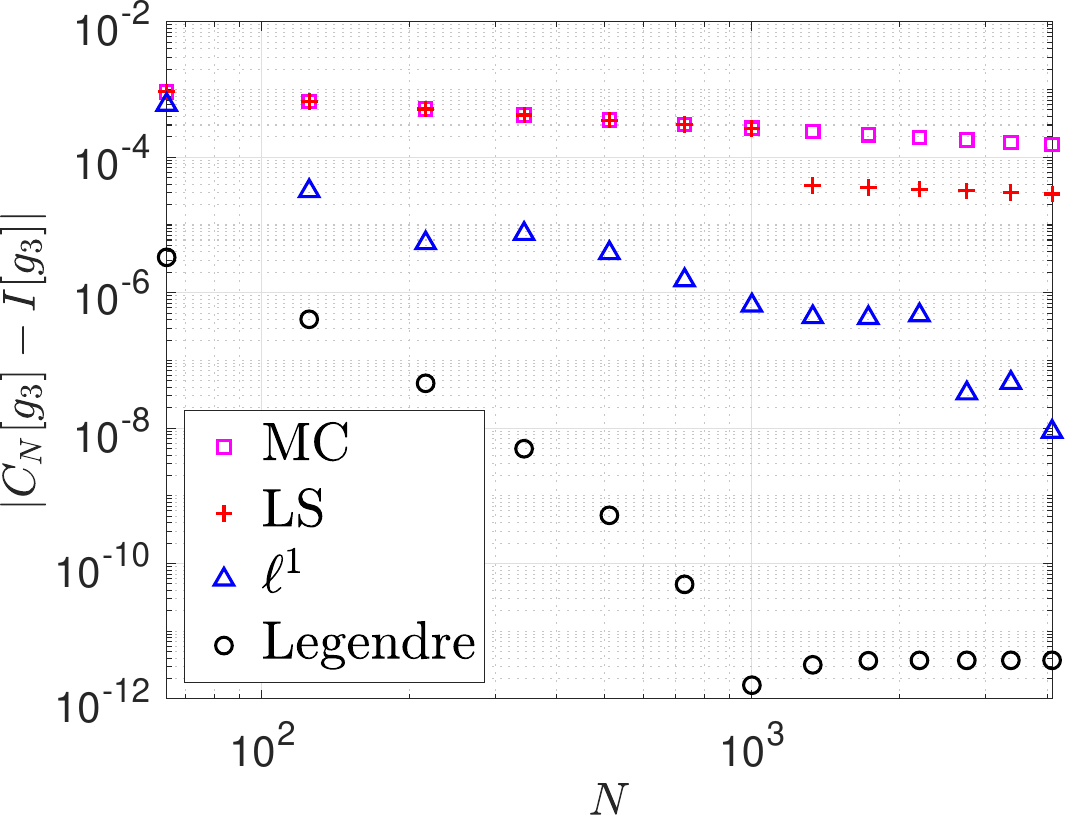} 
    \caption{$\Omega = [0,1]^3$, equidistant points}
    \label{fig:accuracy_Genz3_dim3_equid}
  \end{subfigure}%
  ~
  \begin{subfigure}[b]{0.32\textwidth}
    \includegraphics[width=\textwidth]{%
      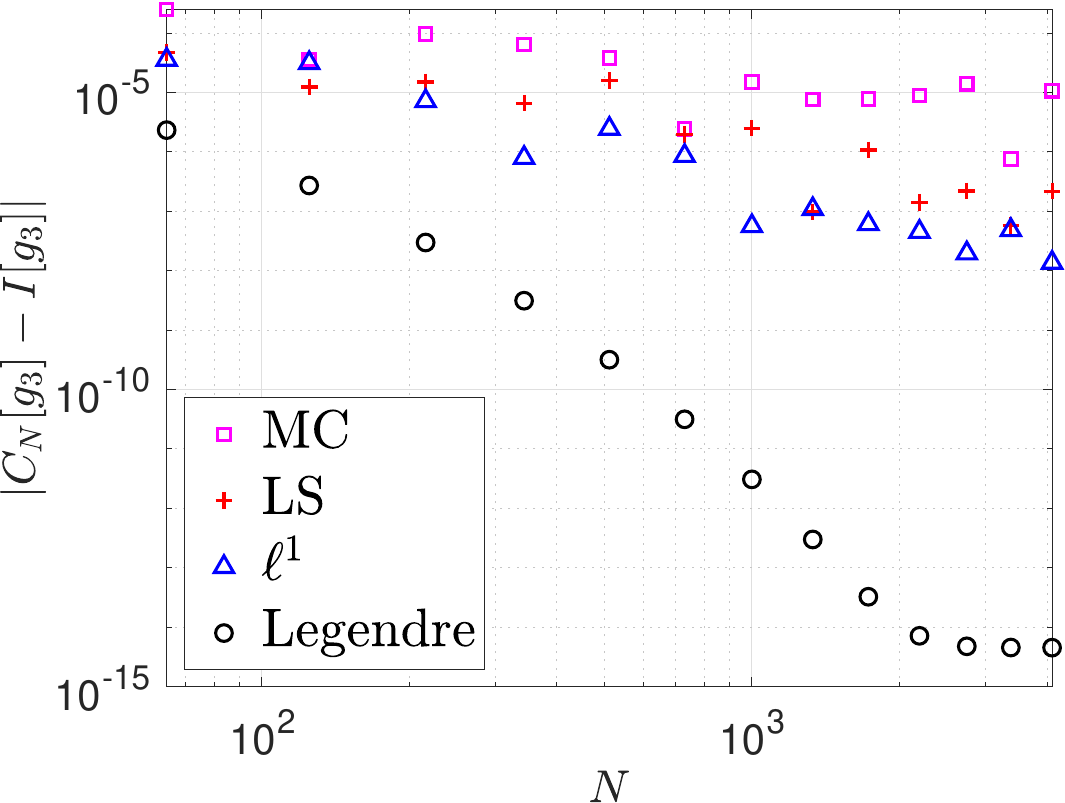} 
    \caption{$\Omega = [0,1]^3$, random points}
    \label{fig:accuracy_Genz3_dim3_uniform}
  \end{subfigure}%
  ~
  \begin{subfigure}[b]{0.32\textwidth}
    \includegraphics[width=\textwidth]{%
      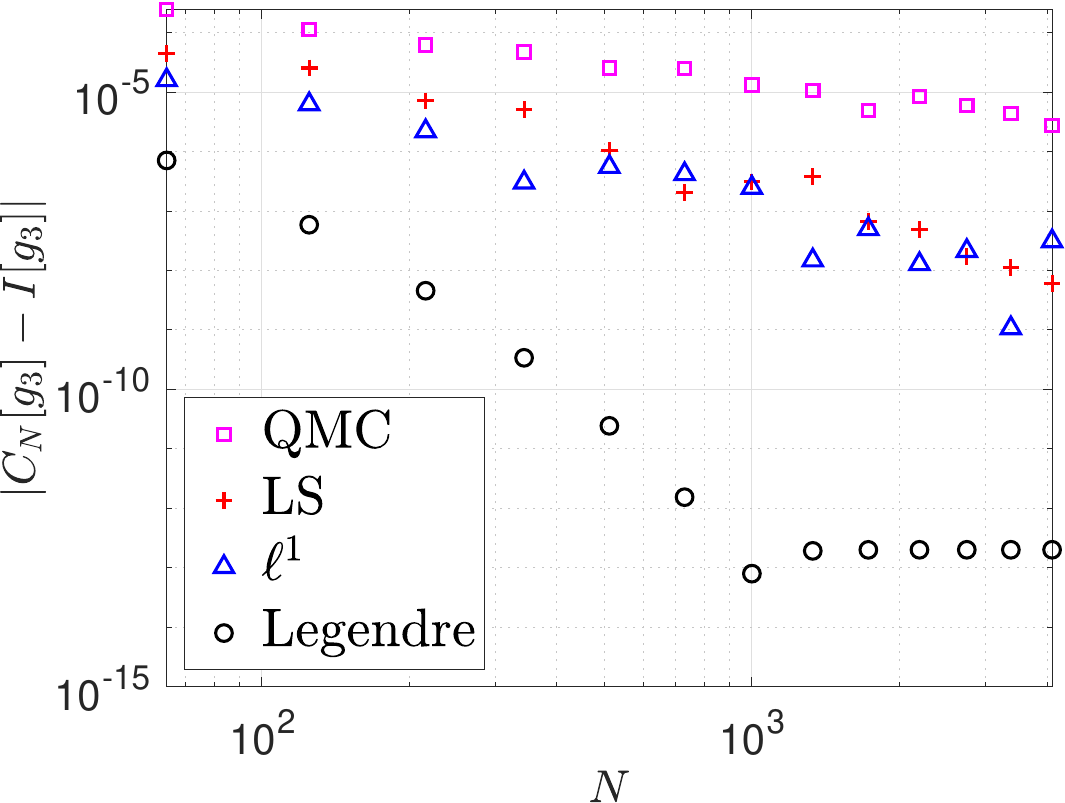} 
    \caption{$\Omega = [0,1]^3$, Halton points}
    \label{fig:accuracy_Genz3_dim3_Halton}
  \end{subfigure}%
  \caption{Errors for the third Genz function $g_3$ (corner peak) as in \eqref{eq:Genz}.}
  \label{fig:Genz3}
\end{figure}

\begin{figure}[tb]
  \centering 
  \captionsetup{justification=centering}
  \begin{subfigure}[b]{0.32\textwidth}
    \includegraphics[width=\textwidth]{%
      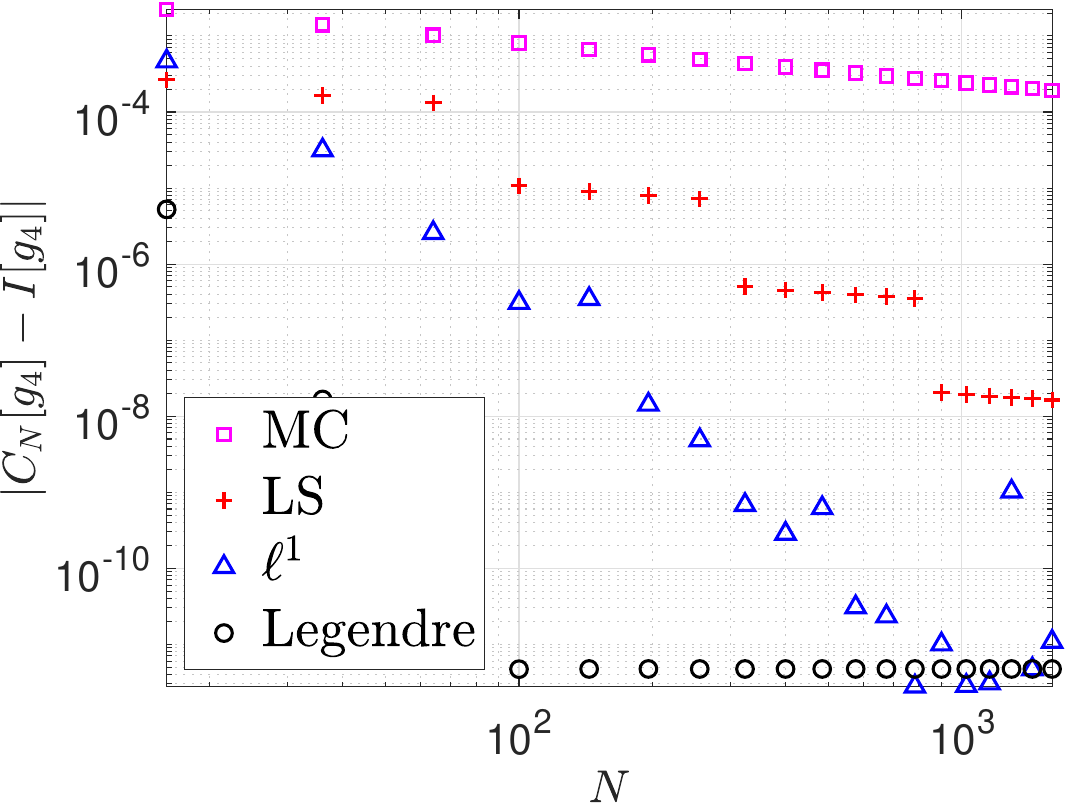} 
    \caption{$\Omega = [0,1]^2$, equidistant points}
    \label{fig:accuracy_Genz4_dim2_equid}
  \end{subfigure}%
  ~
  \begin{subfigure}[b]{0.32\textwidth}
    \includegraphics[width=\textwidth]{%
      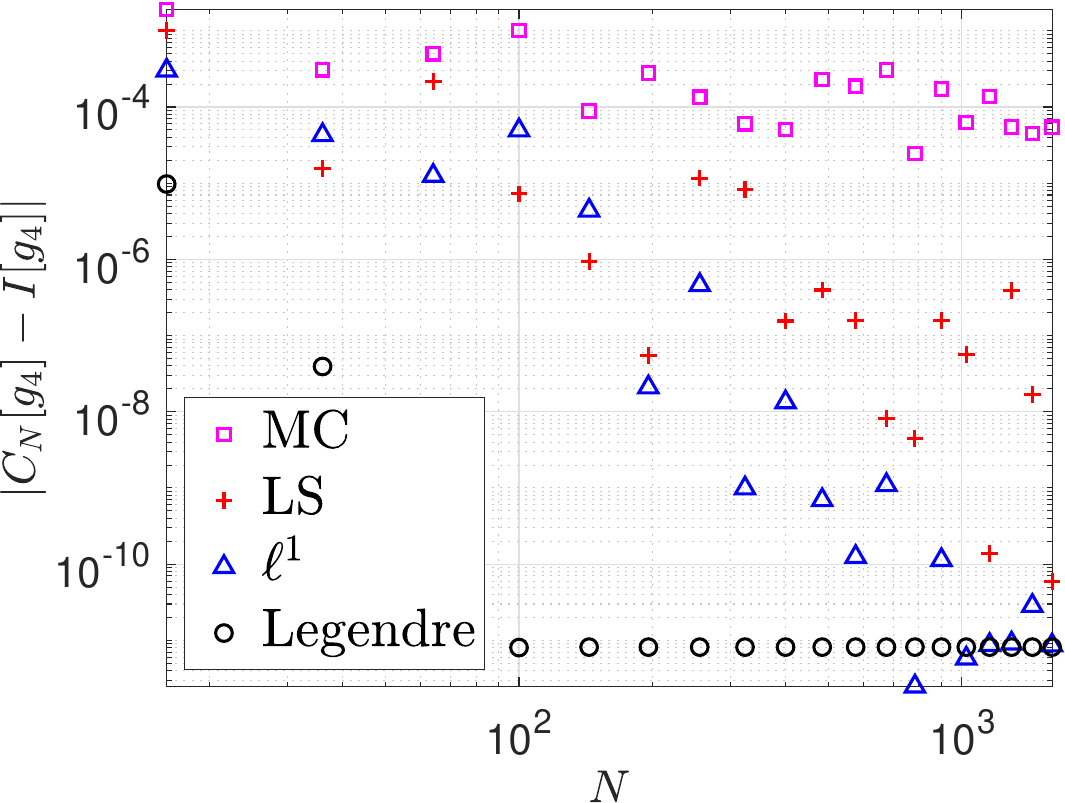} 
    \caption{$\Omega = [0,1]^2$, random points}
    \label{fig:accuracy_Genz4_dim2_uniform}
  \end{subfigure}%
  ~
  \begin{subfigure}[b]{0.32\textwidth}
    \includegraphics[width=\textwidth]{%
      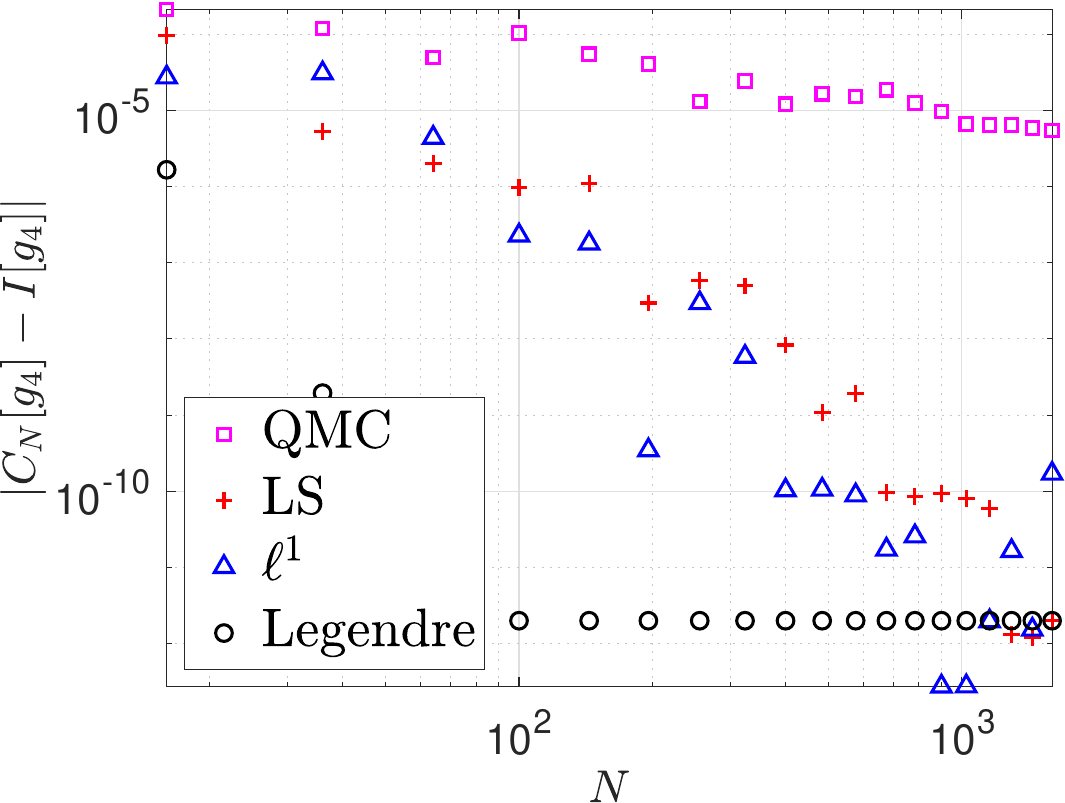} 
    \caption{$\Omega = [0,1]^2$, Halton points}
    \label{fig:accuracy_Genz4_dim2_Halton}
  \end{subfigure}%
  \\
  \begin{subfigure}[b]{0.32\textwidth}
    \includegraphics[width=\textwidth]{%
      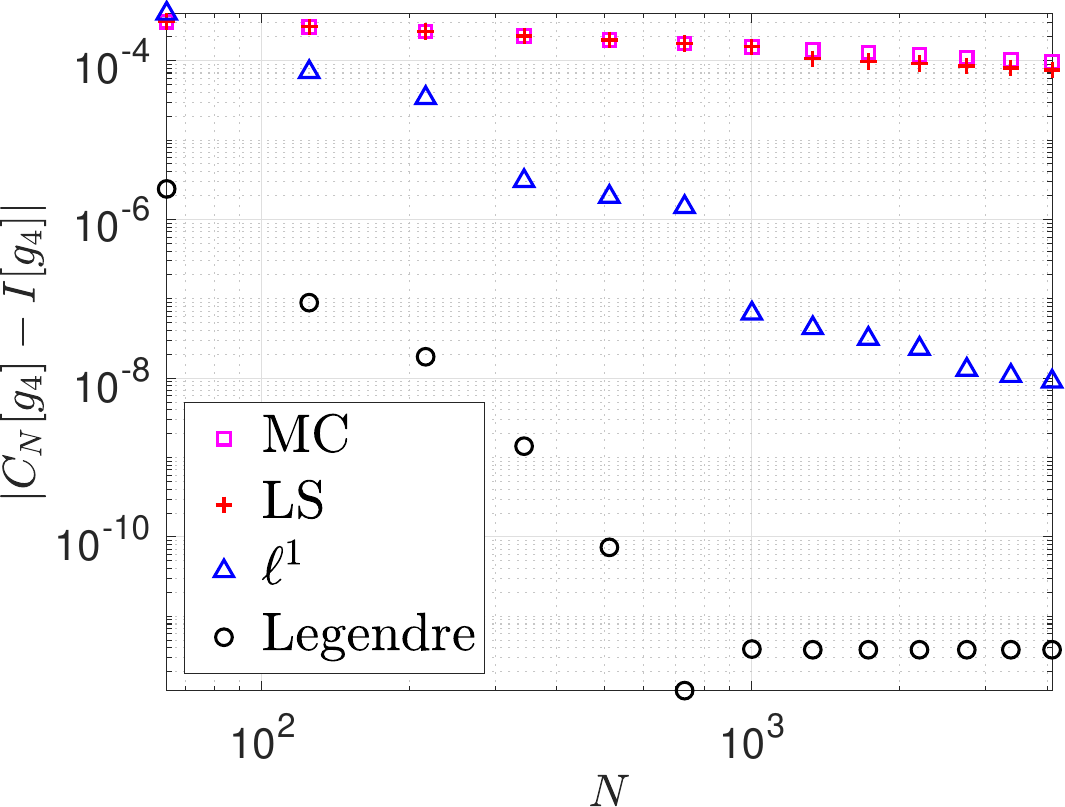} 
    \caption{$\Omega = [0,1]^3$, equidistant points}
    \label{fig:accuracy_Genz4_dim3_equid}
  \end{subfigure}%
  ~
  \begin{subfigure}[b]{0.32\textwidth}
    \includegraphics[width=\textwidth]{%
      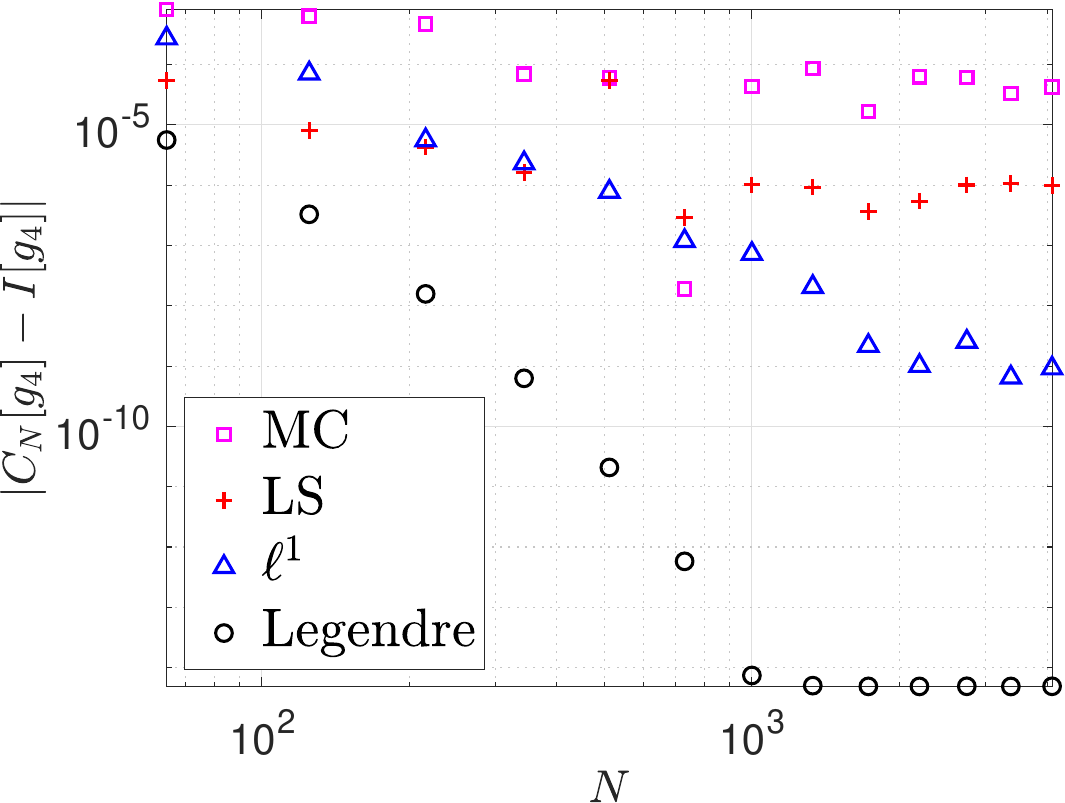} 
    \caption{$\Omega = [0,1]^3$, random points}
    \label{fig:accuracy_Genz4_dim3_uniform}
  \end{subfigure}%
  ~
  \begin{subfigure}[b]{0.32\textwidth}
    \includegraphics[width=\textwidth]{%
      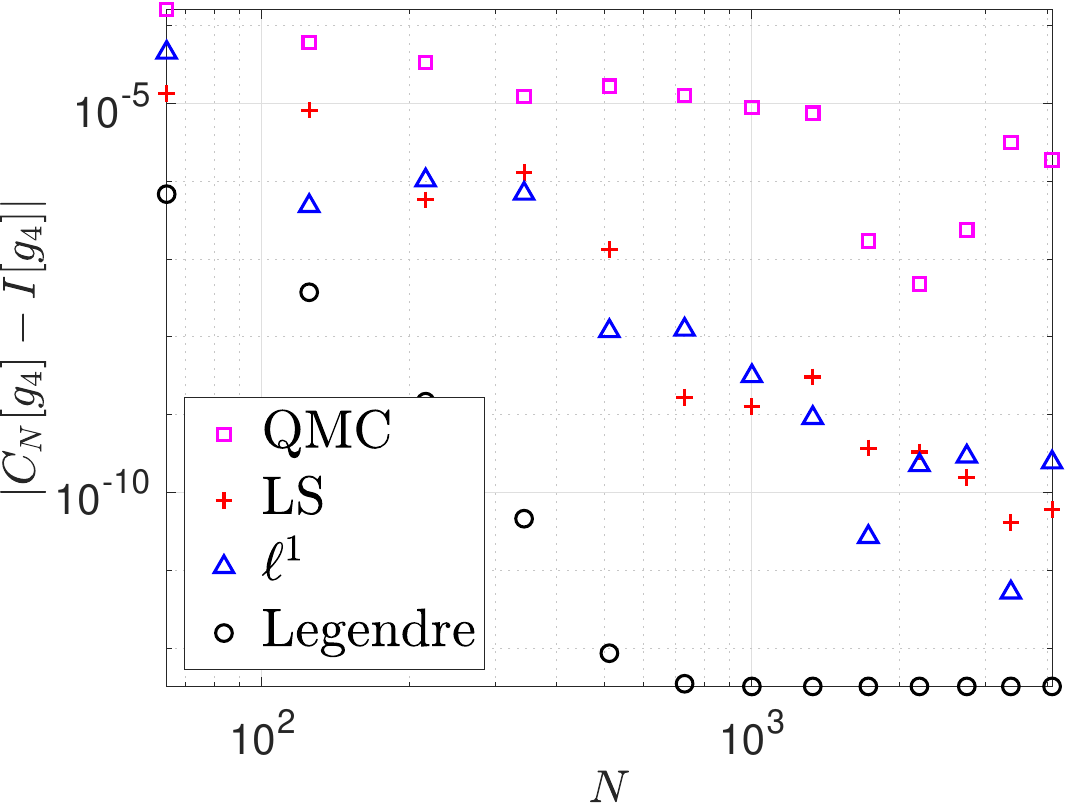} 
    \caption{$\Omega = [0,1]^3$, Halton points}
    \label{fig:accuracy_Genz4_dim3_Halton}
  \end{subfigure}%
  \caption{Errors for the fourth Genz function $g_4$ (Gaussian) as in \eqref{eq:Genz}.}
  \label{fig:Genz4}
\end{figure}

Here, these are chosen randomly.
For each case, the experiment was repeated $50$ times. 
At the same time, for each experiment, the vectors $\mathbf{a}$ and $\mathbf{b}$ were drawn randomly from $[0,1]^q$ and $\mathbf{a}$ was subsequently scaled such that $\|\mathbf{a}\| = 5/2$.

Figures \ref{fig:Genz1}, \ref{fig:Genz2}, \ref{fig:Genz3} and \ref{fig:Genz4}, report the averaged errors. 
It can be observed that in almost all cases the $\ell^1$- and LS-CFs yield more accurate results than the (quasi-)MC method applied to the same set of data points. 
Again, the product Legendre rule is reported only to provide a reference. 
It is not applied to the same set of data points.

\subsection{A Nonstandard Domain} 

We complete our numerical investigation by considering a nonstandard domain $\Omega = B_2 \cup [1,2]^2$. 
That is, $\Omega$ consists of a unit circle and a translated unit cube. 
See Figure \ref{fig:nonstandard_domian} for an illustration. 
Note that the volume of this domain is given by $|\Omega| = \pi + 1$. 

\begin{figure}[tb]
\centering
	\begin{subfigure}[b]{0.4\textwidth}
		\begin{center}
  		\begin{tikzpicture}[domain = -2.5:4.5, scale=.65, line width=0.6pt]
			\draw[->,>=stealth] (-2.5,0) -- (4.5,0) node[below] {$x$};
    			\draw[->,>=stealth] (0,-2.5) -- (0,4.5) node[right] {$y$};
			\draw (-2,0.1) -- (-2,-0.1) node [below] {-1};
			\draw (2,0.1) -- (2,-0.1) node [below] {1};
			\draw (4,0.1) -- (4,-0.1) node [below] {2};
			\draw (-0.1,-2) -- (0.1,-2) node [left] {-1 \ };
			\draw (-0.1,2) -- (0.1,2) node [left] {1 \ };
			\draw (-0.1,4) -- (0.1,4) node [left] {2 \ };

			\draw[blue]  (0,0) circle [radius = 2]; 
			\draw[blue] (2,2) rectangle (4,4);
			\draw (0,2) node (A) {}; 
			\draw (2,2.5) node (B) {}; 
			\draw[blue] (1,3) node (C) {\Large $\Omega$};
			\draw (A) to[out=60, in=180] (C);
			\draw (B) to[out=180, in=0] (C);

		\end{tikzpicture}
		\end{center}
		\caption{Domain}
		\label{fig:nonstandard_domian}
	\end{subfigure}%
	~
  	\begin{subfigure}[b]{0.4\textwidth}
    		\includegraphics[width=\textwidth]{%
      		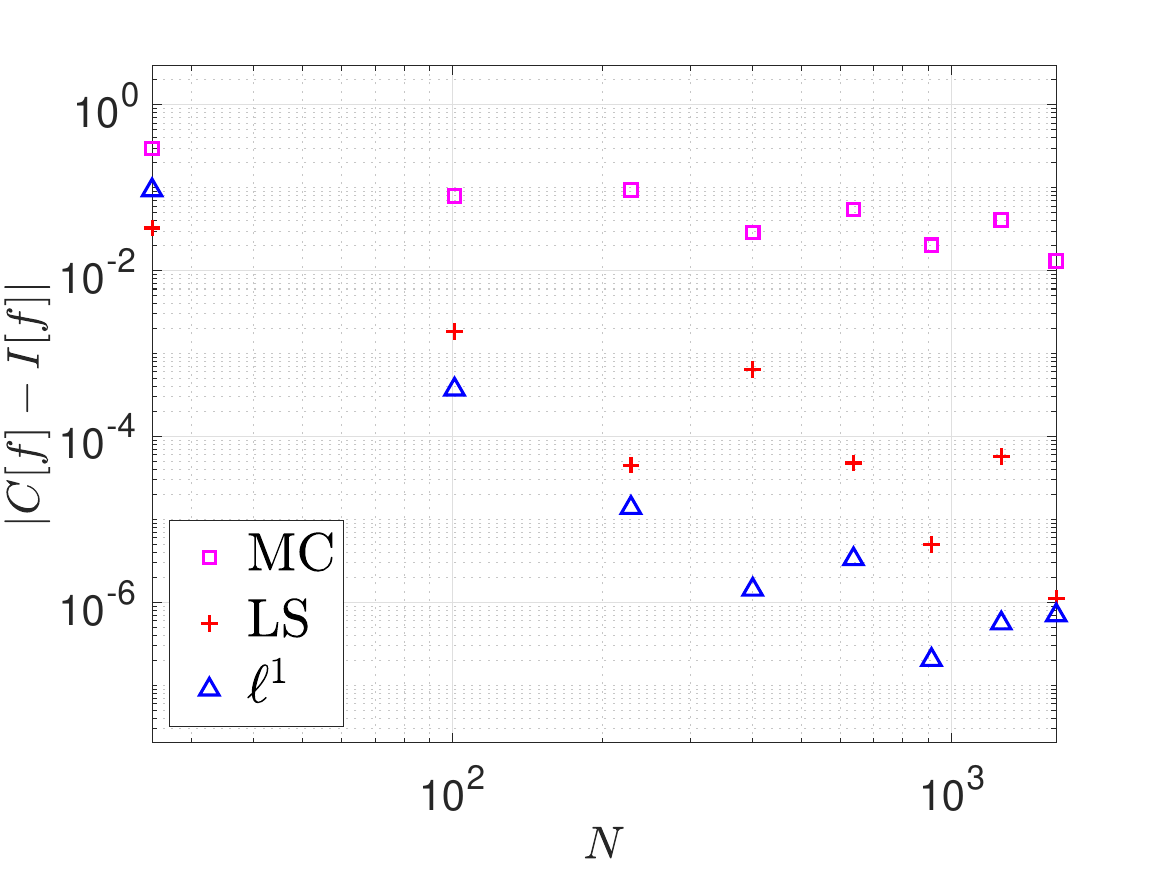} 
    		\caption{Equidistant points}
    		\label{ig:nonstandard_domian_equid}
 	\end{subfigure}%
  	\\
  	\begin{subfigure}[b]{0.4\textwidth}
    		\includegraphics[width=\textwidth]{%
      		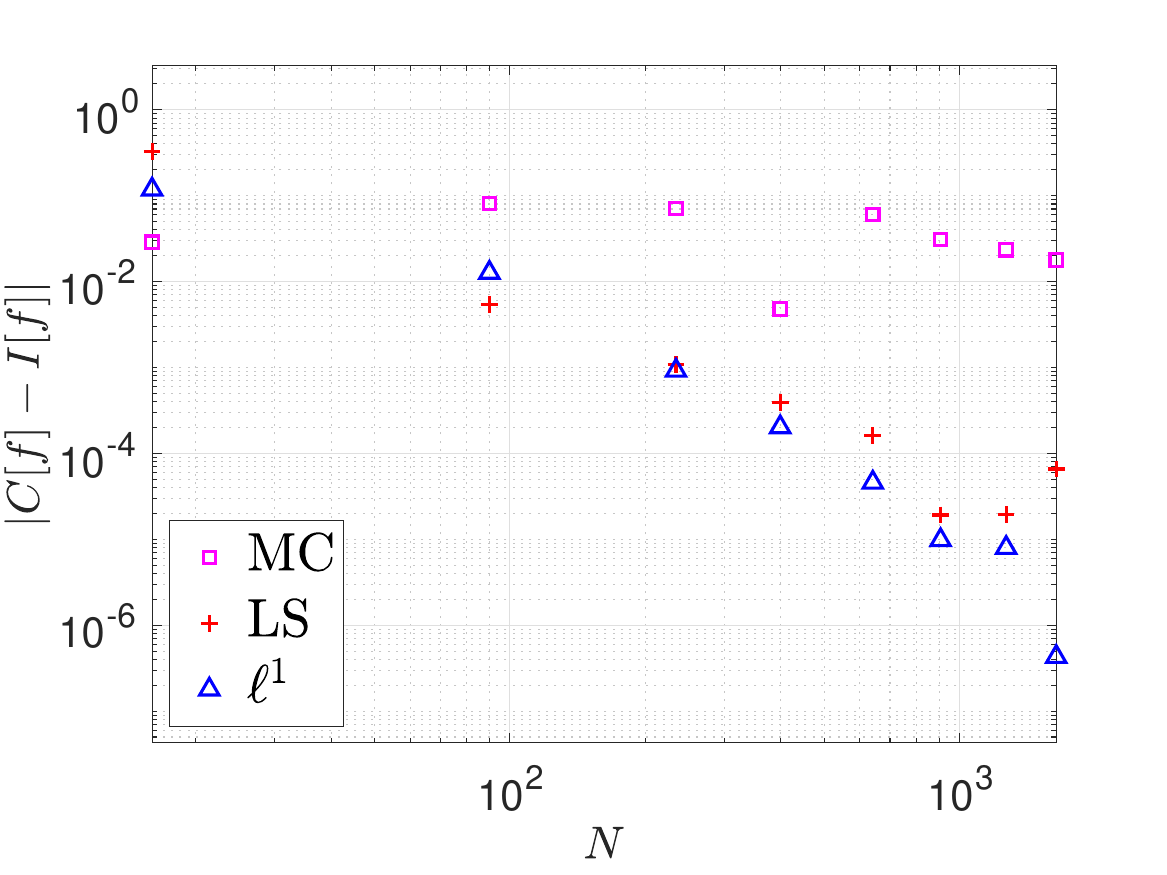} 
    		\caption{Random points}
    		\label{fig:nonstandard_domian_random}
  	\end{subfigure}%
  	~
  	\begin{subfigure}[b]{0.4\textwidth}
    		\includegraphics[width=\textwidth]{%
      		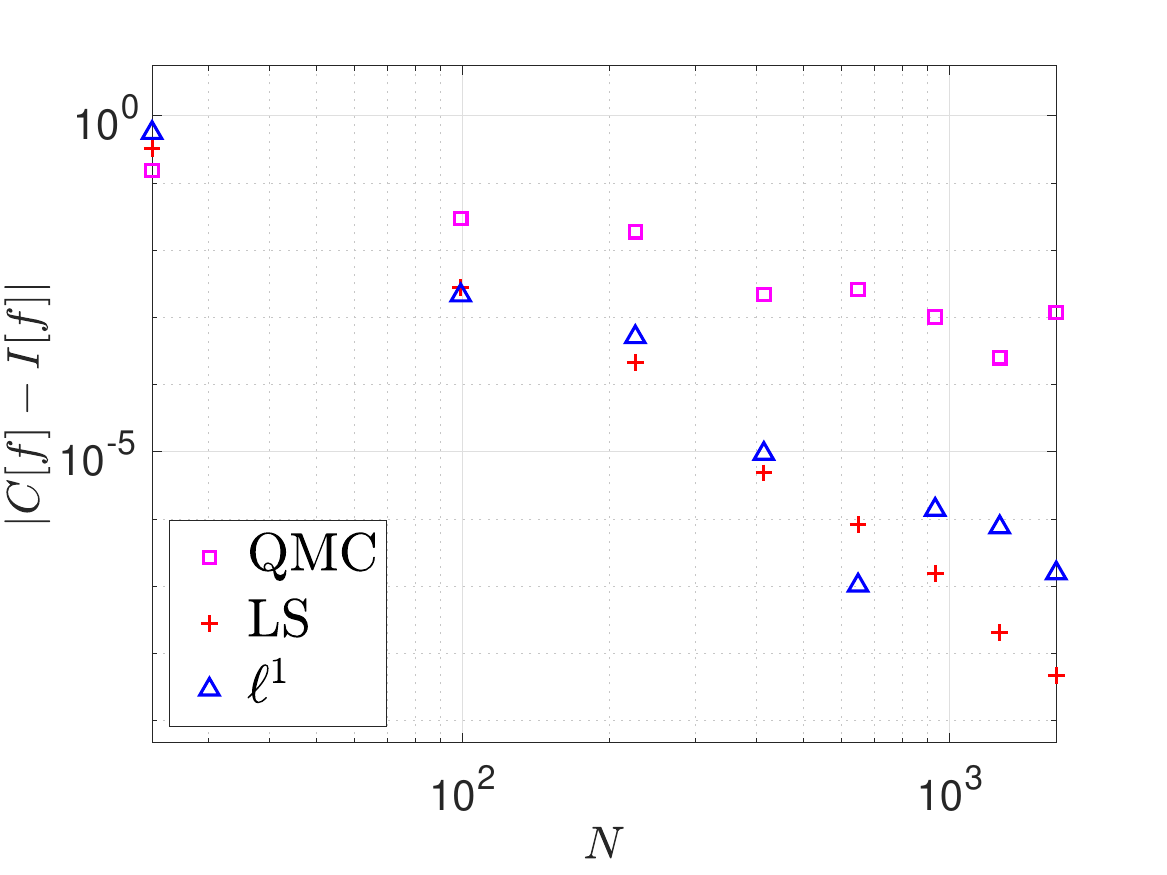} 
    		\caption{Halton points}
    		\label{ig:nonstandard_domian_Halton}
  	\end{subfigure}%
  	\caption{A two dimensional nonstandard domain and errors for $\omega \equiv 1$ and $f(x_1,x_2) = \exp( -x_1^2 - x_2^2 )$}
	\label{fig:nonstandard}
\end{figure}

Here, we equip $\Omega$ with $\omega \equiv 1$ and consider the test function $f(x_1,x_2) = \exp( -x_1^2 - x_2^2 )$. 
The monomial's moments can be computed exactly based on the formulas presented in \ref{sec:moments}. 
The results for the $\ell^1$- and LS-CF as well as the (quasi-)MC are reported in Figure \ref{fig:nonstandard}. 
Again, equidistant, random, and Halton points are considered. 
These were first generated in the larger cube $[-2,2]^2$ and the data points are given by the subset of points falling into the domain $\Omega$. 
We can observe that the $\ell^1$- and LS-CF can provide accurate results also for this nonstandard domain.
Finally, it should be stressed that for the results presented here \emph{no} domain decomposition was performed. 
That is, rather than constructing and adding up separate formulas for the two subdomains, the domain $\Omega$ was considered as a whole when computing the weights for the $\ell^1$- and LS-CFs. 
While it might be computationally more efficient to consider compound CFs (since these can be computed in parallel), this might reduce the DoE and therefore accuracy.  
\section{Summary} 
\label{sec:summary} 

In this work, CFs for experimental data (not fitting a known CF) were considered. 
In particular, we developed novel $\ell^1$- and LS-CFs.
Both of these are---by construction---ensured to be stable (in the sense of nonnegative only cubature weights) while (potentially) being able to achieve high DoE. 
The idea behind these is to allow the number of data points $N$ to be larger than the number of basis functions $K$ for which the desired CF is exact. 
This yielded the linear system corresponding to the exactness conditions to become underdetermined. 
Hence, an $(N-K)$-dimensional affine linear space of solutions $W$ was induced. 
Then, from this space, cubature weights were selected that minimize certain norms corresponding to stability of the CF. 
Here, we investigated two options: 
(1) Minimization w.\,r.\,t.\ the $1$-norm, yielding so-called $\ell^1$-CFs.  
(2) Minimization w.\,r.\,t.\ a weighted $2$-norm, resulting in so-called LS-CFs. 
These CFs were developed for a predefined set of data points. 
Only half of the degrees of freedom could therefore be used for optimization of these CFs compared to many other CFs. 
We still observed the LS- and $\ell^1$-CFs to yield accurate numerical results in a variety of different test cases. 

Future work will include the extension of the stable high-order CFs developed here to non-polynomial function spaces. 
That is, instead of requiring them to be exact for polynomials up to a certain degree, these should be exact for a given finite-dimensional function space (not necessarily consisting of polynomials). 
A first step in this direction has recently been provided in \cite{glaubitz2021construction}. 
However, in this work, it was still assumed that the function space at least included constants. 
Unfortunately, this is not always the case and, for instance, radial basis function spaces might not include constants. 
That said, in a forthcoming work \cite{glaubitz2021towards} we were still able to adapt some of the approaches to prove stability in the context of radial basis function based CFs. 

\appendix 
\section{Moments of the Monomials}
\label{sec:moments}

For the cube, the moments of the one-dimensional monomials, $I[x^k]$, are easy to compute for all cases: 
\begin{equation} 
\begin{aligned} 
	C_1, \omega \equiv 1: \quad  
		& I[ x^k ] = \tilde{m}_k :=
		\begin{cases}
			\ \ 0 & \text{if $k$ is odd}, \\ 
			\frac{2}{k+1} & \text{otherwise},
		\end{cases} \\ 
	C_1, \omega(x) = \sqrt{1 - x^2}: \quad 
		& I[ x^k] = \hat{m}_k :=
		\begin{cases}
			\qquad \quad 0 & \text{if $k$ is odd}, \\ 
			\qquad \quad \frac{\pi}{2} & \text{if $k=0$}, \\
			\frac{(k-1)}{(k+2)} \hat{m}_{k-2} & \text{otherwise}.
		\end{cases} 
\end{aligned} 
\end{equation} 
The moments of the higher-dimensional monomials, $I[\boldsymbol{x}^\mathbf{k}]$ with multi-index $\mathbf{k} = (k_1,\dots,k_q)$, are respectively given by 
\begin{equation} 
	I[\boldsymbol{x}^\mathbf{k}] = \tilde{m}_{k_1} \dots \tilde{m}_{k_q}, \quad 
	I[\boldsymbol{x}^\mathbf{k}] = \hat{m}_{k_1} \dots \hat{m}_{k_q}.
\end{equation}
For the ball, on the other hand, these are given by 
\begin{equation} 
\begin{aligned} 
	B_q, \omega \equiv 1: \quad  
		& I[\boldsymbol{x}^\mathbf{k}] = 
		\frac{2}{k_1 + \dots + k_q + q}
		\begin{cases}
			\qquad 0 & \text{if some $k_i$ is odd}, \\ 
			\frac{\Gamma(\beta_1) \dots \Gamma(\beta_q)}{\Gamma(\beta_1 + \dots + \beta_q)} & \text{otherwise},
		\end{cases} \\ 
	B_q, \omega(\boldsymbol{x}) = \sqrt{ \|\boldsymbol{x}\|_2}: \quad 
		& I[\boldsymbol{x}^\mathbf{k}] = 
		\frac{2}{k_1 + \dots + k_q + q + \frac{1}{2}}
		\begin{cases}
			\qquad 0 & \text{if some $k_i$ is odd}, \\ 
			\frac{\Gamma(\beta_1) \dots \Gamma(\beta_q)}{\Gamma(\beta_1 + \dots + \beta_q)} & \text{otherwise},
		\end{cases} 
\end{aligned} 
\end{equation} 
where $\beta_i = \frac{1}{2}(k_i+1)$; see \cite{folland2001integrate}. 

\section*{Acknowledgements}
The author would like to thank Alina Glaubitz, Dorian Hillebrand, and Simon-Christian Klein as well as the anonymous referees for helpful advice. 

\bibliographystyle{siamplain}
\bibliography{literature}

\end{document}